\documentclass[11pt,reqno]{amsart}
\usepackage{amsthm,amsmath,amssymb,url,cite,color}
\newtheorem{thm}{Theorem}
\newtheorem{lem}{Lemma}
\newtheorem{prop}[thm]{Proposition}

\newtheorem{defi}[thm]{Definition}

\newtheorem{remark}{Remark}

\def\exc{\mathop{wex}}
\def\wex{\mathop{wex}}

\def\cro{\mathop {cr}}

\def\D{{\mathcal D}}
\def\S{{\mathcal S}}

\def\blue{\textcolor{blue}}
\def\red{\textcolor{red}}

\begin{document}
\title{
The Combinatorics of Al-Salam-Chihara\\ $q$-Laguerre polynomials}
%~\thanks{\today}
\author{Anisse Kasraoui}
\address{Universit\'e de Lyon,  Universit\'e Lyon 1,
UMR 5208 du CNRS, Institut Camille Jordan\\
F-69622, Villeurbanne Cedex, France}
\email{anisse@math.univ-lyon1.fr}

\author{Dennis Stanton}
\address{School of Mathematics, University of Minnesota, Minneapolis\\
\small Minnesota 55455, USA} \email{stanton@math.umn.edu}

\author{Jiang Zeng}
\address{Universit\'e de Lyon, Universit\'e Lyon 1,
UMR 5208 du CNRS, Institut Camille Jordan\\
 F-69622, Villeurbanne Cedex, France}
\email{zeng@math.univ-lyon1.fr}

\begin{abstract}
We describe various aspects of the Al-Salam-Chihara $q$-Laguerre
polynomials. These include combinatorial descriptions of the
polynomials, the moments, the orthogonality relation and a
combinatorial interpretation of the linearization coefficients. It
is remarkable that the corresponding moment sequence appears also in
the recent work of Postnikov and Williams on enumeration of totally
positive Grassmann cells.
\end{abstract}

\maketitle {\small \tableofcontents }
 \vskip 0.5cm \noindent{\bf MR Subject
Classifications}: Primary 05A18; Secondary 05A15, 05A30.
%%%%%%%%%%%%%%%%%%%
\section{Introduction}
The monic simple Laguerre polynomials $L_n(x)$ may be  defined by
the explicit formula:
\begin{equation}\label{eq:explag}
L_n(x)=\sum_{k=0}^{n} (-1)^{n-k}\frac{n!}{k!}\,{n\choose  k} x^{k},
\end{equation}
or by the three-term recurrence relation
\begin{align}\label{eq:recurr}
L_{n+1}(x)=(x-(2n+1))L_n(x)-n^2L_{n-1}(x).
\end{align}
The moments are
\begin{equation}\label{eq:momentsimp}
\mu_n={\mathcal L}(x^n)=\int_0^\infty x^ne^{-x}dx=n!.
\end{equation}
The linearization formula reads as follows:
$$
L_{n_1}(x)L_{n_2}(x)=\sum_{n_3}C_{n_1\,n_2}^{n_3}L_{n_3}(x),
$$
where
$$
C_{n_1\,n_2}^{n_3}=\sum_{s\geq
0}\frac{n_1!\,n_2!\,2^{N_2+n_3-2s}\,s!}
{(s-n_1)!(s-n_2)!(s-n_3)!(N_2+n_3-2s)!n_3!}.
$$
Equivalently we have
\begin{equation}\label{eq:lin}
{{\mathcal L}}(L_{n_1}(x)L_{n_2}(x)L_{n_3}(x))= \sum_{s\geq
0}\frac{n_1!\,n_2!\,n_3!\,2^{N_2+n_3-2s}\,s!}
{(s-n_1)!(s-n_2)!(s-n_3)!(N_2+n_3-2s)!}.
\end{equation}

Given positive integers $n_1,n_2,\ldots, n_k$ such that
$n=n_1+\cdots +n_k$, let $S_i$ be the consecutive integer segment
$\{n_1+\cdots n_{i-1}+1, \ldots, n_1+\cdots +n_i\}$ with $n_0=0$,
then $S_1\cup \ldots \cup S_k=[n]$. A permutation $\sigma$ of $[n]$
is said to be a \emph{generalized derangement} of specification $(n_1,\ldots,n_k)$ if $i$ and
$\sigma(i)$ do not belong to a same segment $S_j$ for all $i\in
[n]$. Let $\D(n_1,n_2,\ldots,n_k)$ be the set of generalized
derangements of specification $(n_1,\ldots,n_k)$ then we have
\begin{equation}\label{eq:ordlaguerre}
{\mathcal L}(L_{n_1}(x)\ldots L_{n_k}(x))= \sum_{\sigma\in
\D(n_1,n_2,\ldots,n_k)}1.
\end{equation}

A $q$-version of~\eqref{eq:explag} was studied by Garsia and Remmel
\cite{GR} in 1980. Several $q$-analogues of the recurrence
relation~\eqref{eq:recurr}
 and moments~\eqref{eq:momentsimp}
 were investigated in the last two decades (see \cite{SS1, SS2, CSZ}) in order to obtain new \emph{mahonian} statistics
 on the symmetric groups. On the other hand,
 in view of the unified combinatorial interpretations of several aspects of Sheffer
 orthogonal polynomials (moments, polynomials, and the linearization
 coefficients)(see \cite{Vi, Ze1,KZ}) it is natural to seek for a
 $q$-version of this picture.

 As one can expect,   the first  result in this direction was  the linearization formula for $q$-Hermite polynomials
  due to  Ismail, Stanton and Viennot \cite{ISV}, dated back to 1987.  In particular, their formula provides  a combinatorial evaluation of the
 Askey-Wilson integral.  However,
a similar  formula for
 $q$-Charlier polynomials was discovered only recently  by  Anshelevich~\cite{An},  who used  the machinery of
 $q$-Levy stochastic processes. Short later,  Kim, Stanton and Zeng \cite{KSZ} gave  a  combinatorial proof of Anshelevich's result.

The object of this paper is to give a $q$-version of all the above
formulas for  simple Laguerre polynomials. It is interesting to note
that the corresponding moment sequence appears in the recent work on
enumeration of totally positive Grassmann cells~\cite{Wi,Co}.

The rest of this paper is organized as follows: We recall the
definition of Al-Salam-Chihara polynomials,  prove their
linearization formula introduce the new
$q$-Laguerre polynomials in Section~2. In Section~3  we study the moment sequence of  the $q$-Laguerre polynomials.
In particular we shall give a new proof of Williams' formula for the corresponding moment sequence.
We derive then the linearization coefficients of our $q$-Laguerre polynomials in Section~4.
Finally two  technical lemmas will be proved in Sections~5 and~6, respectively.

%%%%****%%%%****%%%%****%%%%****%%%%****%%%%****%%%%****%%%%*
%%%%
%%%% New Section
%%%%
%%%%****%%%%****%%%%****%%%%****%%%%****%%%%****%%%%****%%%%*

\section{Al-Salam-Chihara polynomials revisited}
The Al-Salam-Chihara polynomials $Q_n(x):=Q_n(x;\alpha,\beta|q)$ may
be defined by the recurrence relation~\cite[Chapter 3]{KK}:
\begin{align}
\begin{cases}
Q_0(x)=1,\quad Q_{-1}(x)=0,\\
Q_{n+1}(x)=(2x-(\alpha+\beta)q^n)Q_n(x)-
(1-q^n)(1-\alpha\beta{q}^{n-1})Q_{n-1}(x),\quad n\geq 0.
\end{cases}
\end{align}
Let $Q_n(x)=2^np_n(x)$ then
\begin{align}\label{eq:normalizedrecurr}
xp_n(x)=p_{n+1}(x)+\frac{1}{2}
(\alpha+\beta)q^np_n(x)+\frac{1}{4}(1-q^n)(1-\alpha\beta
q^{n-1})p_{n-1}(x).
\end{align}
They also have the following explicit expressions:
\begin{align*}
Q_{n}(x;\alpha,\beta|q) &=\frac{(\alpha\beta;\,q)_{n}}{\alpha^n}\,
{}_3\phi_{2}\left( {{q^{-n},\alpha u,\alpha
u^{-1}}\atop{\alpha\beta,0}} \Big| \,q;\,q
\right)\\
&=(\alpha u;\,q)_{n}u^{-n}\, {}_2\phi_{1}\left( {{q^{-n},\beta
u^{-1}}\atop{\alpha^{-1}q^{-n+1}u^{-1}}} \Big| \,q;\alpha^{-1}q u
\right)\\
&=(\beta u^{-1};\,q)_{n}u^{n}\, {}_2\phi_{1}\left( {{q^{-n},\alpha
u}\atop{\beta^{-1}q^{-n+1}u}} \Big| \,q;\beta^{-1}q u^{-1} \right),
\end{align*}
where $x=\frac{u+u^{-1}}2$ or $x=\cos \theta$ if $u=e^{i\theta}$.

The Al-Salam-Chihara polynomials have the following generating
function
$$
G(t,x)=\sum_{n=0}^\infty Q_n(x;\alpha,\beta|q)\frac{t^n}{(q;\,q)_n}=
\frac{(\alpha t, \beta t;\,q)_\infty}{(te^{i\theta},
te^{-i\theta};\,q)_\infty}.
$$
They are orthogonal with respect to the  linear functional
$\hat{\mathcal L}_q$:
\begin{align}
\hat{\mathcal L}_q(x^n)=\frac{1}{2\pi}\int_{0}^\pi(\cos\theta)^n
\frac{(q,\alpha\beta,e^{2i\theta}, e^{-2i\theta};\,q)_\infty}
{(\alpha e^{i\theta}, \alpha e^{-i\theta}, \beta e^{i\theta}, \beta
e^{-i\theta};\,q)_\infty} d\theta,
\end{align}
where $x=\cos \theta$.
 Note that
\begin{align*}
\hat{\mathcal L}_q(Q_n(x)^2)= (q;\,q)_n(\alpha\beta;\,q)_n.
\end{align*}

\begin{thm}\label{thm:alsalamlinear}.
We have
\begin{align}\label{eq:alsalamlinear}
Q_{n_1}(x)Q_{n_2}(x)=\sum_{n_3\geq
0}C_{n_1,n_2}^{n_3}(\alpha,\beta;\,q)Q_{n_3}(x),
\end{align}
where
\begin{align*}
C_{n_1,n_2}^{n_3}(\alpha,\beta;\,q)&=(-1)^{N_2+n_3}
\frac{(q;\,q)_{n_1}(q;\,q)_{n_2}}{(\alpha\beta;\,q)_{n_3}}\\
&\times
\sum_{m_2,m_3}\frac{(\alpha\beta;\,q)_{n_1+m_3}\alpha^{m_2}\beta^{n_3+n_2-n_1-m_2-2m_3}
q^{{m_2\choose 2}+{n_3+n_2-n_1-m_2-2m_3\choose
2}}}{(q;\,q)_{n_3+n_2-n_1-m_2-2m_3}
(q;\,q)_{m_2}(q;\,q)_{m_3+n_1-n_3}(q;\,q)_{m_3+n_1-n_2}(q;\,q)_{m_3}}.
\end{align*}

\end{thm}
\begin{proof}
Clearly $C_{n_1,n_2}^{n_3}(\alpha,\beta;\,q) =\hat{\mathcal
L}_q(Q_{n_1}(x)Q_{n_2}(x)Q_{n_3}(x))/\hat{\mathcal
L}_q(Q_{n_3}(x)Q_{n_3}(x))$. Using the Askey-Wilson integral:
$$
\frac{(q;\,q)_\infty}{2\pi}\int_0^\pi\frac{(e^{2i\theta},
e^{-2i\theta};\,q)_\infty} {\prod_{j=1}^4(t_je^{i\theta},
t_je^{-i\theta};\,q)_\infty}d\theta=
\frac{(t_1t_2t_3t_4;\,q)_\infty}{\prod_{1\leq j<k\leq
4}(t_jt_k;\,q)_\infty},
$$
one can prove \cite[Theorem 3.5]{ISV} that
\begin{align*}
&\hat{\mathcal L}_q( G(t_1,x)G(t_2,x)G(t_3,x))\\
&\hspace{2cm}= \frac{(\alpha t_1t_2t_3, \beta q t_1t_2t_3,
{\alpha\beta}q;\,q)_\infty}{(t_1t_2,t_1t_3,t_2t_3;\,q)_\infty}
{}_3\phi_2\left(\begin{array}{ccc}
t_1t_2,&t_1t_3,&t_2t_3\\
&\alpha t_1t_2t_3,&\beta t_1t_2t_3
\end{array}|q;\alpha\beta\right).
\end{align*}
Therefore
\begin{align}
&\sum_{n_1,n_2,n_3}\hat{\mathcal
L}_q(Q_{n_1}(x)Q_{n_2}(x)Q_{n_3}(x))\frac{t_1^{n_1}}{(q;\,q)_{n_1}}
\frac{t_2^{n_2}}{(q;\,q)_{n_2}}\frac{t_3^{n_3}}{(q;\,q)_{n_3}}\nonumber\\
&\hspace{2cm}=\sum_{k\geq 0}\frac{(\alpha t_1t_2t_3q^k,\, \beta
t_1t_2t_3q^k, \,\alpha\beta; q)_\infty} {(t_1t_2q^k,\, t_1t_3q^k,\,
t_2t_3q^k; q)_\infty}\frac{(\alpha\beta)^k}{(q;\,q)_k}.\label{eq:1}
\end{align}
Using the Euler formulas:
\begin{align*}
(t;\,q)_\infty=\sum_{n\geq 0}\frac{(-1)^nq^{n\choose
2}}{(q;\,q)_n}t^n;\qquad \frac{1}{(t;\,q)_\infty}=\sum_{n\geq
0}\frac{1}{(q;\,q)_n}t^n,
\end{align*}
we can rewrite the sum in \eqref{eq:1} as follows:
\begin{align}
&(\alpha\beta;\,q)_\infty\sum_{k\geq
0}\frac{(\alpha\beta)^k}{(q;\,q)_k}\sum_{l_1,l_2\geq
0}\frac{\alpha^{l_1}\beta^{l_2}
q^{k(l_1+l_2)}(-t_1t_2t_3)^{l_1+l_2}q^{{l_1\choose 2}+{l_2\choose 2}}}{(q;\,q)_{l_1}(q;\,q)_{l_2}}\nonumber\\
&\hspace{2cm}\times \sum_{m_1, m_2,m_3\geq 0}
\frac{q^{(m_1+m_2+m_3)k}t_1^{m_1+m_2}t_2^{m_1+m_3}t_3^{m_1+m_3}}
{(q;\,q)_{m_1}(q;\,q)_{m_2}(q;\,q)_{m_3}}.\label{eq:step1}
\end{align}
Substituting
$$
\sum_{k\geq
0}\frac{({\alpha\beta}q^{l_1+l_2+m_1+m_2+m_3})^k}{(q;\,q)_k}
=\frac{1}{({\alpha\beta}q^{l_1+l_2+m_1+m_2+m_3}; q)_\infty}
$$
 in \eqref{eq:step1},
we get
\begin{align}\label{eq:2}
\sum_{l_1,l_2,m_1,m_2,m_3}t_1^{n_1}t_2^{n_2}t_3^{n_3}
\frac{(\alpha\beta)_{n_1+m_3}\alpha^{l_1}\beta^{l_2}q^{{l_1\choose
2}+{l_2\choose 2}}}
{(q;\,q)_{m_1}(q;\,q)_{m_2}(q;\,q)_{m_3}(q;\,q)_{l_1}(q;\,q)_{l_2}}(-1)^{l_1+l_2},
\end{align}
where $l_1+l_2+m_1+m_2=n_1$, $l_1+l_2+m_1+m_3=n_2$ and
 $l_1+l_2+m_2+m_3=n_3$.

 Since  $l_1+l_2\equiv N_2+n_3\pmod{2}$,
extracting the coefficient of
$\frac{t_1^{n_1}t_2^{n_2}t_3^{n_3}}{(q;\,q)_{n_1}(q;\,q)_{n_2}(q;\,q)_{n_3}}$
in \eqref{eq:2} and dividing by $(q,\alpha\beta;\,q)_{n_3}$
 we obtain \eqref{eq:alsalamlinear} where $l_1$ is replaced by $m_2$.
\end{proof}

%%%%%%%%%%%%%%%%%%%%%%%%%%%%%%%%%%%%%%%%%%%%%%%%
%\section{The new $q$-Laguerre polynomials}
%%%%%%%%%%%%%%%%%%%%%%%%%%%%%%%%%%%%%%%%%%%%%%%%%%%%
We define the new $q$-Laguerre polynomials $L_n(x;\,q)$ by
re-scaling Al-Salam-Chihara polynomials:
\begin{equation}\label{eq:polydef}
L_n(x;\,q)=\left(\frac{\sqrt{y}}{q-1}\right)^n
Q_n\left(\frac{(q-1)x+y+1}{2\sqrt{y}}; \frac{1}{\sqrt{y}},
\sqrt{y}q|q\right).
\end{equation}
It follows from \eqref{eq:normalizedrecurr} that the polynomials
$L_n(x;\,q)$ satisfy the recurrence:
\begin{align}\label{eq:recurrqlaguerre}
L_{n+1}(x;\,q)=(x-y[n+1]_q-[n]_q)L_n(x;\,q)-y[n]_q^2L_{n-1}(x;\,q).
\end{align}

We derive then the explicit formula for $L_{n}(x)$:
\begin{equation}\label{eq:explaguerre}
L_n(x;\,q)=\sum_{k=0}^n (-1)^{n-k}\frac{n!_q}{k!_q}\,{n\brack k}_q
q^{k(k-n)}y^{n-k} \prod_{j=0}^{k-1}\left(x-(1-y q^{-j})[j]_q\right).
\end{equation}
Thus
\begin{align*}
L_1(x;\,q)&=x-y,\\
L_2(x;\,q)&=x^2-(1+2y+qy)x+(1+q)y^2,\\
L_3(x;\,q)&={x}^{3}-({q}^{2}y+3\,y+q+2+2\,qy)x^2\\
&\hspace{1cm}+({q}^{3}{y}^{2}+y{q}^{2}+q+2\,qy+3\,{q}^{2}{y}^{2}+
1+4\,q{y}^{2}+2
\,y+3\,{y}^{2})x\\
&\hspace{2cm}-(2\,{q}^{2}+2\,q+{q}^{3}+1){y}^{3}.
\end{align*}

A combinatorial interpretation of these $q$-Laguerres polynomials
can be derived from the Simion and Stanton's combinatorial model for
the $a=s=u=1$ and  $r=t=q$ special case of the quadrabasic Laguerre
polynomials \cite[p.313]{SS2}.

%%%%****%%%%****%%%%****%%%%****%%%%****%%%%****%%%%****%%%%*
%%%%
%%%% New Section
%%%%
%%%%****%%%%****%%%%****%%%%****%%%%****%%%%****%%%%****%%%%*
\section{Moments of the $q$-Laguerre polynomials}
Let $\S_n$ be the set of permutations of $[n]:=\{1,2, \ldots, n\}$.
For $\sigma\in \S_n$  the \emph{number of crossings} of $\sigma$ is
defined by
$$
\cro(\sigma)=\sum_{i=1}^n\#\{j|j<i\leq \sigma(j)<\sigma(i)\}+
\sum_{i=1}^n\#\{j|j>i>\sigma(j)>\sigma(i)\},
$$
while the number of \emph{weak excedances} of $\sigma$ is defined by
\begin{align*}
\wex(\sigma)&=\#\{i|1\leq i\leq n\;\textrm{and}\; i\leq \sigma(i)\}.
\end{align*}

It is useful to  have a geometric interpretation of  these statistics by associating with each permutation
$\sigma$ of $[n]$ a  diagram as follows:
arrange the integers $1, 2, \ldots , n$
on a line in increasing order from left to right and draw an arc $i\to \sigma(i)$
above (resp. under) the line if $i<
\sigma(i)$ (resp. $i> \sigma(i)$). For example, the permutation
$\sigma=9\,3\,7\,4\,6\,11\,5\,8\,1\,10\,2$ can be depicted as
follows:
$$
{\setlength{\unitlength}{1mm}
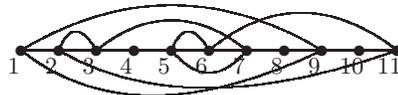
\begin{picture}(55,20)(0,-10)
\qbezier(0,0)(20,12)(40,0)\qbezier(5,0)(7,5)(10,0)\qbezier(20,0)(22,5)(25,0)\qbezier(10,0)(20,8)(30,0)\qbezier(25,0)(37,10)(50,0)
\qbezier(5,0)(22,-10)(50,0)\qbezier(20,0)(25,-6)(30,0)\qbezier(0,0)(15,-12)(40,0)
\put(0,0){\line(1,0){50}}
\put(0,0){\circle*{1,3}}\put(0,0){\makebox(-2,-4)[c]{\footnotesize 1}}
\put(5,0){\circle*{1,3}}\put(5,0){\makebox(-2,-4)[c]{\footnotesize2}}
\put(10,0){\circle*{1,3}}\put(10,0){\makebox(-2,-4)[c]{\footnotesize3}}
\put(15,0){\circle*{1,3}}\put(15,0){\makebox(-2,-4)[c]{\footnotesize4}}
\put(20,0){\circle*{1,3}}\put(20,0){\makebox(-2,-4)[c]{\footnotesize5}}
\put(25,0){\circle*{1,3}}\put(25,0){\makebox(-2,-4)[c]{\footnotesize6}}
\put(30,0){\circle*{1,3}}\put(30,0){\makebox(-2,-4)[c]{\footnotesize7}}
\put(35,0){\circle*{1,3}}\put(35,0){\makebox(-2,-4)[c]{\footnotesize8}}
\put(40,0){\circle*{1,3}}\put(40,0){\makebox(-2,-4)[c]{\footnotesize9}}
\put(45,0){\circle*{1,3}}\put(45,0){\makebox(-2,-4)[c]{\footnotesize10}}
\put(50,0){\circle*{1,3}}\put(50,0){\makebox(-2,-4)[c]{\footnotesize11}}
\end{picture}}
$$

Thus, the number of weak excedances of $\sigma$ is the number of edges drawn above the line  plus the number of isolated points,
while the  number  of crossings of $\sigma$ is  the number
of pairs of edges above the line that cross or touch $\left({\setlength{\unitlength}{0.6mm}
\begin{picture}(33,5)(0,0)
\put(2,0){\line(1,0){31}}
\put(5,0){\circle*{1,3}}
\put(25,0){\circle*{1,3}}
\put(10,0){\circle*{1,3}}
\put(30,0){\circle*{1,3}}
\qbezier(5,0)(15,10)(25,0) \qbezier(10,0)(20,10)(30,0)
\end{picture}} \text{ or}   {\setlength{\unitlength}{0.6mm}
\begin{picture}(33,5)(0,0)
\put(2,0){\line(1,0){29}}
\put(5,0){\circle*{1,3}}
\put(17,0){\circle*{1,3}}
\put(29,0){\circle*{1,3}}
\qbezier(5,0)(11,10)(17,0) \qbezier(17,0)(23,10)(29,0)
\end{picture}} \right)$,
plus the  number
of pairs of edges under the line that cross
$\left(
{\setlength{\unitlength}{0.6mm}
\begin{picture}(35,5)(0,-4)
\put(2,0){\line(1,0){31}}
\put(5,0){\circle*{1,3}}
\put(25,0){\circle*{1,3}}
\put(10,0){\circle*{1,3}}
\put(30,0){\circle*{1,3}}
\qbezier(5,0)(15,-10)(25,0)
 \qbezier(10,0)(20,-10)(30,0)
\end{picture}}\right)$.

Let  $\mu_n^{(\ell)}(y,q)$ be the enumerating polynomial of
permutations in $\S_n$  with respect to numbers of weak excedances and crossings:
\begin{align}\label{defmoment}
\mu_n^{(\ell)}(y,q):=\sum_{\sigma\in
S_n}y^{\exc(\sigma)}q^{\cro(\sigma)}.
\end{align}
It has been proved  in  \cite{SS2,RA,Co} that the generating
function of the moment sequence has
 the  following continued fraction
expansion:
\begin{equation}\label{eq:co}
E(y, q, t):=\sum_{n\geq
0}\mu_n^{(\ell)}(y,q)t^n=\frac{1}{1-b_0t-\displaystyle
\frac{\lambda_1 t^2}{1-b_1t-\displaystyle\frac{\lambda_2
t^2}{\ddots}}},
\end{equation}
where $b_n=y[n+1]_q+[n]_q$ and $\lambda_n=y[n]_q^2$.

We derive then from the classical theory of orthogonal polynomials
the following interpretation for  the moments of the $q$-Laguerre
polynomials.
\begin{thm} The $n$-th moment of the $q$-Laguerre polynomials  is equal to $\mu_n^{(\ell)}(y,q)$. More precisely,
let ${{\mathcal L}}_q$ be the linear functional defined by
${{\mathcal L}}_q(x^n)=\mu_n^{(\ell)}(y,q)$, then
\begin{align}\label{eq:orthog}
{\mathcal
L}_q(L_{n_1}(x;\,q)L_{n_2}(x;\,q))=y^{n_1}(n_1!_q)^2\delta_{n_1\,n_2}.
\end{align}
\end{thm}

The first values of the moment sequence  are as follows:
\begin{align*}
\mu_1^{(\ell)}(y,q)&=y,\\
\mu_2^{(\ell)}(y,q)&=y+y^2,\\
\mu_3^{(\ell)}(y,q)&=y+(3+q)y^2+y^3,\\
\mu_4^{(\ell)}(y,q)&=y+(6+4q+q^2)y^2+(6+4q+q^2)y^3+y^4.
\end{align*}

Combining  the results of Corteel~\cite{Co},
Williams~\cite[Proposition 4.11]{Wi} and the classical theory of
orthogonal polynomials, one can write the moments of the above
$q$-Laguerre polynomials as a finite double sum (cf.
\eqref{eq:momentw}).  Here we propose  a direct  proof of this
result. Actually  we shall give such a formula  for the moments of
Al-Salam-Chihara polynomials.

\begin{defi}
Define the $y$-versions of the $q$-Stirling numbers of the second
kind by
\begin{align}\label{eq:y-stirling}
X^n=\sum_{k=1}^nS_q(n,k,y)\prod_{j=0}^{k-1}(X-[j]_q(1-yq^{-j})).
\end{align}
The $y$-versions of  $q$-Stirling numbers of the first kind can be
defined by the inverse matrix or equivalently
$$
\prod_{j=0}^{n-1}(X-[j]_q(1-yq^{-j}))=\sum_{k=1}^{n}s_q(n,k,y)X^k.
$$
\end{defi}
\begin{remark}  We have
\begin{align*}
S_q(n,k,y)|_{q=1}=S(n,k) (1-y)^{n-k},\quad  S_q(n,k,0)=S_q(n,k),
\end{align*}
where $S(n,k)$ and $S_q(n,k)$ are,  respectively, the Stirling
numbers of the second kind and their
 well-known $q$-analogues,  see \cite{GO}.
\end{remark}

Consider the rescaled Al-Salam-Chihara polynomials $P_n(x)$:
\begin{align}\label{rescal}
P_n(X)&=Q_n(((q-1)X+1/\alpha^2+1)\alpha/2;\,\alpha,\beta|q)\nonumber\\
&=\alpha^{-n}\sum_{k=0}^n\frac{(q^{-n};\,q)_k}{(q;\,q)_k}q^k(\alpha\beta
q^k;\,q)_{n-k}(1-q)^k q^{k\choose 2}\alpha^{2k}\nonumber\\
&\qquad \qquad\times
\prod_{j=0}^{k-1}\left(X-[i]_q(1-q^{-i}/\alpha^2)\right).
\end{align}
\begin{lem}\label{th:ACmomformula}
The moments of the rescaled Al-Salam-Chihara polynomials $P_n(X)$
are
\begin{align}\label{eq:explicit1}
\mu_n(\alpha,\beta)=\sum_{k=1}^n S_q(n,k,1/\alpha^2)
(\alpha\beta;\,q)_k q^{-{k\choose 2}} (1-q)^{-k} \alpha^{-2k}.
\end{align}
\end{lem}
\begin{proof}
Let $L: X^n\mapsto \mu_n(\alpha,\beta)$ be the linear functional.
 We check that these moments do satisfy $L(P_n(X))=0$ for $n>0$. Let $a_k$ be the coefficients in front of the product in
\eqref{rescal}, then we have, using $y$-Stirling orthogonality,
\begin{align*}
 L(P_n(X))&=\sum_{k=0}^n a_k\sum_{j=1}^k s_q(k,j,1/\alpha^2)\sum_{t=1}^jS_q(j,t,1/\alpha^2)
(\alpha\beta;\,q)_t q^{-{t\choose 2}} (1-q)^{-t}\alpha^{-2t}\\
&=\sum_{k=0}^n a_k(\alpha\beta;\,q)_k q^{-{k\choose 2}} (1-q)^{-k}\alpha^{-2k}\\
&=\alpha^{-n}(\alpha\beta;\,q)_n\sum_{k=0}^n\frac{(q^{-n};\,
q)_k}{(q;\,q)_k}q^k=0.
\end{align*}
Note that the last equality follows by applying the $q$-binomial
formula.
\end{proof}

\begin{lem}\label{lem2} Let $p=1/q$. We have
\begin{align}
\sum_{k=0}^\infty \frac{(\alpha\beta;\,q)_k q^{-{k\choose
2}}(1-q)^{-k}\alpha^{-2k}t^k}{\prod_{i=1}^k(1-[i]_q
t(1-q^{-i}/\alpha^2))}= \sum_{i\geq
0}\frac{c_i(\alpha,\beta)}{1-[i]_q
t(1-q^{-i}/\alpha^2)},\label{eq:partialfraction}
\end{align}
where
$$
c_i(\alpha,\beta)=\frac{(\alpha\beta ;q)_i}{(q;q)_i}
\frac{q^{i-i^2}\alpha^{-2i}}{(q^{1-2i}/\alpha^2;q)_i}
\frac{(p^{1+i}\alpha\beta/\alpha^2;p)_\infty}{(p^{1+2i}/\alpha^2;p)_\infty}.
$$
\end{lem}
\begin{proof}
Note the following partial fraction decomposition formula:
$$
\frac{t^k}{(1-a_1t)(1-a_2t)\ldots (1-a_kt)}=\frac{(-1)^k}{a_1\cdots
a_k}+\sum_{i=1}^k\frac{a_i^{-1}\prod_{j=1, j\neq
i}^{k}(a_i-a_j)^{-1}}{1-a_it}.
$$
Therefore
\begin{align}\label{eq:keydecomp}
\frac{t^k}{\prod_{i=1}^k(1-[i]_qt(1-q^{-i}/\alpha^2))}
=\sum_{i=0}^k\frac{\gamma_k(i)}{1-[i]_qt(1-q^{-i}/\alpha^2)},
\end{align}
where
\begin{align*}
\gamma_k(i)=\frac{1}{k!_q}{k\brack i}_q
\frac{\alpha^{2(k-i)}q^{{k\choose
2}+k-i^2}}{(q^{1-2i}/\alpha^2;q)_i(q^{1+2i}\alpha^2;q)_{k-i}}
 \qquad (0\leq i\leq k).
\end{align*}
Substituting this in \eqref{eq:partialfraction} yields
\begin{align*}
c_i(\alpha,\beta)&=\sum_{k\geq i}
\frac{(\alpha\beta;q)_k}{(q;q)_k}{k\brack i}_q
\frac{q^{k-i^2}\alpha^{-2i}}{(q^{1-2i}/\alpha^2;q)_i(q^{1+2i}\alpha^2;q)_{k-i}}\\
&=\frac{(\alpha\beta ;q)_i}{(q;q)_i}
\frac{q^{i-i^2}\alpha^{-2i}}{(q^{1-2i}/\alpha^2;q)_i} \sum_{k\geq 0}
\frac{(\alpha\beta q^i;q)_k}{(q;q)_k}
\frac{q^k}{(q^{1+2i}\alpha^2;q)_{k}}.
\end{align*}
The result  follows then by applying the ${}_1\Phi_1$ summation
formula (see \cite[II.5]{Gas-Rah}).
\end{proof}

\begin{thm} The moments $\mu_n(\alpha, \beta)$ have the explicit formula
\begin{align}
\mu_n(\alpha,\beta)= \sum_{k=1}^n\sum_{i=1}^k
 {k\brack i}_q     \frac{q^{k-i^2} \alpha^{-2i}}{(q;q)_k}
  \frac{([i]_q(1-q^{-i}/\alpha^2))^n  (\alpha\beta;q)_k}   {(q^{1-2i}/\alpha^2;q)_i (q^{1+2i}\alpha^2;q)_{k-i}}.
\end{align}
\end{thm}
\begin{proof}
By definition \eqref{eq:y-stirling} we have
$$
S_q(n,k,y)=S_q(n-1,k-1,y)+[k]_q(1-yq^{-k})S_q(n-1,k,y).
$$
Therefore
\begin{align}\label{eq:gfstirling}
\sum_{n\geq
k}S_q(n,k,y)t^n=\frac{t^k}{\prod_{i=1}^k(1-[i]_qt(1-q^{-i}y))}.
\end{align}
It follows from \eqref{eq:keydecomp} and \eqref{eq:gfstirling} that
\begin{equation}
S_q(n,k,y)=\frac{q^{-{k\choose 2}}}{k!_q} \sum_{i=1}^k {k\brack i}_q
y^{i-k}q^{k^2-i^2}
\frac{([i]_q(1-q^{-i}y))^n}{(q^{1-2i}y;q)_i(q^{1+2i}/y;q)_{k-i}}.
\end{equation}
Substituting this into \eqref{eq:explicit1} yields the desired formula.
\end{proof}

By Lemma~\ref{th:ACmomformula}  and \eqref{eq:gfstirling} we obtain
the generating function for the moments $\mu_n(\alpha,\beta)$:
\begin{equation}\label{momAC}
\sum_{n=0}^\infty \mu_n(\alpha,\beta) t^n= \sum_{k=0}^\infty
\frac{(\alpha\beta;\,q)_k q^{-{k\choose
2}}(1-q)^{-k}\alpha^{-2k}t^k}{\prod_{i=1}^k(1-[i]_q
t(1-q^{-i}/\alpha^2))}.
\end{equation}

The moment of $q$-Charlier polynomials corresponds to the
$\beta=0$, $\alpha=-1/\sqrt{a(1-q)}$ case, while that of
$q$-Laguerre polynomials corresponds to  the $\alpha=1/\sqrt{y}$,
$\alpha\beta=q$ case. Therefore,
\begin{align}
\sum_{n=0}^\infty \mu_n^{(c)}(a,q)t^n&=\sum_{k=0}^\infty \frac{(aqt)^k}{\prod_{i=1}^k(q^i-q^i[i]_qt+a(1-q)[i]_qt)};\label{eq:momcharlier}\\
\sum_{n=0}^\infty \mu_n^{(\ell)}(y,q)t^n&=\sum_{k=0}^\infty
\frac{k!_q(qty)^k}{\prod_{i=1}^k(q^i-q^i[i]_qt+[i]_qty)}.\label{eq:momlaguerre}
\end{align}
By Lemma~\ref{lem2},   we obtain, setting $p=1/q$,
\begin{align}
\sum_{n=0}^\infty \mu_n^{(c)}(a,q)t^n&=\sum_{i\geq 0}
\frac{a^i q^{2i}(1-a(1-q)p^{2i})/(a(1-q)p^i;p)_\infty}{i!_q q^{i^2}(q^i-q^i[i]_qt+a[i]_qt(1-q))},\label{eq:part}\\
\sum_{n=0}^\infty\mu_n^{(\ell)}(y,q)t^n&= \sum_{i\geq
0}\frac{y^i(q^{2i}-y)}{q^{i^2}(q^i-q^i[i]_{q}t+[i]_{q}ty)}.\label{eq:wi}
\end{align}
We derive then 
the following polynomial formulae in $a$ and 
$y$ for the corresponding moments:
\begin{align}
\mu_n^{(c)}(a,q)&= \sum_{k=1}^n a^k \sum_{l=0}^{k} \frac{[k-l]_q^n(-1)^{l}}{(k-l)!_q } 
\sum_{j=0}^{l} \frac{(1-q)^{j}}{(l-j)!_{q}} q^{\binom{l-j+1}{2}-k(k-l)}\biggl( \binom{n}{j} q^{k-l}+\binom{n}{j-1}
\biggr);\label{eq:momencharlier}\\
\mu_n^{(\ell)}(y,q)&=\sum_{k=1}^ny^k\sum_{i=0}^{k-1}(-1)^i[k-i]_q^n
q^{k(i-k)} \left({n\choose i}q^{k-i}+{n\choose i-1}\right).\label{eq:momentw}
\end{align}
Note that \eqref{eq:momencharlier} is simpler than the formula given in \cite[Proposition 5]{KSZ}.

%%%%****%%%%****%%%%****%%%%****%%%%****%%%%****%%%%****%%%%*
%%%%
%%%% New Section
%%%%
%%%%****%%%%****%%%%****%%%%****%%%%****%%%%****%%%%****%%%%*
\section{Linearization coefficients of the $q$-Laguerre polynomials}
Define the linearization coefficients of the $q$-Laguerre polynomials by
$$
I(n_{1},\ldots, n_{k})={{\mathcal L}}_q(L_{n_1}(x;\,q)\ldots L_{n_k}(x;\,q))\quad (k\geq 1,  n_{1},\ldots, n_{k}\geq 0).
$$
The following is our main result of this section.
\begin{thm}\label{eq:qlinear} We have
\begin{align}\label{eq:qlin}
I(n_{1},\ldots, n_{k})=\sum_{\sigma\in \D(n_1,\ldots, n_k)}y^{\exc
(\sigma)}q^{\cro(\sigma)}.
\end{align}
\end{thm}

For brevity, if
$n_1=\ldots=n_k=1$, we shall write $(1^{k}):=(n_1,\ldots,n_k)$ and  $\D_k:=\D(1^k)$.
Hence  $\D_n$ is just the set of  usual  derangements of $[n]$. Define also
$$
d_{n}(y,q)=\sum_{\sigma\in
\D_n}y^{\wex(\sigma)}q^{\cro(\sigma)}.
$$

A proof  \`a la Viennot (cf. \cite{ISV,KSZ}) of \eqref{eq:qlin}
would use the combinatorial interpretations  for the moments and
$q$-Laguerre polynomials to rewrite the left-hand side of
\eqref{eq:qlin} and then construct an adequate  \emph{killing
involution} on the resulting set. For the time being we do not have
such a proof to offer, instead we provide an inductive proof.

Since  $L_1(x;q)=x-y$, writing \eqref{eq:recurrqlaguerre} as
$$
L_1(x;q)L_n(x;q)=L_{n+1}(x;q)+(y q+1)[n]_qL_n(x;q)+y[n]_q^2L_{n-1}(x;q),
$$
we see immediately that
\begin{align}\label{keyrecurrenceprodpoly}
I(1,n, n_{1},\ldots, n_{k})&=
I(n+1, n_{1},\ldots, n_{k})\nonumber\\
&+(y q+1)[n]_q\,I(n, n_{1},\ldots, n_{k})
+y [n]_q^2\,I(n-1,n_{1},\ldots, n_{k}).
\end{align}
Therefore,  the sequence
$(I(n_{1},\ldots, n_{k}))$ ($k\geq 1, n_{1}, \ldots, n_{k}\geq 0$)  is
completely determined  by the  recurrence relation \eqref{keyrecurrenceprodpoly} and the following items:
\begin{itemize}
\item[(i)] the special values $I(1^{k})$ for all $k\geq 1$,
\item[(ii)]   the symmetry of $I(n_{1},\ldots, n_{k})$ with respect to the  indices
$n_1,\ldots,n_k$.
\end{itemize}
Our proof of Theorem~\ref{eq:qlinear} will consist in verifying that the right-hand side  of \eqref{eq:qlin} has the same special values  at $(1^{k})$ as the right-hand side, is invariant by  rearrangement of the indices and  satisfies  the same recurrence relation.

%%%%%%%%%
\begin{lem} \label{lem: n=1}
We have $I(1^{n})=d_{n}(y,q)$ for all $n\geq 1$.
% ${{\mathcal
%L}}_q(L_1(x;q)^n)=\sum_{\sigma\in
%\D_n}y^{\wex(\sigma)}q^{\cro(\sigma)}$.
\end{lem}
\begin{proof}
Since  $L_1(x;q)=x-y$,
by definition,
$$
I(1^{n})={{\mathcal L}}_q((x-y)^n)=\sum_{k=0}^n(-1)^{n-k}{n\choose k}y^{n-k}
 \mu_k^{(\ell)}(y,q).
$$
By binomial inversion and \eqref{defmoment}, it suffices to prove that
$$
 \sum_{\sigma\in
S_n}y^{\exc(\sigma)}q^{\cro(\sigma)}=\sum_{k=0}^n{n\choose k}y^kd_{n-k}(y,q).
$$
But the latter identity is obvious.
\end{proof}

Since the two cyclic permutations $(1,2)$ and $(1,2,3,\ldots,k)$
generate the symmetric group $\S_k$,
the invariance of $\sum_{\sigma\in \D(n_1,n_2,\ldots, n_k)}y^{\wex
(\sigma)}q^{\cro(\sigma)}$ by permuting the $n_i's$  will
be  a consequence  of the following two special cases.
\begin{lem}\label{lem:1} We have
\begin{align}\label{eqlem:1}
\sum_{\sigma\in \D(n_1,n_2,\ldots, n_k)}y^{\exc
(\sigma)}q^{\cro(\sigma)}&= \sum_{\sigma\in
\D(n_2,n_3,\ldots,n_k,n_1)}y^{\exc (\sigma)}q^{\cro(\sigma)}.
\end{align}
\end{lem}

\begin{lem}\label{lem:2}
We have
\begin{align}
\sum_{\sigma\in \D(n_1,n_2,\ldots, n_k)}y^{\exc
(\sigma)}q^{\cro(\sigma)}&= \sum_{\sigma\in
\D(n_2,n_1,n_3\ldots,n_k)}y^{\exc (\sigma)}q^{\cro(\sigma)}.\label{eqlem:2}
\end{align}
\end{lem}
We postpone the proof of the above two lemmas to the next two
sections.

\medskip
%%%%%%%%%%

\noindent \emph{\bf Proof of Theorem~\ref{eq:qlinear}}. By Lemmas~\ref{lem: n=1}, \ref{lem:1} and \ref{lem:2},
 it   suffices to check that
\begin{align}\label{keyrecurrence}
\sum_{\sigma\in \D(1,n,n_1,\ldots, n_k)}w(\sigma)= \sum_{\sigma\in
\D(n+1,n_1,\ldots,n_k)}w(\sigma)
+(y q+1)[n]_q\sum_{\sigma\in \D(n,n_1,\ldots,n_k)}w(\sigma)\\
+y [n]_q^2\sum_{\sigma\in \D(n-1,n_1,\ldots,n_k)}w(\sigma),\nonumber
\end{align}
where $w(\sigma)=y^{\exc (\sigma)}q^{\cro(\sigma)}$.

 For derangements  $\sigma\in \D(1,n,n_1,\ldots, n_k)$ we will distinguish four cases.
 In  each case, we shall describe a mapping to compute the corresponding  enumerative polynomial.
The reader is refereed to  Table~1 and Table~4 in Section~5 for an illustration of these mappings in order
 to have a better understanding of their properties.

\begin{itemize}
\item[a)] $\sigma(1), \sigma^{-1}(1)>n+1$. We can identify such a derangement in $ \D(1,n,n_1,\ldots, n_k)$ with
a derangement in $\D(n+1,n_1,\ldots,n_k)$.  So the corresponding enumerative polynomial
is
$$
\sum_{\sigma\in \D(n+1,n_1,\ldots,n_k)}y^{\exc
(\sigma)}q^{\cro(\sigma)}.
$$

\item[b)] $\sigma(1)\in \{2,\ldots, n+1\}$ and $\sigma^{-1}(1)>n+1$. Let
$\sigma(1)=\ell$. We define the mapping $\sigma\mapsto \sigma'\in
\D(n,n_1,\ldots,n_k)$ by
$$
\left\{%
\begin{array}{ll}
   \sigma'(i)= \sigma(i+1)-1, & \hbox{if $1\leq i\leq n$;} \\
   \sigma'^{-1}(i)=\sigma^{-1}(i)-1 & \hbox{if $1\leq i\leq \ell-1$;} \\
   \sigma'^{-1}(i)=\sigma^{-1}(i+1)-1 & \hbox{if $\ell\leq i\leq n$;} \\
   \sigma'(i-1)=\sigma(i)-1 & \hbox{if $\sigma(i)> i>n+1$;} \\
   \sigma'^{-1}(i-1)=\sigma^{-1}(i)-1 & \hbox{if $\sigma^{-1}(i)>i>n+1$;} \\
\end{array}%
\right.
$$
 Clearly
$w(\sigma)=y q^{\ell-1}w(\sigma')$. Moreover, for each given $\ell \in \{2,\ldots, n+1\}$,  the above mapping is a bijection from permutations $\sigma \in
\D(1,n,n_1,\ldots, n_k)$ satisfying $\sigma(1)=\ell$ and $\sigma^{-1}(1)>n+1$ to permutations
in~$\D(n,n_1,\ldots,n_k)$.
Summing over all $\ell=2,\ldots, n+1$ yields the generating function:
$$
qy[n]_q\sum_{\sigma\in \D(n,n_1,\ldots,n_k)}y^{\exc
(\sigma)}q^{\cro(\sigma)}.
$$
%%%%%%%
\item[c)] $\sigma(1)>n+1$ and $\sigma^{-1}(1)\in \{2,\ldots, n+1\}$.
Let $\sigma^{-1}(1)=\ell$. We define the mapping $\sigma\mapsto \sigma'\in
\D(n,n_1,\ldots,n_k)$ by
$$
\left\{%
\begin{array}{ll}
   \sigma'(i)= \sigma(i)-1, & \hbox{if $1\leq i\leq \ell-1$;}\\
   \sigma'(i)=\sigma(i+1)-1 & \hbox{if $\ell\leq i\leq n$;} \\
   \sigma'^{-1}(i)=\sigma^{-1}(i+1)-1 & \hbox{if $1\leq i\leq n$;} \\
   \sigma'(i-1)=\sigma(i)-1 & \hbox{if $\sigma(i)> i>n+1$;} \\
   \sigma'^{-1}(i-1)=\sigma^{-1}(i)-1 & \hbox{if $\sigma^{-1}(i)>i>n+1$;} \\
\end{array}%
\right.
$$
 Clearly
$w(\sigma)=q^{\ell-2}w(\sigma')$. Moreover, for each given $\ell \in \{2,\ldots, n+1\}$,  the above mapping is a bijection from permutations $\sigma \in
\D(1,n,n_1,\ldots, n_k)$ satisfying $\sigma^{-1}(1)=\ell$ and $\sigma(1)>n+1$ to permutations
in~$\D(n,n_1,\ldots,n_k)$.  Summing over all $\ell=2,\ldots, n+1$
yields the generating function:
$$
[n]_q\sum_{\sigma\in \D(n,n_1,\ldots,n_k)}y^{\exc
(\sigma)}q^{\cro(\sigma)}.
$$
%%%%%%%%%%%
\item[d)] $\sigma(1), \sigma^{-1}(1)\in \{2,\ldots, n+1\}$.
Let $\sigma(1)=\ell_1$ and $\sigma^{-1}(1)=\ell_2$.
Then we define the mapping $\sigma\mapsto \sigma'\in \D(n-1,n_1,\ldots,n_k)$
by
$$
\left\{%
\begin{array}{ll}
   \sigma'(i)= \sigma(i+1)-2, & \hbox{if $1\leq i\leq \ell_2-2$;}\\
   \sigma'(i)=\sigma(i+2)-2 & \hbox{if $\ell_2-1\leq i\leq n-1$;} \\
   \sigma'^{-1}(i)=\sigma^{-1}(i+1)-2 & \hbox{if $1\leq i\leq \ell_1-2$;} \\
   \sigma'^{-1}(i)=\sigma^{-1}(i+2)-2 & \hbox{if $\ell_1-1\leq i\leq n-1$;} \\
   \sigma'(i-2)=\sigma(i)-2 & \hbox{if $\sigma(i)> i>n+1$;} \\
   \sigma'^{-1}(i-2)=\sigma^{-1}(i)-2 & \hbox{if $\sigma^{-1}(i)>i>n+1$;} \\
\end{array}%
\right.
$$
Clearly $w(\sigma)=y q^{(\ell_1+\ell_2-4)}w(\sigma')$.
Moreover, for each given $\ell_1,\ell_2\in \{2,\ldots, n+1\}$, the above mapping is a bijection from permutations $\sigma \in
\D(1,n,n_1,\ldots, n_k)$ satisfying $\sigma(1)=\ell_1$ and $\sigma^{-1}(1)=\ell_2$ to permutations
in~$\D(n-1,n_1,\ldots,n_k)$. Summing over all $\ell_1,\ell_2\in \{2,\ldots, n+1\}$ yields the
generating function:
$$
y[n]_q^2\sum_{\sigma\in \D(n-1,n_1,\ldots,n_k)}y^{\exc
(\sigma)}q^{\cro(\sigma)}.
$$

\end{itemize}
Summing up we obtain \eqref{keyrecurrence}. \qed

\medskip
When $k=2$, Theorem~\ref{eq:qlinear} reduces to  the orthogonality
of the $q$-Laguerre polynomials~\eqref{eq:orthog}. When $k=3$, we
can derive the following explicit formula from
Theorem~\ref{thm:alsalamlinear}.
%%%%%%%%%%%%%%%%%%%%
\begin{thm}\label{lin3} We have
\begin{align*}
I(n_{1}, n_{2}, n_{3})
&=\sum_{s}\frac{n_1!_q\,
n_2!_q\,n_3!_q\,s!_q\,y^s}{(n_1+n_2+n_3-2s)!_q
(s-n_3)!_q(s-n_2)!_q(s-n_1)!_q}\\
&\qquad \times \sum_{k}{n_1+n_2+n_3-2s\brack k}_q
y^{k}q^{{k+1\choose 2}+ {n_1+n_2+n_3-2s-k\choose 2}}.
\end{align*}
\end{thm}\begin{proof}
By Theorem~\ref{thm:alsalamlinear} with $a=\frac{1}{\sqrt{y}}$ and
$b=\sqrt{y}q$ we have
\begin{align*}
%&{\mathcal L}_q(L_{n_1}(x;\,q)L_{n_2}(x;\,q)L_{n_3}(x;\,q))
I(n_{1}, n_{2}, n_{3})
&= {\mathcal L}_q(L_{n_3}(x;\,q)^2)\left(\frac{\sqrt{y}}{q-1}\right)^{n_1+n_2-n_3} C_{n_1,n_2}^{n_3}(a,b;\,q)\\
&=\sum_{m_2,m_3}\frac{\,n_1!_q
\,n_2!_q\,n_3!_q\,(n_1+m_3)!_q\,y^{n_2+n_3-m_2-m_3}\, q^{{m_2\choose
2}+{M+1\choose 2}}}{M!_q\,
m_2!_q\,(m_3+n_1-n_3)!_q\,(m_3+n_1-n_2)!_q\,m_3!_q\,},
\end{align*}
where $M=n_3+n_2-n_1-m_2-2m_3$.  Substituting  $s=n_1+m_3$ and $k=n_3+n_2-n_1-m_2-2m_3$ in the last
sum yields the desired formula.
\end{proof}

\begin{remark}
It would be interesting to give a combinatorial proof of the above
result as in \cite{ISV,KSZ}. When $q=1$ such a proof  was given in
\cite{Ze2}.
\end{remark}

We end this section with an example. If ${\bf n}=(2,2,1)$,  by
Theorem~\ref{lin3} we have
\begin{align}
I(2,2,1)&=
\sum_{s}\frac{2!_{q}2!_{q}1!_{q}s!_{q}y^s}{(5-2s)!_{q}(s-1)!_{q}(s-2)!_{q}(s-2)!_{q}}
%\nonumber \\
%&\qquad\qquad \times
\sum_{k\geq 0}{5-2s\brack k}_{q} y^kq^{{k+1\choose 2}+{5-2s-k\choose 2}}\nonumber \\
&=(1+q)^3(1+qy )y^2.\label{eq:check}
\end{align}

On the other  hand, the sixteen generalized derangements in $\D(2,2,1)$,  depicted by their
diagrams  and the corresponding weights are tabulated as follows:
%%%%%%%%%%%%%%%%%%%%%%%%%%%%%
%\newpage
%%%%%%%%%%%%%%{1}%%%%%%%%%%%%%%

$$
 {\setlength{\unitlength}{0.6mm}
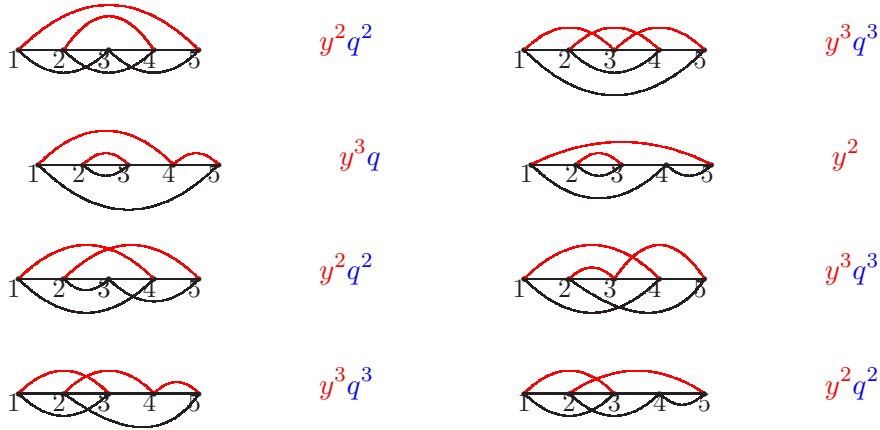
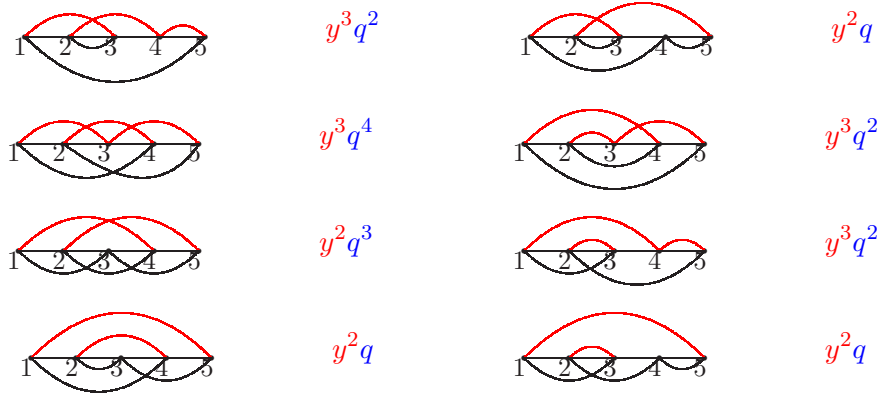
\begin{picture}(50,20)(0,0)
\put(0,0){\circle*{1,3}}\put(0,0){\makebox(-2,-4)[c]{\small 1}}
\put(10,0){\circle*{1,3}}\put(10,0){\makebox(-2,-4)[c]{\small 2}}
\put(20,0){\circle*{1,3}}\put(20,0){\makebox(-2,-4)[c]{\small 3}}
\put(30,0){\circle*{1,3}}\put(30,0){\makebox(-2,-4)[c]{\small 4}}
\put(40,0){\circle*{1,3}}\put(40,0){\makebox(-2,-4)[c]{\small 5}}
\red{\qbezier(0,0)(20,20)(40,0) \qbezier(10,0)(20,15)(30,0)}
\put(0,0){\line(1,0){40}} \qbezier(0,0)(10,-10)(20,0)
\qbezier(10,0)(20,-10)(30,0) \qbezier(20,0)(30,-10)(40,0)
\end{picture}
}\hspace{1cm} \red{y^2}\blue{q^2} \hspace{2cm}
{\setlength{\unitlength}{0.6mm}
\begin{picture}(50,20)(0,0)
\put(0,0){\circle*{1,3}}\put(0,0){\makebox(-2,-4)[c]{\small 1}}
\put(10,0){\circle*{1,3}}\put(10,0){\makebox(-2,-4)[c]{\small 2}}
\put(20,0){\circle*{1,3}}\put(20,0){\makebox(-2,-4)[c]{\small 3}}
\put(30,0){\circle*{1,3}}\put(30,0){\makebox(-2,-4)[c]{\small 4}}
\put(40,0){\circle*{1,3}}\put(40,0){\makebox(-2,-4)[c]{\small 5}}
\red{\qbezier(0,0)(10,10)(20,0) \qbezier(20,0)(30,10)(40,0)
\qbezier(10,0)(20,10)(30,0)} \qbezier(0,0)(20,-20)(40,0)
\qbezier(10,0)(20,-10)(30,0)} \put(0,0){\line(1,0){40}
\end{picture}
}\hspace{1cm} \red{y^3}\blue{q^3}
$$
%%%%%%%%%%%%%%%%%%
$$
{\setlength{\unitlength}{0.6mm}
\begin{picture}(50,20)(0,0)
\put(0,0){\circle*{1,3}}\put(0,0){\makebox(-2,-4)[c]{\small 1}}
\put(10,0){\circle*{1,3}}\put(10,0){\makebox(-2,-4)[c]{\small 2}}
\put(20,0){\circle*{1,3}}\put(20,0){\makebox(-2,-4)[c]{\small 3}}
\put(30,0){\circle*{1,3}}\put(30,0){\makebox(-2,-4)[c]{\small 4}}
\put(40,0){\circle*{1,3}}\put(40,0){\makebox(-2,-4)[c]{\small 5}}
\red{\qbezier(0,0)(15,15)(30,0) \qbezier(10,0)(15,5)(20,0)
\qbezier(30,0)(35,5)(40,0)} \put(0,0){\line(1,0){40}}
\qbezier(0,0)(20,-20)(40,0) \qbezier(10,0)(15,-5)(20,0)
\end{picture}
}\hspace{1cm} \red{y^3}\blue{q} \hspace{2cm}
{\setlength{\unitlength}{0.6mm}
\begin{picture}(50,20)(0,0)
\put(0,0){\circle*{1,3}}\put(0,0){\makebox(-2,-4)[c]{\small 1}}
\put(10,0){\circle*{1,3}}\put(10,0){\makebox(-2,-4)[c]{\small 2}}
\put(20,0){\circle*{1,3}}\put(20,0){\makebox(-2,-4)[c]{\small 3}}
\put(30,0){\circle*{1,3}}\put(30,0){\makebox(-2,-4)[c]{\small 4}}
\put(40,0){\circle*{1,3}}\put(40,0){\makebox(-2,-4)[c]{\small 5}}
\red{\qbezier(0,0)(20,10)(40,0) \qbezier(10,0)(15,5)(20,0)}
\put(0,0){\line(1,0){40}}
%%%%%%%%%%%%%%
\qbezier(0,0)(15,-15)(30,0) \qbezier(10,0)(15,-5)(20,0)
\qbezier(0,0)(15,-15)(30,0) \qbezier(30,0)(35,-5)(40,0)
\end{picture}
}\hspace{1cm} \red{y^2}
$$
%%%%%%%%%%%%%%%%%%%%
$$
{\setlength{\unitlength}{0.6mm}
\begin{picture}(50,20)(0,0)
\put(0,0){\circle*{1,3}}\put(0,0){\makebox(-2,-4)[c]{\small 1}}
\put(10,0){\circle*{1,3}}\put(10,0){\makebox(-2,-4)[c]{\small 2}}
\put(20,0){\circle*{1,3}}\put(20,0){\makebox(-2,-4)[c]{\small 3}}
\put(30,0){\circle*{1,3}}\put(30,0){\makebox(-2,-4)[c]{\small 4}}
\put(40,0){\circle*{1,3}}\put(40,0){\makebox(-2,-4)[c]{\small 5}}
\red{\qbezier(0,0)(15,15)(30,0) \qbezier(10,0)(25,15)(40,0) }
\put(0,0){\line(1,0){40}}
%%%%%%%%%%%%%%
\qbezier(0,0)(15,-15)(30,0) \qbezier(10,0)(15,-5)(20,0)
\qbezier(20,0)(30,-10)(40,0)
\end{picture}
}\hspace{1cm} \red{y^2}\blue{q^2} \hspace{2cm}
{\setlength{\unitlength}{0.6mm}
\begin{picture}(50,20)(0,0)
\put(0,0){\circle*{1,3}}\put(0,0){\makebox(-2,-4)[c]{\small 1}}
\put(10,0){\circle*{1,3}}\put(10,0){\makebox(-2,-4)[c]{\small 2}}
\put(20,0){\circle*{1,3}}\put(20,0){\makebox(-2,-4)[c]{\small 3}}
\put(30,0){\circle*{1,3}}\put(30,0){\makebox(-2,-4)[c]{\small 4}}
\put(40,0){\circle*{1,3}}\put(40,0){\makebox(-2,-4)[c]{\small 5}}
\red{\qbezier(0,0)(15,15)(30,0) \qbezier(10,0)(15,5)(20,0)
\qbezier(20,0)(30,15)(40,0) } \put(0,0){\line(1,0){40}}
%%%%%%%%%%%%%%
\qbezier(0,0)(15,-15)(30,0) \qbezier(10,0)(30,-15)(40,0)
\end{picture}
}\hspace{1cm} \red{y^3}\blue{q^3}
$$
%%%%%%%%%%%%%%%%%%%%%%%%%%%%%%%
$$
{\setlength{\unitlength}{0.6mm}
\begin{picture}(50,20)(0,0)
\put(0,0){\circle*{1,3}}\put(0,0){\makebox(-2,-4)[c]{\small 1}}
\put(10,0){\circle*{1,3}}\put(10,0){\makebox(-2,-4)[c]{\small 2}}
\put(20,0){\circle*{1,3}}\put(20,0){\makebox(-2,-4)[c]{\small 3}}
\put(30,0){\circle*{1,3}}\put(30,0){\makebox(-2,-4)[c]{\small 4}}
\put(40,0){\circle*{1,3}}\put(40,0){\makebox(-2,-4)[c]{\small 5}}
\red{\qbezier(0,0)(10,10)(20,0) \qbezier(10,0)(20,10)(30,0)
\qbezier(30,0)(35,5)(40,0) } \put(0,0){\line(1,0){40}}
%%%%%%%%%%%%%%
\qbezier(0,0)(10,-10)(20,0) \qbezier(10,0)(30,-15)(40,0)
\end{picture}
}\hspace{1cm} \red{y^3}\blue{q^3} \hspace{2cm}
{\setlength{\unitlength}{0.6mm}
\begin{picture}(50,20)(0,0)
\put(0,0){\circle*{1,3}}\put(1,0){\makebox(-2,-4)[c]{\small 1}}
\put(10,0){\circle*{1,3}}\put(11,0){\makebox(-2,-4)[c]{\small 2}}
\put(20,0){\circle*{1,3}}\put(21,0){\makebox(-2,-4)[c]{\small 3}}
\put(30,0){\circle*{1,3}}\put(31,0){\makebox(-2,-4)[c]{\small 4}}
\put(40,0){\circle*{1,3}}\put(41,0){\makebox(-2,-4)[c]{\small 5}}
\red{\qbezier(0,0)(10,10)(20,0) \qbezier(10,0)(25,10)(40,0) }
\put(0,0){\line(1,0){40}}
%%%%%%%%%%%%%%
\qbezier(0,0)(10,-10)(20,0) \qbezier(10,0)(20,-10)(30,0)
\qbezier(30,0)(35,-5)(40,0)
\end{picture}
}\hspace{1cm} \red{y^2}\blue{q^2}
$$
%%%%%%%%%%%%%%%%%%%
$$
{\setlength{\unitlength}{0.6mm}
\begin{picture}(50,20)(0,0)
\put(0,0){\circle*{1,3}}\put(0,0){\makebox(-2,-4)[c]{\small 1}}
\put(10,0){\circle*{1,3}}\put(10,0){\makebox(-2,-4)[c]{\small 2}}
\put(20,0){\circle*{1,3}}\put(20,0){\makebox(-2,-4)[c]{\small 3}}
\put(30,0){\circle*{1,3}}\put(30,0){\makebox(-2,-4)[c]{\small 4}}
\put(40,0){\circle*{1,3}}\put(40,0){\makebox(-2,-4)[c]{\small 5}}
\red{\qbezier(0,0)(10,10)(20,0) \qbezier(10,0)(20,10)(30,0)
\qbezier(30,0)(35,5)(40,0) } \put(0,0){\line(1,0){40}}
%%%%%%%%%%%%%%
\qbezier(0,0)(20,-20)(40,0) \qbezier(10,0)(15,-5)(20,0)
\end{picture}
}\hspace{1cm} \red{y^3}\blue{q^2} \hspace{2cm}
 {\setlength{\unitlength}{0.6mm}
\begin{picture}(50,20)(0,0)
\put(0,0){\circle*{1,3}}\put(0,0){\makebox(-2,-4)[c]{\small 1}}
\put(10,0){\circle*{1,3}}\put(10,0){\makebox(-2,-4)[c]{\small 2}}
\put(20,0){\circle*{1,3}}\put(20,0){\makebox(-2,-4)[c]{\small 3}}
\put(30,0){\circle*{1,3}}\put(30,0){\makebox(-2,-4)[c]{\small 4}}
\put(40,0){\circle*{1,3}}\put(40,0){\makebox(-2,-4)[c]{\small 5}}
\red{\qbezier(0,0)(10,10)(20,0) \qbezier(10,0)(25,15)(40,0)}
\put(0,0){\line(1,0){40}}
%%%%%%%%%%%%%%
\qbezier(0,0)(15,-15)(30,0) \qbezier(10,0)(15,-5)(20,0)
\qbezier(30,0)(35,-5)(40,0)
\end{picture}
}\hspace{1cm} \red{y^2}\blue{q}
$$
%%%%%%%%%%%%%%%%%%%
$$
 {\setlength{\unitlength}{0.6mm}
\begin{picture}(50,20)(0,0)
\put(0,0){\circle*{1,3}}\put(0,0){\makebox(-2,-4)[c]{\small 1}}
\put(10,0){\circle*{1,3}}\put(10,0){\makebox(-2,-4)[c]{\small 2}}
\put(20,0){\circle*{1,3}}\put(20,0){\makebox(-2,-4)[c]{\small 3}}
\put(30,0){\circle*{1,3}}\put(30,0){\makebox(-2,-4)[c]{\small 4}}
\put(40,0){\circle*{1,3}}\put(40,0){\makebox(-2,-4)[c]{\small 5}}
\red{\qbezier(0,0)(10,10)(20,0) \qbezier(10,0)(20,10)(30,0)
\qbezier(20,0)(30,10)(40,0) } \put(0,0){\line(1,0){40}}
%%%%%%%%%%%%%%
\qbezier(0,0)(15,-15)(30,0) \qbezier(10,0)(30,-15)(40,0)
\end{picture}
}\hspace{1cm} \red{y^3}\blue{q^4} \hspace{2cm}
{\setlength{\unitlength}{0.6mm}
\begin{picture}(50,20)(0,0)
\put(0,0){\circle*{1,3}}\put(0,0){\makebox(-2,-4)[c]{\small 1}}
\put(10,0){\circle*{1,3}}\put(10,0){\makebox(-2,-4)[c]{\small 2}}
\put(20,0){\circle*{1,3}}\put(20,0){\makebox(-2,-4)[c]{\small 3}}
\put(30,0){\circle*{1,3}}\put(30,0){\makebox(-2,-4)[c]{\small 4}}
\put(40,0){\circle*{1,3}}\put(40,0){\makebox(-2,-4)[c]{\small 5}}
\red{\qbezier(0,0)(15,15)(30,0) \qbezier(10,0)(15,5)(20,0)
\qbezier(20,0)(30,10)(40,0) } \put(0,0){\line(1,0){40}}
%%%%%%%%%%%%%%
\qbezier(0,0)(20,-20)(40,0) \qbezier(10,0)(20,-10)(30,0)
\end{picture}
}\hspace{1cm} \red{y^3}\blue{q^2}
$$
%%%%%%%%%%%%%%%%%%%
$$
{\setlength{\unitlength}{0.6mm}
\begin{picture}(50,20)(0,0)
\put(0,0){\circle*{1,3}}\put(0,0){\makebox(-2,-4)[c]{\small 1}}
\put(10,0){\circle*{1,3}}\put(10,0){\makebox(-2,-4)[c]{\small 2}}
\put(20,0){\circle*{1,3}}\put(20,0){\makebox(-2,-4)[c]{\small 3}}
\put(30,0){\circle*{1,3}}\put(30,0){\makebox(-2,-4)[c]{\small 4}}
\put(40,0){\circle*{1,3}}\put(40,0){\makebox(-2,-4)[c]{\small 5}}
\red{\qbezier(0,0)(15,15)(30,0) \qbezier(10,0)(25,15)(40,0)}
\put(0,0){\line(1,0){40}}
%%%%%%%%%%%%%%
\qbezier(0,0)(10,-10)(20,0) \qbezier(10,0)(20,-10)(30,0)
\qbezier(20,0)(30,-10)(40,0)
\end{picture}
}\hspace{1cm} \red{y^2}\blue{q^3} \hspace{2cm}
 {\setlength{\unitlength}{0.6mm}
\begin{picture}(50,20)(0,0)
\put(0,0){\circle*{1,3}}\put(0,0){\makebox(-2,-4)[c]{\small 1}}
\put(10,0){\circle*{1,3}}\put(10,0){\makebox(-2,-4)[c]{\small 2}}
\put(20,0){\circle*{1,3}}\put(20,0){\makebox(-2,-4)[c]{\small 3}}
\put(30,0){\circle*{1,3}}\put(30,0){\makebox(-2,-4)[c]{\small 4}}
\put(40,0){\circle*{1,3}}\put(40,0){\makebox(-2,-4)[c]{\small 5}}
\red{\qbezier(0,0)(15,15)(30,0) \qbezier(10,0)(15,5)(20,0)
\qbezier(30,0)(35,5)(40,0)} \put(0,0){\line(1,0){40}}
%%%%%%%%%%%%%%%%%%%%%%%%%%%%%%%%
\qbezier(0,0)(10,-10)(20,0) \qbezier(10,0)(25,-15)(40,0)
\end{picture}
}\hspace{1cm} \red{y^3}\blue{q^2}
$$
%%%%%%%%%%%%%%%%%%%
$$
{\setlength{\unitlength}{0.6mm}
\begin{picture}(50,20)(0,0)
\put(0,0){\circle*{1,3}}\put(0,0){\makebox(-2,-4)[c]{\small 1}}
\put(10,0){\circle*{1,3}}\put(10,0){\makebox(-2,-4)[c]{\small 2}}
\put(20,0){\circle*{1,3}}\put(20,0){\makebox(-2,-4)[c]{\small 3}}
\put(30,0){\circle*{1,3}}\put(30,0){\makebox(-2,-4)[c]{\small 4}}
\put(40,0){\circle*{1,3}}\put(40,0){\makebox(-2,-4)[c]{\small 5}}
\red{\qbezier(0,0)(20,20)(40,0) \qbezier(10,0)(20,10)(30,0)}
\put(0,0){\line(1,0){40}}
%%%%%%%%%%%%%%
\qbezier(0,0)(15,-15)(30,0) \qbezier(10,0)(15,-5)(20,0)
\qbezier(20,0)(30,-10)(40,0)
\end{picture}
}\hspace{1cm} \red{y^2}\blue{q} \hspace{2cm}
{\setlength{\unitlength}{0.6mm}
\begin{picture}(50,20)(0,0)
\put(0,0){\circle*{1,3}}\put(0,0){\makebox(-2,-4)[c]{\small 1}}
\put(10,0){\circle*{1,3}}\put(10,0){\makebox(-2,-4)[c]{\small 2}}
\put(20,0){\circle*{1,3}}\put(20,0){\makebox(-2,-4)[c]{\small 3}}
\put(30,0){\circle*{1,3}}\put(30,0){\makebox(-2,-4)[c]{\small 4}}
\put(40,0){\circle*{1,3}}\put(40,0){\makebox(-2,-4)[c]{\small 5}}
\red{\qbezier(0,0)(20,20)(40,0) \qbezier(10,0)(15,5)(20,0)}
\put(0,0){\line(1,0){40}}
%%%%%%%%%%%%%%
\qbezier(0,0)(10,-10)(20,0) \qbezier(10,0)(20,-10)(30,0)
\qbezier(30,0)(35,-5)(40,0)
\end{picture}
}\hspace{1cm} \red{y^2}\blue{q}
$$
\medskip

Summing up we get $\sum_{\sigma \in \D(2,2,1)} y^{\exc \sigma}
q^{\cro\sigma}= \red{y^2}(1+\blue{q} \red{y}) (1+\blue{q})^3$, which
coincides with~\eqref{eq:check}.

%%%%****%%%%****%%%%****%%%%****%%%%****%%%%****%%%%****%%%%*
%%%%
%%%% New Section
%%%%
%%%%****%%%%****%%%%****%%%%****%%%%****%%%%****%%%%****%%%%*
\section{Proof of  Lemma~\ref{lem:1}}
For each fixed $k\in [n]$ define the two subsets of $\S_n$:
\begin{align*}
{^k\S_{n}}&= \{\sigma\,\in S_{n}\,|\, \sigma(i)>k\quad \text{for $1\leq i\leq k$} \}, \\
\S_{n}^k&=\{\sigma\,\in S_{n}\,|\, \sigma(n+1-i)< n+1-k \quad
\text{for $1\leq i\leq k$} \}.
\end{align*}
We first  define  a simple bijection $\Phi_k: \sigma\mapsto \sigma'$
from $ {^k\S_{n}}$ to $\S_{n}^k$ as follows:  for $1\leq i\leq n$,
$$
\sigma'(i)=\left\{%
\begin{array}{ll}
    \sigma(i+k)-k, & \hbox{if $1\leq i\leq n-k$ and $\sigma(i+k)>k$;} \\
    \sigma(i+k)+n-k, & \hbox{if $1\leq i\leq n-k$ and $\sigma(i+k)\leq k$;} \\
\sigma(i+k-n)-k,& \hbox{if $n-k+1\leq i\leq n$.} \\
\end{array}%
\right.
$$
The map  is  illustrated  by  the diagrams of permutations  in Table~1.
\begin{table}[h]
$$
\begin{array}{|cccc|}
\hline
& \sigma &\longrightarrow &\hspace{1cm} \sigma'\\
\hline
%%%%%%%%%%%%%%%%%%%%%%%%%%%%%%%%%%%%%%%%%%%%%%%%%%%%%%%%% type 1
&\hspace{1cm} {\setlength{\unitlength}{0.6mm}
\begin{picture}(80,15)(0,-5)
\put(0,0){\line(1,0){25}}\put(35,0){\line(1,0){35}}
\put(0,0){\circle*{1,3}}\put(0,0){\makebox(-2,-6)[c]{\tiny $1$}}
\put(25,0){\circle*{1,3}}\put(25,0){\makebox(-2,-6)[c]{\tiny $k$}}
\put(35,0){\circle*{1,3}}\put(35,0){\makebox(-5,-6)[c]{\tiny $k+1$}}
\put(70,0){\circle*{1,3}}\put(70,0){\makebox(-2,-6)[c]{\tiny $n$}}
\qbezier(45,0)(50,10)(55,0)\put(45,0){\makebox(-2,-6)[c]{\tiny $i$}}
\put(55,0){\makebox(-2,-6)[c]{\tiny $\sigma(i)$}}
\end{picture}
}&\longrightarrow&\hspace{1cm}{\setlength{\unitlength}{0.6mm}
\begin{picture}(80,15)(0,-5)
\put(0,0){\line(1,0){35}}\put(45,0){\line(1,0){25}}
\put(0,0){\circle*{1,3}}\put(0,0){\makebox(-2,-6)[c]{\tiny $1$}}
\put(35,0){\circle*{1,3}}\put(35,0){\makebox(2,-6)[c]{\tiny
$n\!-\!k$}} \put(45,0){\circle*{1,3}}
\put(70,0){\circle*{1,3}}\put(70,0){\makebox(-2,-6)[c]{\tiny $n$}}
\qbezier(8,0)(15,10)(22,0) \put(10,0){\makebox(-5,-6)[c]{\tiny
$i\!-\!k$}} \put(20,0){\makebox(2,-6)[c]{\tiny $\sigma(i)\!-\!k$}}
\end{picture}
}\\
%%%%%%%%%%%%%%%%%%%%%%%%%%%%%%%%%%%%%%%%%%%%%%%%%%%%%%%%% type 2
&\hspace{1cm} {\setlength{\unitlength}{0.6mm}
\begin{picture}(80,15)(0,-5)
\put(0,0){\line(1,0){25}}\put(35,0){\line(1,0){35}}
\put(0,0){\circle*{1,3}}\put(0,0){\makebox(-2,-6)[c]{\tiny $1$}}
\put(25,0){\circle*{1,3}}\put(25,0){\makebox(-2,-6)[c]{\tiny $k$}}
\put(35,0){\circle*{1,3}}
\put(70,0){\circle*{1,3}}\put(70,0){\makebox(-2,-6)[c]{\tiny $n$}}
\qbezier(45,0)(50,-10)(55,0)\put(45,0){\makebox(-2,6)[c]{\tiny
$\sigma(i)$}} \put(55,0){\makebox(-2,6)[c]{\tiny $i$}}
\end{picture}
}&\longrightarrow&\hspace{1cm}{\setlength{\unitlength}{0.6mm}
\begin{picture}(80,15)(0,-5)
\put(0,0){\line(1,0){35}}\put(45,0){\line(1,0){25}}
\put(0,0){\circle*{1,3}}\put(0,0){\makebox(-2,-6)[c]{\tiny $1$}}
\put(35,0){\circle*{1,3}}\put(35,0){\makebox(2,-6)[c]{\tiny
$n\!-\!k$}} \put(45,0){\circle*{1,3}}
\put(70,0){\circle*{1,3}}\put(70,0){\makebox(-2,-6)[c]{\tiny $n$}}
\qbezier(8,0)(15,-10)(22,0) \put(10,0){\makebox(-2,6)[c]{\tiny
$\sigma(i)\!-\!k$}} \put(20,0){\makebox(5,6)[c]{\tiny $i\!-\!k$}}
\end{picture}
}\\
%%%%%%%%%%%%%%%%%%%%%%%%%%%%%%%%%%%%%%%%%%%%%%%%%%%%%%%%type 4
&\hspace{1cm} {\setlength{\unitlength}{0.6mm}
\begin{picture}(80,15)(0,-5)
\put(0,0){\line(1,0){25}}\put(35,0){\line(1,0){35}}
\put(0,0){\circle*{1,3}}\put(0,0){\makebox(-2,-6)[c]{\tiny $1$}}
\put(25,0){\circle*{1,3}}\put(25,0){\makebox(-2,-6)[c]{\tiny $k$}}
\put(35,0){\circle*{1,3}}
\put(70,0){\circle*{1,3}}\put(70,0){\makebox(-2,-6)[c]{\tiny $n$}}
\qbezier(10,0)(32,10)(55,0)\put(10,0){\makebox(0,-6)[c]{\tiny $i$}}
\put(55,0){\makebox(2,-6)[c]{\tiny $\sigma(i)$}}
\end{picture}
}&\longrightarrow&\hspace{1cm}{\setlength{\unitlength}{0.6mm}
\begin{picture}(80,15)(0,-5)
\put(0,0){\line(1,0){35}}\put(45,0){\line(1,0){25}}
\put(0,0){\circle*{1,3}}\put(0,0){\makebox(-2,-6)[c]{\tiny $1$}}
\put(35,0){\circle*{1,3}}\put(35,0){\makebox(2,6)[c]{\tiny
$n\!-\!k$}} \put(45,0){\circle*{1,3}}
\put(70,0){\circle*{1,3}}%\put(70,0){\makebox(-2,-6)[c]{\tiny $n$}}
\qbezier(15,0)(37,-10)(60,0)\put(15,0){\makebox(0,6)[c]{\tiny
$\sigma(i)\!-\!k$}} \put(55,0){\makebox(6,6)[c]{\tiny
$n\!-k\!+\!i$}}
\end{picture}
}\\
%%%%%%%%%%%%%%%%%%%%%%%%%%%%%%%%%%%%%%%%%%%%%%%%%%%%%%%%%type 3
&\hspace{1cm}{\setlength{\unitlength}{0.6mm}
\begin{picture}(80,15)(0,-5)
\put(0,0){\line(1,0){25}}\put(35,0){\line(1,0){35}}
\put(0,0){\circle*{1,3}}\put(0,0){\makebox(-2,-6)[c]{\tiny $1$}}
\put(25,0){\circle*{1,3}}\put(25,0){\makebox(-2,6)[c]{\tiny $k$}}
\put(35,0){\circle*{1,3}}
\put(70,0){\circle*{1,3}}\put(70,0){\makebox(-2,-6)[c]{\tiny $n$}}
\qbezier(10,0)(32,-10)(55,0)\put(10,0){\makebox(0,6)[c]{\tiny
$\sigma(i)$}} \put(55,0){\makebox(2,6)[c]{\tiny $i$}}
\end{picture}
}&\longrightarrow&\hspace{1cm}{\setlength{\unitlength}{0.6mm}
\begin{picture}(80,15)(0,-5)
\put(0,0){\line(1,0){35}}\put(45,0){\line(1,0){25}}
\put(0,0){\circle*{1,3}}\put(0,0){\makebox(-2,-6)[c]{\tiny $1$}}
\put(35,0){\circle*{1,3}}\put(35,0){\makebox(2,-6)[c]{\tiny
$n\!-\!k$}} \put(45,0){\circle*{1,3}}
\put(70,0){\circle*{1,3}}%\put(70,0){\makebox(-2,-6)[c]{\tiny $n$}}
\qbezier(15,0)(37,10)(60,0)\put(15,0){\makebox(0,-6)[c]{\tiny
$i\!-\!k$}} \put(55,0){\makebox(6,-6)[c]{\tiny
$n\!-k\!+\!\sigma(i)$}}
\end{picture}
}\\[0.5cm]
\hline

\end{array}
$$
\caption{The mapping $\Phi_k: \sigma\mapsto \sigma'$.}
\end{table}

For  example, consider the permutation $\sigma\in{^3\S_{15}}$,
whose diagram is given below.\\
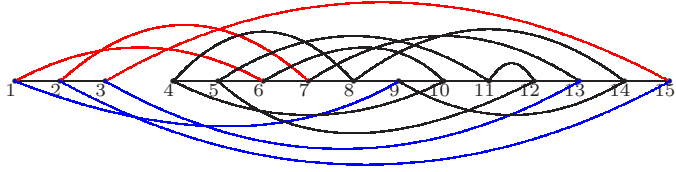
\begin{center}
{\setlength{\unitlength}{0.6mm}
\begin{picture}(180,35)(0,40)
\put(-5,60){\line(1,0){20}} \put(30,60){\line(1,0){110}}
%GRAPHES DE SIGMA
\put(-5,60){\circle*{1}}\put(-5,60){\makebox(-2,-4)[c]{\tiny 1}}
\put(5,60){\circle*{1}}\put(5,60){\makebox(-2,-4)[c]{\tiny 2}}
\put(15,60){\circle*{1}}\put(15,60){\makebox(-2,-4)[c]{\tiny 3}}
\put(30,60){\circle*{1}}\put(30,60){\makebox(-2,-4)[c]{\tiny 4}}
\put(40,60){\circle*{1}}\put(40,60){\makebox(-2,-4)[c]{\tiny 5}}
\put(50,60){\circle*{1}}\put(50,60){\makebox(-2,-4)[c]{\tiny 6}}
\put(60,60){\circle*{1}}\put(60,60){\makebox(-2,-4)[c]{\tiny 7}}
\put(70,60){\circle*{1}}\put(70,60){\makebox(-2,-4)[c]{\tiny 8}}
\put(80,60){\circle*{1}}\put(80,60){\makebox(-2,-4)[c]{\tiny 9}}
\put(90,60){\circle*{1}}\put(90,60){\makebox(-2,-4)[c]{\tiny 10}}
\put(100,60){\circle*{1}}\put(100,60){\makebox(-2,-4)[c]{\tiny
11}}\put(110,60){\circle*{1}}\put(110,60){\makebox(-2,-4)[c]{\tiny
12}}\put(120,60){\circle*{1}}\put(120,60){\makebox(-2,-4)[c]{\tiny
13}}\put(130,60){\circle*{1}}\put(130,60){\makebox(-2,-4)[c]{\tiny
14}}\put(140,60){\circle*{1}}\put(140,60){\makebox(-2,-4)[c]{\tiny
15}}
%EXCEDANCES DE SIGMA
\red{ \qbezier(-5,60)(22.5,75)(50,60)
\qbezier(5,60)(32.5,85)(60,60)\qbezier(15,60)(75,95)(140,60)}
\qbezier(30,60)(50,82)(70,60) \qbezier(40,60)(70,80)(100,60)
\qbezier(50,60)(75,75)(90,60)\qbezier(60,60)(90,80)(120,60)
\qbezier(70,60)(100,83)(130,60) \qbezier(100,60)(105,68)(110,60)
%NONEXCEDANCES DE SIGMA
\blue{ \qbezier(-5,60)(46,40)(80,60)
\qbezier(5,60)(72.5,23)(140,60)\qbezier(15,60)(67.5,30)(120,60)}
\qbezier(30,60)(60,45)(90,60) \qbezier(40,60)(65,37)(110,60)
\qbezier(80,60)(105,45)(130,60)
\end{picture}
}
\end{center}
Then the diagram of $\Phi_3(\sigma)$ is given by
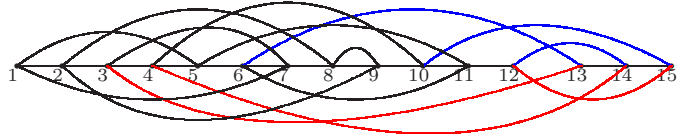
\begin{center}
{\setlength{\unitlength}{0.6mm}
\begin{picture}(180,30)(0,-5)
\put(30,10){\line(1,0){110}}\put(155,10){\line(1,0){20}}
%GRAPHES DE TAU SIGMA
\put(30,10){\circle*{1}}\put(30,10){\makebox(-2,-4)[c]{\tiny 1}}
\put(40,10){\circle*{1}}\put(40,10){\makebox(-2,-4)[c]{\tiny 2}}
\put(50,10){\circle*{1}}\put(50,10){\makebox(-2,-4)[c]{\tiny 3}}
\put(60,10){\circle*{1}}\put(60,10){\makebox(-2,-4)[c]{\tiny 4}}
\put(70,10){\circle*{1}}\put(70,10){\makebox(-2,-4)[c]{\tiny 5}}
\put(80,10){\circle*{1}}\put(80,10){\makebox(-2,-4)[c]{\tiny 6}}
\put(90,10){\circle*{1}}\put(90,10){\makebox(-2,-4)[c]{\tiny 7}}
\put(100,10){\circle*{1}}\put(100,10){\makebox(-2,-4)[c]{\tiny
8}}\put(110,10){\circle*{1}}\put(110,10){\makebox(-2,-4)[c]{\tiny
9}}\put(120,10){\circle*{1}}\put(120,10){\makebox(-2,-4)[c]{\tiny
10}}\put(130,10){\circle*{1}}\put(130,10){\makebox(-2,-4)[c]{\tiny
11}}\put(140,10){\circle*{1}}\put(140,10){\makebox(-2,-4)[c]{\tiny
12}}\put(155,10){\circle*{1}}\put(155,10){\makebox(-2,-4)[c]{\tiny
13}} \put(165,10){\circle*{1}}\put(165,10){\makebox(-2,-4)[c]{\tiny
14}} \put(175,10){\circle*{1}}\put(175,10){\makebox(-2,-4)[c]{\tiny
15}}
%EXCEDANCES DE TAU SIGMA
\blue{\qbezier(80,10)(117.5,35)(155,10)\qbezier(120,10)(142.5,28)(175,10)
\qbezier(140,10)(152.5,20)(165,10)} \qbezier(40,10)(70,35)(100,10)
\qbezier(30,10)(50,25)(70,10)
\qbezier(50,10)(75,27)(90,10)\qbezier(60,10)(90,38)(120,10)
\qbezier(70,10)(100,28)(130,10)\qbezier(100,10)(105,18)(110,10)
%NONEXCEDANCES DE TAU SIGMA
\red{\qbezier(50,10)(80,-15)(155,10)
\qbezier(60,10)(130,-20)(165,10)\qbezier(140,10)(157.5,-5)(175,10)}
\qbezier(30,10)(60,-5)(90,10) \qbezier(40,10)(65,-14)(110,10)
\qbezier(80,10)(105,-5)(130,10)
\end{picture}
}
\end{center}

The main properties of $\Phi_k$ are summarized in following proposition.
\begin{prop}\label{prop:Phi}
For each positive integer $k\in [n]$, the map $\Phi_k: {}^k\S_{n}\to
\S_{n}^k$ is a bijection such that for any $\sigma\in {}^k\S_{n}$
there holds
\begin{align}\label{eq:propPhi}
(\exc, \cro)\Phi_k(\sigma)=(\exc,\cro)\sigma.
\end{align}
\end{prop}

We first show how  to derive Lemma~\ref{lem:1} from
Proposition~\ref{prop:Phi}.  Let $n=n_1+\cdots+n_k$. Then
$\D(n_1,n_2,\ldots, n_k)\subseteq {}^{n_1}\S_{n}$.  By definition of
$\Phi_{n_1}$, for any $\sigma\in {}^{n_1}\S_n$ and $i\in[n-n_1]$
satisfying $\sigma(i+n_1)>n_1$, we have
$i-\Phi_{n_1}(\sigma)(i)=i+n_1-\sigma(i+n_1)$, so
$\Phi_{n_1}(\D(n_1,n_2,\ldots, n_k))\subseteq
\D(n_2,n_3,\ldots,n_k,n_1)$. Since the cardinality of
$\D(n_1,n_2,\ldots, n_k)$ is invariant by permutations of the
$n_i$'s and $\Phi_{n_1}$ is bijective, we have
$\Phi_{n_1}(\D(n_1,n_2,\ldots, n_k))=\D(n_2,n_3,\ldots,n_k,n_1)$.
The result  follows then by applying \eqref{eq:propPhi}.\\

\begin{table}[ht]
$$
\begin{array}{|c|c|c|}
\hline
 i&L_i(\sigma)&R_i(\sigma')\\
 \hline
%%%%%%%%%%%%%%%%%%%%%%%%%%%%%%%%%%%%%%%%%%%%%%%%%%%%%%%%% type 1
1&{\setlength{\unitlength}{0.6mm}
\begin{picture}(80,20)(0,-10)
\put(0,0){\line(1,0){25}}\put(35,0){\line(1,0){35}}
\put(0,0){\circle*{1,3}}\put(0,0){\makebox(-2,-6)[c]{\tiny 1}}
\put(25,0){\circle*{1,3}}\put(25,0){\makebox(-2,-6)[c]{\tiny $k$}}
\put(35,0){\circle*{1,3}}
\put(70,0){\circle*{1,3}}\put(70,0){\makebox(-2,-6)[c]{\tiny $n$}}
\qbezier(40,0)(50,10)(60,0) \qbezier(45,0)(55,10)(65,0)
\end{picture}}
&\hspace{1cm} {\setlength{\unitlength}{0.6mm}
\begin{picture}(80,10)(0,-10)
\put(0,0){\line(1,0){35}}\put(45,0){\line(1,0){25}}
\put(0,0){\circle*{1,3}}\put(0,0){\makebox(-2,-6)[c]{\tiny 1}}
\put(35,0){\circle*{1,3}}\put(35,0){\makebox(-2,-6)[c]{\tiny $n-k$}}
\put(45,0){\circle*{1,3}}
\put(70,0){\circle*{1,3}}\put(70,0){\makebox(-2,-6)[c]{\tiny $n$}}
\qbezier(5,0)(15,10)(25,0) \qbezier(10,0)(20,10)(30,0)
\end{picture}
}\\[-0.5cm]
%%%%%%%%%%%%%%%%%%%%%%%%%%%%%%%%%%%%%
& {\setlength{\unitlength}{0.6mm}
\begin{picture}(80,10)(0,0)
\put(0,0){\line(1,0){25}}\put(35,0){\line(1,0){35}}
\put(0,0){\circle*{1,3}}\put(0,0){\makebox(-2,-6)[c]{\tiny 1}}
\put(25,0){\circle*{1,3}}\put(25,0){\makebox(-2,-6)[c]{\tiny $k$}}
\put(35,0){\circle*{1,3}}
\put(70,0){\circle*{1,3}}\put(70,0){\makebox(-2,-6)[c]{\tiny $n$}}
\qbezier(40,0)(45,10)(50,0) \qbezier(50,0)(55,10)(60,0)
\end{picture}
}&\hspace{1cm}{\setlength{\unitlength}{0.6mm}
\begin{picture}(80,10)(0,0)
\put(0,0){\line(1,0){35}}\put(45,0){\line(1,0){25}}
\put(0,0){\circle*{1,3}}\put(0,0){\makebox(-2,-6)[c]{\tiny 1}}
\put(35,0){\circle*{1,3}}\put(35,0){\makebox(-2,-6)[c]{\tiny $n-k$}}
\put(45,0){\circle*{1,3}}
\put(70,0){\circle*{1,3}}\put(70,0){\makebox(-2,-6)[c]{\tiny $n$}}
\qbezier(5,0)(10,10)(15,0) \qbezier(15,0)(20,10)(25,0)
\end{picture}
}\\
%%%%%%%%%%%%%%%%%%%%%%%%%%%%%%%%%%%%%%%%%%%%%%%%%%%%%%%%%%%
&{\setlength{\unitlength}{0.6mm}
\begin{picture}(80,20)(0,-10)
\put(0,0){\line(1,0){25}}\put(35,0){\line(1,0){35}}
\put(0,0){\circle*{1,3}}\put(0,0){\makebox(-2,-6)[c]{\tiny 1}}
\put(25,0){\circle*{1,3}}\put(25,0){\makebox(-2,-6)[c]{\tiny $k$}}
\put(35,0){\circle*{1,3}}
\put(70,0){\circle*{1,3}}\put(70,0){\makebox(-2,-6)[c]{\tiny $n$}}
\qbezier(40,0)(50,-10)(60,0) \qbezier(45,0)(55,-10)(65,0)
\end{picture}
}&\hspace{1cm}{\setlength{\unitlength}{0.6mm}
\begin{picture}(80,10)(0,-10)
\put(0,0){\line(1,0){35}}\put(45,0){\line(1,0){25}}
\put(0,0){\circle*{1,3}}\put(0,0){\makebox(-2,-6)[c]{\tiny 1}}
\put(35,0){\circle*{1,3}}\put(35,0){\makebox(-2,-6)[c]{\tiny $n-k$}}
\put(45,0){\circle*{1,3}}
\put(70,0){\circle*{1,3}}\put(70,0){\makebox(-2,-6)[c]{\tiny $n$}}
\qbezier(5,0)(15,-10)(25,0) \qbezier(10,0)(20,-10)(30,0)
\end{picture}
}\\
\hline
%%%%%%%%%%%%%%%%%%%%%%%%%%%%%%%%%%%%%%%%%%%%%%%%%%%%%%%%type 2
2&\hspace{0.5cm}{\setlength{\unitlength}{0.6mm}
\begin{picture}(80,20)(0,-10)
\put(0,0){\line(1,0){25}}\put(35,0){\line(1,0){35}}
\put(0,0){\circle*{1,3}}\put(0,0){\makebox(-2,-6)[c]{\tiny 1}}
\put(25,0){\circle*{1,3}}\put(25,0){\makebox(-2,-6)[c]{\tiny $k$}}
\put(35,0){\circle*{1,3}}
\put(70,0){\circle*{1,3}}\put(70,0){\makebox(-2,-6)[c]{\tiny $n$}}
\qbezier(5,0)(30,10)(45,0) \qbezier(20,0)(40,10)(60,0)
\end{picture}
}&\hspace{1cm}{\setlength{\unitlength}{0.6mm}
\begin{picture}(80,20)(0,-10)
\put(0,0){\line(1,0){35}}\put(45,0){\line(1,0){25}}
\put(0,0){\circle*{1,3}}\put(0,0){\makebox(-2,-6)[c]{\tiny $1$}}
\put(35,0){\circle*{1,3}}\put(35,0){\makebox(-2,-6)[c]{\tiny $n-k$}}
\put(45,0){\circle*{1,3}}
\put(70,0){\circle*{1,3}}\put(70,0){\makebox(-2,-6)[c]{\tiny $n$}}
\qbezier(10,0)(40,10)(50,0) \qbezier(25,0)(45,10)(65,0)
\end{picture}
}\\[-0.5cm]
%%%%%%%%%%%%%%%%%%%%%%%%%%%%%%%%%%%%%%%%%%%%%%%%%%%%%%%%%
&\hspace{0.5cm}{\setlength{\unitlength}{0.6mm}
\begin{picture}(80,20)(0,-10)
\put(0,0){\line(1,0){25}}\put(35,0){\line(1,0){35}}
\put(0,0){\circle*{1,3}}\put(0,0){\makebox(-2,-6)[c]{\tiny 1}}
\put(25,0){\circle*{1,3}}\put(25,0){\makebox(-2,6)[c]{\tiny $k$}}
\put(35,0){\circle*{1,3}}
\put(70,0){\circle*{1,3}}\put(70,0){\makebox(-2,-6)[c]{\tiny $n$}}
\qbezier(5,0)(30,-10)(45,0) \qbezier(20,0)(40,-10)(60,0)
\end{picture}
}&\hspace{1cm}{\setlength{\unitlength}{0.6mm}
\begin{picture}(80,20)(0,-10)
\put(0,0){\line(1,0){35}}\put(45,0){\line(1,0){25}}
\put(0,0){\circle*{1,3}}\put(0,0){\makebox(-2,-6)[c]{\tiny 1}}
\put(35,0){\circle*{1,3}}\put(35,0){\makebox(-2,6)[c]{\tiny $n-k$}}
\put(45,0){\circle*{1,3}}
\put(70,0){\circle*{1,3}}\put(70,0){\makebox(-2,-6)[c]{\tiny $n$}}
\qbezier(10,0)(40,-10)(50,0) \qbezier(25,0)(45,-10)(65,0)
\end{picture}
}\\
\hline
%%%%%%%%%%%%%%%%%%%%%%%%%%%%%%%%%%%%%%%%%%%%%%%%%%%%%%%%%type 3
3&{\setlength{\unitlength}{0.6mm}
\begin{picture}(80,20)(0,-10)
\put(0,0){\line(1,0){25}}\put(35,0){\line(1,0){35}}
\put(0,0){\circle*{1,3}}\put(0,0){\makebox(-2,-6)[c]{\tiny 1}}
\put(25,0){\circle*{1,3}}\put(25,0){\makebox(-2,-6)[c]{\tiny $k$}}
\put(35,0){\circle*{1,3}}
\put(70,0){\circle*{1,3}}\put(70,0){\makebox(-2,-6)[c]{\tiny $n$}}
\qbezier(12.5,0)(37.5,10)(52.5,0) \qbezier(45,0)(52.5,10)(60,0)
\end{picture}
}&\hspace{1cm}{\setlength{\unitlength}{0.6mm}
\begin{picture}(80,10)(0,-10)
\put(0,0){\line(1,0){35}}\put(45,0){\line(1,0){25}}
\put(0,0){\circle*{1,3}}\put(0,0){\makebox(-2,-6)[c]{\tiny 1}}
\put(35,0){\circle*{1,3}}\put(35,0){\makebox(-2,-6)[c]{\tiny $n-k$}}
\put(45,0){\circle*{1,3}}
\put(70,0){\circle*{1,3}}\put(70,0){\makebox(-2,-6)[c]{\tiny $n$}}
\qbezier(10,0)(17.5,10)(25,0) \qbezier(17.5,0)(37.5,10)(57.5,0)
\end{picture}
}\\[-0.5cm]
%%%%%%%%%%%%%%%%%%%%%%%%%%%%%%%%%%%%%
& {\setlength{\unitlength}{0.6mm}
\begin{picture}(80,10)(0,0)
\put(0,0){\line(1,0){25}}\put(35,0){\line(1,0){35}}
\put(0,0){\circle*{1,3}}\put(0,0){\makebox(-2,-6)[c]{\tiny 1}}
\put(25,0){\circle*{1,3}}\put(25,0){\makebox(-2,-6)[c]{\tiny $k$}}
\put(35,0){\circle*{1,3}}
\put(70,0){\circle*{1,3}}\put(70,0){\makebox(-2,-6)[c]{\tiny $n$}}
\qbezier(12.5,0)(37.5,10)(52.5,0) \qbezier(52.5,0)(57.5,10)(62.5,0)
\end{picture}
}&\hspace{1cm}{\setlength{\unitlength}{0.6mm}
\begin{picture}(80,10)(0,0)
\put(0,0){\line(1,0){35}}\put(45,0){\line(1,0){25}}
\put(0,0){\circle*{1,3}}\put(0,0){\makebox(-2,-6)[c]{\tiny 1}}
\put(35,0){\circle*{1,3}}\put(35,0){\makebox(-2,-6)[c]{\tiny $n-k$}}
\put(45,0){\circle*{1,3}}
\put(70,0){\circle*{1,3}}\put(70,0){\makebox(-2,-6)[c]{\tiny $n$}}
\qbezier(10,0)(17.5,10)(25,0) \qbezier(25,0)(40,10)(57.5,0)
\end{picture}
}\\[-0.1cm]
%%%%%%%%%%%%%%%%%%%%%%%%%%%%%%%%%%%%%%%%%%%%%%%%%
&{\setlength{\unitlength}{0.6mm}
\begin{picture}(80,20)(0,-10)
\put(0,0){\line(1,0){25}}\put(35,0){\line(1,0){35}}
\put(0,0){\circle*{1,3}}\put(0,0){\makebox(-2,-6)[c]{\tiny 1}}
\put(25,0){\circle*{1,3}}\put(25,0){\makebox(-2,6)[c]{\tiny $k$}}
\put(35,0){\circle*{1,3}}
\put(70,0){\circle*{1,3}}\put(70,0){\makebox(-2,-6)[c]{\tiny $n$}}
\qbezier(12.5,0)(37.5,-10)(52.5,0) \qbezier(45,0)(52.5,-10)(60,0)
\end{picture}
}&\hspace{1cm}{\setlength{\unitlength}{0.6mm}
\begin{picture}(80,10)(0,-10)
\put(0,0){\line(1,0){35}}\put(45,0){\line(1,0){25}}
\put(0,0){\circle*{1,3}}\put(0,0){\makebox(-2,-6)[c]{\tiny 1}}
\put(35,0){\circle*{1,3}}\put(35,0){\makebox(-2,6)[c]{\tiny $n-k$}}
\put(45,0){\circle*{1,3}}
\put(70,0){\circle*{1,3}}\put(70,0){\makebox(-2,-6)[c]{\tiny $n$}}
\qbezier(10,0)(17.5,-10)(25,0) \qbezier(17.5,0)(37.5,-10)(57.5,0)
\end{picture}
}\\
\hline
\end{array}
$$
\caption{Forms of crossings in $L_i(\sigma)$ and
$R_i(\sigma')$.}\label{tab:forme-cr-1}
\end{table}

\noindent{\bf Proof of Proposition~\ref{prop:Phi}.}  It is easy to see that
$\Phi_k$ is a bijection. Let $\sigma\in{}^k\S_{n}$ and
$\sigma'=\Phi_k(\sigma)$. The equality $\exc(\sigma')=\exc(\sigma)$
follows directly from  the definition of $\Phi_k$. It then remains
to prove that $\cro(\sigma')=\cro(\sigma)$.  We first decompose the
crossings of $\sigma$ and $\sigma'$ into three subsets. Set
\begin{align*}
L_1(\sigma)&= \{(i,j)\;|\;k<i<j\leq \sigma(i)<\sigma(j)\quad\text{or}\quad i>j>\sigma(i)>\sigma(j)>k\},\\
L_2(\sigma)&=\{(i,j)\;|\;i<j\leq k< \sigma(i)<\sigma(j) \quad\text{or}\quad i>j>k\geq\sigma(i)>\sigma(j)\},\\
L_3(\sigma)&=\{(i,j)\;|\;i\leq k<j\leq \sigma(i)<\sigma(j)
\quad\text{or}\quad i>j>\sigma(i)>k\geq \sigma(j)\},
\end{align*}
\vspace{-0.2cm} and
\begin{align*}
R_1(\sigma')&=\{(i,j)\;|\;i<j\leq\sigma'(i)<\sigma'(j)\leq n-k \quad\text{or}\quad n-k\geq i>j>\sigma'(i)>\sigma'(j)\},\\
R_2(\sigma')&= \{(i,j)\;|\; i<j\leq n-k<\sigma'(i)<\sigma'(j)\quad\text{or}\quad i>j>n-k\geq\sigma'(i)>\sigma'(j)\},\\
R_3(\sigma')&=\{(i,j)\;|\;i<j\leq {\sigma}'(i)\leq
n-k<{\sigma}'(j)\quad\text{or}\quad i>n-k\geq j>\sigma(i)>
\sigma(j)\}.
\end{align*}
The crossings in $L_i$'s and $R_i$'s are illustrated   in
Table~\ref{tab:forme-cr-1}. Clearly, we have
$\cro(\sigma)=\sum_{i=1}^3|L_i(\sigma)|$ and
$\cro(\sigma')=\sum_{i=1}^3|R_i(\sigma')|$ since
$\sigma\in{}^k\S_{n}$ and $\sigma'\in\S^k_{n}$.

\begin{table}[ht]
$$
\begin{array}{|ccc|}
\hline
 \sigma&\longrightarrow&\sigma'\\
 \hline
%%%%%%%%%%%%%%%%%%%%%%%%%%%%%%%%%%%%%%%%%%%%%%%%%%%%%%%%%%%%%%%%%%%%%%%%%%%%%%%%%%%%%%%%%%%%%%%% type 1
{\setlength{\unitlength}{0.6mm}
\begin{picture}(90,20)(0,-10)
\put(0,0){\line(1,0){25}}\put(35,0){\line(1,0){55}}
%\put(0,0){\circle*{1,3}}\put(0,0){\makebox(-2,-6)[c]{\tiny 1}}
\put(25,0){\circle*{1,3}}\put(25,0){\makebox(-2,-6)[c]{\tiny $k$}}
\put(40,0){\makebox(-2,-6)[c]{\tiny$i$}}\put(55,0){\makebox(-2,-6)[c]{\tiny$j$}}
\put(70,0){\makebox(-2,-6)[c]{\tiny$\sigma_i$}}\put(83,0){\makebox(-2,-6)[c]{\tiny$\sigma_j$}}
%\put(35,0){\circle*{1,3}}
%\put(90,0){\circle*{1,3}}\put(92,0){\makebox(-2,-6)[c]{\tiny $n$}}
\qbezier(40,0)(55,10)(70,0) \qbezier(55,0)(70,10)(83,0)
\end{picture}}
&& {\setlength{\unitlength}{0.6mm}
\begin{picture}(90,20)(0,-10)
\put(0,0){\line(1,0){55}}\put(65,0){\line(1,0){25}}
%\put(0,0){\circle*{1,3}}\put(0,0){\makebox(-2,-6)[c]{\tiny 1}}
\put(5,0){\makebox(-2,-6)[c]{\tiny$i$-$k$}}\put(35,0){\makebox(-2,-6)[c]{\tiny$\sigma_i$-$k$}}
\put(20,0){\makebox(-2,-6)[c]{\tiny$j$-$k$}}\put(48,0){\makebox(-2,-6)[c]{\tiny$\sigma_j$-$k$}}
\put(55,0){\circle*{1,3}}\put(58,0){\makebox(-2,-6)[c]{\tiny$n$-$k$}}
%\put(90,0){\circle*{1,3}}\put(90,0){\makebox(-2,-6)[c]{\tiny $n$}}
\qbezier(5,0)(20,10)(35,0) \qbezier(20,0)(34,10)(48,0)
\end{picture}}\\[-0.5cm]
%%%%%%%%%%%%%%%%%%%%%%%%%%%%%%%%%%%%% b
{\setlength{\unitlength}{0.6mm}
\begin{picture}(90,20)(0,-10)
\put(0,0){\line(1,0){25}}\put(35,0){\line(1,0){55}}
%\put(0,0){\circle*{1,3}}\put(0,0){\makebox(-2,-6)[c]{\tiny 1}}
\put(25,0){\circle*{1,3}}\put(25,0){\makebox(-2,-6)[c]{\tiny $k$}}
\put(45,0){\makebox(-2,-6)[c]{\tiny$i$}}\put(65,0){\makebox(-2,-6)[c]{\tiny$j$}}
\put(85,0){\makebox(-2,-6)[c]{\tiny$\sigma_j$}}
%\put(35,0){\circle*{1,3}}
%\put(90,0){\circle*{1,3}}\put(92,0){\makebox(-2,-6)[c]{\tiny $n$}}
\qbezier(45,0)(55,10)(65,0) \qbezier(65,0)(75,10)(85,0)
\end{picture}}
&&  {\setlength{\unitlength}{0.6mm}
\begin{picture}(90,20)(0,-10)
\put(0,0){\line(1,0){55}}\put(65,0){\line(1,0){25}}
%\put(0,0){\circle*{1,3}}\put(0,0){\makebox(-2,-6)[c]{\tiny 1}}
\put(6,0){\makebox(-2,-6)[c]{\tiny$i$-$k$}}\put(25,0){\makebox(-2,-6)[c]{\tiny$j$-$k$}}
\put(42,0){\makebox(-2,-6)[c]{\tiny$\sigma_j$-$k$}}
\put(55,0){\circle*{1,3}}\put(58,0){\makebox(-2,-6)[c]{\tiny$n$-$k$}}
%\put(90,0){\circle*{1,3}}\put(90,0){\makebox(-2,-6)[c]{\tiny $n$}}
\qbezier(6,0)(15,10)(25,0) \qbezier(25,0)(33,10)(42,0)
\end{picture}}\\[-0.5cm]
%%%%%%%%%%%%%%%%%%%%%%%%%%%%%%%%%%%%%%%%%%%%%%%%%%%%%%%%%% c
{\setlength{\unitlength}{0.6mm}
\begin{picture}(90,20)(0,-10)
\put(0,0){\line(1,0){25}}\put(35,0){\line(1,0){55}}
%\put(0,0){\circle*{1,3}}\put(0,0){\makebox(-2,-6)[c]{\tiny 1}}
\put(25,0){\circle*{1,3}}\put(25,0){\makebox(2,6)[c]{\tiny $k$}}
\put(40,0){\makebox(2,6)[c]{\tiny$\sigma_j$}}\put(55,0){\makebox(2,6)[c]{\tiny$\sigma_i$}}
\put(70,0){\makebox(2,6)[c]{\tiny$j$}}\put(83,0){\makebox(2,6)[c]{\tiny$i$}}
%\put(35,0){\circle*{1,3}}
%\put(90,0){\circle*{1,3}}\put(92,0){\makebox(-2,-6)[c]{\tiny $n$}}
\qbezier(40,0)(55,-10)(70,0) \qbezier(55,0)(70,-10)(83,0)
\end{picture}}
&&  {\setlength{\unitlength}{0.6mm}
\begin{picture}(90,20)(0,-10)
\put(0,0){\line(1,0){55}}\put(65,0){\line(1,0){25}}
%\put(0,0){\circle*{1,3}}\put(0,0){\makebox(-2,-6)[c]{\tiny 1}}
\put(5,0){\makebox(2,6)[c]{\tiny$\sigma_j$-$k$}}\put(35,0){\makebox(2,6)[c]{\tiny$j$-$k$}}
\put(20,0){\makebox(2,6)[c]{\tiny$\sigma_i$-$k$}}\put(48,0){\makebox(2,6)[c]{\tiny$i$-$k$}}
\put(55,0){\circle*{1,3}}\put(58,0){\makebox(2,6)[c]{\tiny$n$-$k$}}
%\put(90,0){\circle*{1,3}}\put(90,0){\makebox(-2,-6)[c]{\tiny $n$}}
\qbezier(5,0)(20,-10)(35,0) \qbezier(20,0)(34,-10)(48,0)
\end{picture}}\\
\hline
%%%%%%%%%%%%%%%%%%%%%%%%%%%%%%%%%%%%%%%%%%%%%%%%%%%%%%%%%%%%%%%%%%%%%%%%%%%%%%%%%%%%%%%%%%%%%%%%%%%%%%%%%%%%%%%%%%%%%%%%%%type 2
{\setlength{\unitlength}{0.6mm}
\begin{picture}(90,20)(0,-10)
\put(0,0){\line(1,0){40}}\put(50,0){\line(1,0){40}}
%\put(0,0){\circle*{1,3}}\put(0,0){\makebox(-2,-6)[c]{\tiny 1}}
\put(40,0){\circle*{1,3}}\put(40,0){\makebox(-2,-6)[c]{\tiny $k$}}
\put(10,0){\makebox(-2,-6)[c]{\tiny$i$}}\put(30,0){\makebox(-2,-6)[c]{\tiny$j$}}
\put(60,0){\makebox(-2,-6)[c]{\tiny$\sigma_i$}}\put(80,0){\makebox(-2,-6)[c]{\tiny$\sigma_j$}}
%\put(35,0){\circle*{1,3}}
%\put(90,0){\circle*{1,3}}\put(92,0){\makebox(-2,-6)[c]{\tiny $n$}}
\qbezier(10,0)(35,12)(60,0) \qbezier(30,0)(55,12)(80,0)
\end{picture}}
&&  {\setlength{\unitlength}{0.6mm}
\begin{picture}(90,20)(0,-10)
\put(0,0){\line(1,0){40}}\put(50,0){\line(1,0){40}}
%\put(0,0){\circle*{1,3}}\put(0,0){\makebox(-2,-6)[c]{\tiny 1}}
\put(40,0){\circle*{1,3}}\put(40,0){\makebox(2,6)[c]{\tiny $n$-$k$}}
\put(10,0){\makebox(2,6)[c]{\tiny$\sigma_i$-$k$}}\put(30,0){\makebox(2,6)[c]{\tiny$\sigma_j$-$k$}}
\put(60,0){\makebox(2,6)[c]{\tiny$n$-$k$+$i$}}\put(80,0){\makebox(2,6)[c]{\tiny$n$-$k$+$j$}}
%\put(35,0){\circle*{1,3}}
%\put(90,0){\circle*{1,3}}\put(92,0){\makebox(-2,-6)[c]{\tiny $n$}}
\qbezier(10,0)(35,-12)(60,0) \qbezier(30,0)(55,-12)(80,0)
\end{picture}}\\[-0.3cm]
%%%%%%%%%%%%%%%%%%%%%%%%%%%%%%%%%%%%% b
 {\setlength{\unitlength}{0.6mm}
\begin{picture}(90,20)(0,-10)
\put(0,0){\line(1,0){40}}\put(50,0){\line(1,0){40}}
%\put(0,0){\circle*{1,3}}\put(0,0){\makebox(-2,-6)[c]{\tiny 1}}
\put(40,0){\circle*{1,3}}\put(40,0){\makebox(2,6)[c]{\tiny $k$}}
\put(10,0){\makebox(2,6)[c]{\tiny$\sigma_j$}}\put(30,0){\makebox(2,6)[c]{\tiny$\sigma_i$}}
\put(60,0){\makebox(2,6)[c]{\tiny$j$}}\put(80,0){\makebox(2,6)[c]{\tiny$i$}}
%\put(35,0){\circle*{1,3}}
%\put(90,0){\circle*{1,3}}\put(92,0){\makebox(-2,-6)[c]{\tiny $n$}}
\qbezier(10,0)(35,-12)(60,0) \qbezier(30,0)(55,-12)(80,0)
\end{picture}}
&&  {\setlength{\unitlength}{0.6mm}
\begin{picture}(90,20)(0,-10)
\put(0,0){\line(1,0){40}}\put(50,0){\line(1,0){40}}
%\put(0,0){\circle*{1,3}}\put(0,0){\makebox(-2,-6)[c]{\tiny 1}}
\put(40,0){\circle*{1,3}}\put(40,0){\makebox(-2,-6)[c]{\tiny
$n$-$k$}}
\put(10,0){\makebox(-2,-6)[c]{\tiny$j$-$k$}}\put(30,0){\makebox(-2,-6)[c]{\tiny$i$-$k$}}
\put(60,0){\makebox(-2,-6)[c]{\tiny$n$-$k$+$\sigma_j$}}\put(80,0){\makebox(-2,-6)[c]{\tiny$n$-$k$+$\sigma_i$}}
%\put(35,0){\circle*{1,3}}
%\put(90,0){\circle*{1,3}}\put(92,0){\makebox(-2,-6)[c]{\tiny $n$}}
\qbezier(10,0)(35,12)(60,0) \qbezier(30,0)(55,12)(80,0)
\end{picture}}\\
\hline
%%%%%%%%%%%%%%%%%%%%%%%%%%%%%%%%%%%%%%%%%%%%%%%%%%%%%%%%%%%%%%%%%%%%%%%%%%%%%%%%%%%%%%%%%%%%%%%%%%%%%%%%%%%%%%%%%%%%%%%% type 3
{\setlength{\unitlength}{0.6mm}
\begin{picture}(90,20)(0,-10)
\put(0,0){\line(1,0){25}}\put(35,0){\line(1,0){55}}
%\put(0,0){\circle*{1,3}}\put(0,0){\makebox(-2,-6)[c]{\tiny 1}}
\put(25,0){\circle*{1,3}}\put(25,0){\makebox(-2,-6)[c]{\tiny $k$}}
\put(15,0){\makebox(-2,-6)[c]{\tiny$i$}}\put(50,0){\makebox(-2,-6)[c]{\tiny$j$}}
\put(65,0){\makebox(-2,-6)[c]{\tiny$\sigma_i$}}\put(75,0){\makebox(-2,-6)[c]{\tiny$\sigma_j$}}
%\put(35,0){\circle*{1,3}}
%\put(90,0){\circle*{1,3}}\put(92,0){\makebox(-2,-6)[c]{\tiny $n$}}
\qbezier(15,0)(40,12)(65,0) \qbezier(50,0)(62,10)(75,0)
\end{picture}}
&&  {\setlength{\unitlength}{0.6mm}
\begin{picture}(90,20)(0,-10)
\put(0,0){\line(1,0){55}}\put(65,0){\line(1,0){25}}
%\put(0,0){\circle*{1,3}}\put(0,0){\makebox(-2,-6)[c]{\tiny 1}}
\put(10,0){\makebox(-2,-6)[c]{\tiny$j$-$k$}}\put(25,0){\makebox(-2,4)[c]{\tiny$\sigma_i$-$k$}}
\put(40,0){\makebox(5,-6)[c]{\tiny$\sigma_j$-$k$}}\put(78,0){\makebox(-2,6)[c]{\tiny$n$-$k$+$i$}}
\put(55,0){\circle*{1,3}}\put(58,0){\makebox(-2,6)[c]{\tiny$n$-$k$}}
%\put(90,0){\circle*{1,3}}\put(90,0){\makebox(-2,-6)[c]{\tiny $n$}}
\qbezier(10,0)(25,12)(40,0) \qbezier(25,0)(45,-13)(78,0)
\end{picture}}\\[-0.5cm]
%%%%%%%%%%%%%%%%%%%%%%%%%%%%%%%%%%%%% b
{\setlength{\unitlength}{0.6mm}
\begin{picture}(90,20)(0,-10)
\put(0,0){\line(1,0){25}}\put(35,0){\line(1,0){55}}
%\put(0,0){\circle*{1,3}}\put(0,0){\makebox(-2,-6)[c]{\tiny 1}}
\put(25,0){\circle*{1,3}}\put(25,0){\makebox(-2,-6)[c]{\tiny $k$}}
\put(15,0){\makebox(-2,-6)[c]{\tiny$i$}}\put(55,0){\makebox(-2,-6)[c]{\tiny$j$}}
\put(75,0){\makebox(-2,-6)[c]{\tiny$\sigma_j$}}
%\put(35,0){\circle*{1,3}}
%\put(90,0){\circle*{1,3}}\put(92,0){\makebox(-2,-6)[c]{\tiny $n$}}
\qbezier(15,0)(35,12)(55,0) \qbezier(55,0)(65,10)(75,0)
\end{picture}}
&&  {\setlength{\unitlength}{0.6mm}
\begin{picture}(90,20)(0,-10)
\put(0,0){\line(1,0){55}}\put(65,0){\line(1,0){25}}
%\put(0,0){\circle*{1,3}}\put(0,0){\makebox(-2,-6)[c]{\tiny 1}}
\put(10,0){\makebox(-2,-6)[c]{\tiny$j$-$k$}}\put(35,0){\makebox(2,-6)[c]{\tiny$\sigma_j$-$k$}}
\put(80,0){\makebox(-2,6)[c]{\tiny$n$-$k$+$i$}}
\put(55,0){\circle*{1,3}}\put(58,0){\makebox(-2,6)[c]{\tiny$n$-$k$}}
%\put(90,0){\circle*{1,3}}\put(90,0){\makebox(-2,-6)[c]{\tiny $n$}}
\qbezier(12,0)(24,10)(35,0) \qbezier(12,0)(45,-12)(78,0)
\end{picture}}\\[-0.3cm]
%%%%%%%%%%%%%%%%%%%%%%%%%%%%%%%%%%%%%%%%%%%%%%%%%%%%%%%%%% c
{\setlength{\unitlength}{0.6mm}
\begin{picture}(90,20)(0,-10)
\put(0,0){\line(1,0){25}}\put(35,0){\line(1,0){55}}
%\put(0,0){\circle*{1,3}}\put(0,0){\makebox(-2,-6)[c]{\tiny 1}}
\put(25,0){\circle*{1,3}}\put(25,0){\makebox(-2,6)[c]{\tiny $k$}}
\put(15,0){\makebox(-2,6)[c]{\tiny$\sigma_j$}}\put(50,0){\makebox(-2,6)[c]{\tiny$\sigma_i$}}
\put(65,0){\makebox(-2,6)[c]{\tiny$j$}}\put(75,0){\makebox(-2,6)[c]{\tiny$i$}}
%\put(35,0){\circle*{1,3}}
%\put(90,0){\circle*{1,3}}\put(92,0){\makebox(-2,-6)[c]{\tiny $n$}}
\qbezier(15,0)(40,-12)(65,0) \qbezier(50,0)(62,-10)(75,0)
\end{picture}}
&&  {\setlength{\unitlength}{0.6mm}
\begin{picture}(90,20)(0,-10)
\put(0,0){\line(1,0){55}}\put(65,0){\line(1,0){25}}
%\put(0,0){\circle*{1,3}}\put(0,0){\makebox(-2,-6)[c]{\tiny 1}}
\put(10,0){\makebox(-2,6)[c]{\tiny$\sigma_i$-$k$}}\put(25,0){\makebox(-2,-6)[c]{\tiny$j$-$k$}}
\put(40,0){\makebox(5,6)[c]{\tiny$i$-$k$}}\put(78,0){\makebox(-2,-6)[c]{\tiny$n$-$k$+$\sigma_j$}}
\put(55,0){\circle*{1,3}}\put(58,0){\makebox(-2,-6)[c]{\tiny$n$-$k$}}
%\put(90,0){\circle*{1,3}}\put(90,0){\makebox(-2,-6)[c]{\tiny $n$}}
\qbezier(10,0)(25,-12)(40,0) \qbezier(25,0)(45,13)(78,0)
\end{picture}}\\
\hline
%%%%%%%%%%%%%%%%%%%%%%%%%%%%%%%%%%%%%%%%%%%%%%%%%%%%%%%%%%%%%%%%%%%%%%%%%%%%%%%%%%%%%%%%%%%%%%%%%%%%%%%%%%%%%%%%%% type4
{\setlength{\unitlength}{0.6mm}
\begin{picture}(90,20)(0,-10)
\put(0,0){\line(1,0){25}}\put(35,0){\line(1,0){55}}
%\put(0,0){\circle*{1,3}}\put(0,0){\makebox(-2,-6)[c]{\tiny 1}}
\put(25,0){\circle*{1,3}}\put(25,0){\makebox(-2,6)[c]{\tiny $k$}}
\put(15,0){\makebox(-2,6)[c]{\tiny$\sigma_i$}}\put(50,0){\makebox(-2,-6)[c]{\tiny$j$}}
\put(65,0){\makebox(-2,4)[c]{\tiny$i$}}\put(75,0){\makebox(-2,-6)[c]{\tiny$\sigma_j$}}
%\put(35,0){\circle*{1,3}}
%\put(90,0){\circle*{1,3}}\put(92,0){\makebox(-2,-6)[c]{\tiny $n$}}
\qbezier(15,0)(40,-12)(65,0) \qbezier(50,0)(62,12)(75,0)
\end{picture}}
&&  {\setlength{\unitlength}{0.6mm}
\begin{picture}(90,20)(0,-10)
\put(0,0){\line(1,0){55}}\put(65,0){\line(1,0){25}}
%\put(0,0){\circle*{1,3}}\put(0,0){\makebox(-2,-6)[c]{\tiny 1}}
\put(10,0){\makebox(-2,-6)[c]{\tiny$j$-$k$}}\put(25,0){\makebox(-2,-6)[c]{\tiny$i$-$k$}}
\put(40,0){\makebox(5,-6)[c]{\tiny$\sigma_j$-$k$}}\put(78,0){\makebox(-2,-6)[c]{\tiny$n$-$k$+$\sigma_i$}}
\put(55,0){\circle*{1,3}}\put(58,0){\makebox(-2,-6)[c]{\tiny$n$-$k$}}
%\put(90,0){\circle*{1,3}}\put(90,0){\makebox(-2,-6)[c]{\tiny $n$}}
\qbezier(10,0)(25,12)(40,0) \qbezier(25,0)(45,13)(78,0)
\end{picture}}\\[-0.5cm]
%%%%%%%%%%%%%%%%%%%%%%%%%%%%%%%%%%%%% b
{\setlength{\unitlength}{0.6mm}
\begin{picture}(90,20)(0,-10)
\put(0,0){\line(1,0){25}}\put(35,0){\line(1,0){55}}
%\put(0,0){\circle*{1,3}}\put(0,0){\makebox(-2,-6)[c]{\tiny 1}}
\put(25,0){\circle*{1,3}}\put(25,0){\makebox(-2,6)[c]{\tiny $k$}}
\put(15,0){\makebox(-2,-6)[c]{\tiny$\sigma_i$}}\put(55,0){\makebox(-2,-6)[c]{\tiny$j$}}
\put(75,0){\makebox(5,-6)[c]{\tiny$i$}}
%\put(35,0){\circle*{1,3}}
%\put(90,0){\circle*{1,3}}\put(92,0){\makebox(-2,-6)[c]{\tiny $n$}}
\qbezier(15,0)(45,-13)(75,0) \qbezier(55,0)(65,10)(75,0)
\end{picture}}
&&  {\setlength{\unitlength}{0.6mm}
\begin{picture}(90,20)(0,-10)
\put(0,0){\line(1,0){55}}\put(65,0){\line(1,0){25}}
%\put(0,0){\circle*{1,3}}\put(0,0){\makebox(-2,-6)[c]{\tiny 1}}
\put(10,0){\makebox(-2,-6)[c]{\tiny$j$-$k$}}\put(35,0){\makebox(2,-6)[c]{\tiny$i$-$k$}}
\put(80,0){\makebox(-2,-6)[c]{\tiny$n$-$k$+$\sigma_i$}}
\put(55,0){\circle*{1,3}}\put(58,0){\makebox(-2,-6)[c]{\tiny$n$-$k$}}
%\put(90,0){\circle*{1,3}}\put(90,0){\makebox(-2,-6)[c]{\tiny $n$}}
\qbezier(12,0)(24,10)(35,0) \qbezier(35,0)(55,12)(78,0)
\end{picture}}\\[-0.3cm]
%%%%%%%%%%%%%%%%%%%%%%%%%%%%%%%%%%%%%%%%%%%%%%%%%%%%%%%%%% c
{\setlength{\unitlength}{0.6mm}
\begin{picture}(90,20)(0,-10)
\put(0,0){\line(1,0){25}}\put(35,0){\line(1,0){55}}
%\put(0,0){\circle*{1,3}}\put(0,0){\makebox(-2,-6)[c]{\tiny 1}}
\put(25,0){\circle*{1,3}}\put(25,0){\makebox(-2,-6)[c]{\tiny $k$}}
\put(15,0){\makebox(-2,-6)[c]{\tiny$i$}}\put(50,0){\makebox(-2,4)[c]{\tiny$\sigma_j$}}
\put(65,0){\makebox(-2,-6)[c]{\tiny$\sigma_i$}}\put(75,0){\makebox(-2,6)[c]{\tiny$j$}}
%\put(35,0){\circle*{1,3}}
%\put(90,0){\circle*{1,3}}\put(92,0){\makebox(-2,-6)[c]{\tiny $n$}}
\qbezier(15,0)(40,12)(65,0) \qbezier(50,0)(62,-12)(75,0)
\end{picture}}
&&  {\setlength{\unitlength}{0.6mm}
\begin{picture}(90,20)(0,-10)
\put(0,0){\line(1,0){55}}\put(65,0){\line(1,0){25}}
%\put(0,0){\circle*{1,3}}\put(0,0){\makebox(-2,-6)[c]{\tiny 1}}
\put(10,0){\makebox(-2,6)[c]{\tiny$\sigma_j$-$k$}}\put(25,0){\makebox(-2,6)[c]{\tiny$\sigma_i$-$k$}}
\put(40,0){\makebox(5,6)[c]{\tiny$j$-$k$}}\put(78,0){\makebox(-2,6)[c]{\tiny$n$-$k$+$i$}}
\put(55,0){\circle*{1,3}}\put(58,0){\makebox(-2,6)[c]{\tiny$n$-$k$}}
%\put(90,0){\circle*{1,3}}\put(90,0){\makebox(-2,-6)[c]{\tiny $n$}}
\qbezier(10,0)(25,-12)(40,0) \qbezier(25,0)(45,-13)(78,0)
\end{picture}}\\
\hline
\end{array}
$$
\caption{Effects of the mapping $\Phi_k$ on the crossings of
$\sigma$ and $\sigma'$.}\label{tab:effet Psi sur cr}
\end{table}

By  the definition of $\Phi_k$,  it is readily seen (see Row~1 in
Table~\ref{tab:effet Psi sur cr}) that $ (i,j)\in L_1(\sigma)$ if
and only if $(i-k,j-k)\in R_1(\sigma')$, and thus
$|L_1(\sigma)|=|R_1(\sigma')|$. Similarly, we have (see Row 2 in
Table~\ref{tab:effet Psi sur cr}) that
$|L_2(\sigma)|=|R_2(\sigma')|$. It then remains to prove that
$|L_3(\sigma)|=|R_3(\sigma')|$. Let
\begin{align*}
L_4(\sigma)&=\{(i,j)\;|\;\sigma(i)\leq k<j<i\leq \sigma(j)
\quad\text{or}\quad i\leq k<\sigma(j)<\sigma(i)<j\}.
\end{align*}
Then it is not difficult to show (see Row~4 of Table~\ref{tab:effet Psi sur cr}) that $|R_3(\sigma')|=|L_4(\sigma)|$.
The result will thus follow from the following Lemma.
%%%%%%%%%%%%%%%%%%
\begin{lem} For all $\sigma\in {^k\S_{n}}$ we have  $|L_3(\sigma)|=|L_4(\sigma)|$.
\end{lem}
%%%%%%%%%%%%%%%%%%%
\begin{proof}
Suppose $ \sigma([1,k])=\{i_1, i_2, \ldots, i_k\}_<$ and
$\sigma^{-1}([1,k])=\{j_1, j_2, \ldots, j_k\}_<$.  Then
\begin{align*}
 |L_3(\sigma)|&=\sum_{s=1}^k(|\{\ell\;|\;k<\ell\leq i_s<\sigma(\ell)\}|
 +|\{\ell\;|\;\ell>j_s>\sigma(\ell)>k\}|), \\%\label{eq:L3-decompo}\\
 |L_4(\sigma)|&=\sum_{s=1}^k(|\{\ell\;|\;\ell> i_s>\sigma(\ell)>k\}|
 +|\{\ell\;|\;k<\ell<j_s\leq\sigma(\ell)\}|).%\label{eq:L4-decompo}
\end{align*}
For  $i\in[n]$  define the set $A_i(\sigma)=\{j\;|\; j\leq i < \sigma(j)\}$.
Then it is easily seen that
\begin{align}\label{eq:Ai-sigma-sigma-1}
|A_i(\sigma)|=|\{j\;|\; j> i\geq
\sigma(j)\}|=|A_i(\sigma^{-1})|.
\end{align}
Noticing that, for  $s\in[k]$, 
\begin{align*}
|\{\ell\;|\;k<\ell\leq i_s<\sigma(\ell)\}|&=|\{\ell\;|\;\ell\leq
i_s<\sigma(\ell)\}|
-|\{\ell\;|\ell \leq k<i_s<\sigma(\ell)\}|\nonumber\\
&=|A_{i_s}(\sigma)|-|\{t\;|i_t>i_s\}|,%\label{eq:L3-a}
\end{align*}
\begin{align*}
|\{\ell\;|\;\ell>j_s>\sigma(\ell)>k\}|
&=|\{\ell\;|\;\ell>j_s>\sigma(\ell)\}|-|\{\ell\;|\;\ell>j_s>k\geq\sigma(\ell)\}|\nonumber\\
&=|\{\ell\;|\;\ell>j_s>\sigma(\ell)\}|-|\{t\;|j_t>j_s\}|\nonumber\\
&=|A_{j_s}(\sigma^{-1})|-\chi(\sigma^{-1}(j_s)>j_s)-|\{t\;|j_t>j_s\}|,%\label{eq:L3-b}
\end{align*}
and 
\begin{align*}
|\{\ell\;|\;\ell>i_s>\sigma(\ell)>k\}|&=|A_{i_s}(\sigma^{-1})|-|\{t\;|\;
j_t>i_s\}|,\\%\label{eq:L4-a}\\
 |\{\ell\;|\;k\leq
\ell<j_s\leq\sigma(\ell)\}|&=|A_{j_s}(\sigma)|+\chi(\sigma^{-1}(j_s)<j_s)-|\{t\;|\;
i_t\geq j_s\}|,%\label{eq:L4-b}
\end{align*}
%%%%%%%%%%%%%%%%%%%%
 we can rewrite $|L_3(\sigma)|$ and $|L_4(\sigma)|$,  using \eqref{eq:Ai-sigma-sigma-1},  as follows:
\begin{align}
|L_3(\sigma)|&=A-\sum_{s=1}^k(\chi(\sigma^{-1}(j_s)>j_s)+|\{t\;|\;
i_t>i_s\}|+|\{t\;|\; j_t>j_s\}|),\label{eq:L3-simplifie}\\
|L_4(\sigma)|&=A+\sum_{s=1}^k(\chi(\sigma^{-1}(j_s)<j_s)
-|\{t\;|\; j_t>i_s\}|-|\{t\;|\; i_t\geq
j_s\}|),\label{eq:L4-simplifie}
\end{align}
where $A=\sum_{s=1}^k(|A_{i_s}(\sigma)|+|A_{j_s}(\sigma)|$.

Since $|\{t\;|\; i_t>i_s\}|=|\{t\;|\;
j_t>j_s\}|=k-s$, we have
\begin{align*}
\sum_{s=1}^k(|\{t\;|\; i_t>i_s\}|+|\{t\;|\;
j_t>j_s\}|)=k(k-1).%\label{eq:detail1a}
\end{align*}
Also, 
\begin{align*}
\sum_{s=1}^k(|\{t\;|\; j_t>i_s\}|+|\{t\;|\; i_t\geq j_s\}|)=\sum_{s,t=1}^k(\chi( j_t>i_s)+\chi(i_t\geq j_{s}))
=k^{2}.%\label{eq:detail2a}
\end{align*}
Substituting the above values 
 into  \eqref{eq:L3-simplifie} and
\eqref{eq:L4-simplifie} leads to
\begin{align*}
|L_3(\sigma)|-|L_4(\sigma)|
&=k-\sum_{s=1}^k(\chi(\sigma^{-1}(j_s)>j_s)+\chi(\sigma^{-1}(j_s)<j_s))=0,
\end{align*}
where the last equality follows from the fact that
$\sigma^{-1}(j_s)\neq j_s$ for all $s\in [k]$.
\end{proof}

%%%%****%%%%****%%%%****%%%%****%%%%****%%%%****%%%%****%%%%*
%%%%
%%%% New Section
%%%%
%%%%****%%%%****%%%%****%%%%****%%%%****%%%%****%%%%****%%%%*
\section{Proof of  Lemma~\ref{lem:2}}
Let  $N_2:=n_1+n_2\leq n$ and 
define 
$$
{\S_{n}^{\,(n_1,n_2)}}:=\{\sigma\in \S_{n}: (i,\sigma(i))\notin [1,n_1]^2\cup [n_1+1,N_2]^2\}.
$$
Hence, in the graph of any permutation in ${\S_{n}^{\,(n_1,n_2)}}$ there is no arc between any
two integers in $[1,n_1]$  or $[n_1+1,\, N_2]$.

 We now construct a mapping  $\Gamma^{(n_1,n_2)}: \sigma\mapsto \sigma'$ from
$\S^{\,(n_1,n_2)}_{n}$ to $\S^{\,(n_2,n_1)}_{n}$ as follows. For
$i=1,\ldots,n$,
\begin{enumerate}
\item If $i>N_2$ and $\sigma(i)>N_2$, set
$\sigma'(i)=\sigma(i)$.

\item Suppose
\begin{align*}
\{(i,\sigma(i))\;|\;i<\sigma(i)\leq N_2\}&=\{(i_1,N_2+1-j_1),\ldots,(i_p,N_2+1-j_p)\},\\
\{(\sigma(i),i)\;|\;\sigma(i)<i\leq
N_2\}&=\{(k_1,N_2+1-\ell_1),\ldots,(k_q,N_2+1-\ell_q)\}.
\end{align*}
Then set $\sigma'(j_s)=N_2+1-i_s$ and $\sigma'(N_2+1-k_t)=\ell_t$
for any $s\in[p]$ and $t\in[q]$.

\item  Let
$C=\{i\in[1,N_2]:\sigma(i)>N_2\}$ and 
$D=\{i\in[1,N_2]:\sigma^{-1}(i)>N_2\}$.
It is clear that $|C|$=$|D|$. Suppose
$C=\{c_1,c_2,\ldots,c_u\}_{<}$,
 $D=\{d_1,d_2,\ldots,d_u\}_{<}$,
$\sigma(C)=\{r_1,r_2,\ldots,r_u \}_{<}$ and
$\sigma^{-1}(D)=\{s_1,s_2,\ldots,s_u \}_{<}$. Then, there are (unique) permutations 
$\alpha,\beta\in\S_u$   satisfying
$\sigma(c_i)=r_{\alpha(i)}$ and $\sigma^{-1}(d_i)=s_{\beta(i)}$ for
each $1\leq i\leq u$. Let
\begin{align*}
E&=[1,N_2]\setminus\{j_1,\ldots,j_p,N_2+1-k_1,\ldots,N_2+1-k_q\},\\
F&=[1,N_2]\setminus\{N_2+1-i_1,\ldots,N_2+1-i_p,\ell_1,\ldots,\ell_q\}.
\end{align*}
Clearly, we have $|E|=|C|$ and $|F|=|D|$. Suppose
$E=\{e_1,\ldots,e_u\}_{<}$ and $F=\{f_1,\ldots,f_u\}_{<}$. Then set
$\sigma'(e_i)=r_{\alpha(i)}$ and $\sigma'(s_i)=f_{\beta(i)}$ for
each $1\leq i\leq u$.
\end{enumerate}
The mapping is illustrated in 
Table~\ref{tab:descriptionGamma}.

\begin{table}[h]
$$
\begin{array}{|ccc|}
\hline
\sigma &\longrightarrow &\hspace{1cm} \sigma'\\
\hline
%%%%%%%%%%%%%%%%%%%%%%%%%%%%%%%%%%%%%%%%%%%%%%%%%%%%%%%%%%%%%%%%%%%%%%%%%%%%%%%%%%%%%%%%%%%%%%%%%% type 1
 {\setlength{\unitlength}{0.6mm}\begin{picture}(70,15)(0,-5)
\put(0,0){\line(1,0){25}}\put(35,0){\line(1,0){35}}
\put(0,0){\circle*{1,3}}\put(0,0){\makebox(-2,-6)[c]{\tiny 1}}
\put(25,0){\circle*{1,3}}\put(25,0){\makebox(-2,-6)[c]{\tiny $N_2$}}
\put(35,0){\circle*{1,3}}
\put(70,0){\circle*{1,3}}\put(70,0){\makebox(-2,-6)[c]{\tiny $n$}}
 \qbezier(45,0)(52.5,10)(60,0)\put(45,0){\makebox(-2,-6)[c]{\tiny $i$}}
 \put(60,0){\makebox(-2,-6)[c]{\tiny $\sigma(i)$}}
\end{picture}}&\longrightarrow&\hspace{1cm}{\setlength{\unitlength}{0.6mm}\begin{picture}(70,15)(0,-5)
\put(0,0){\line(1,0){25}}\put(35,0){\line(1,0){35}}
\put(0,0){\circle*{1,3}}\put(0,0){\makebox(-2,-6)[c]{\tiny 1}}
\put(25,0){\circle*{1,3}}\put(25,0){\makebox(-2,-6)[c]{\tiny $N_2$}}
\put(35,0){\circle*{1,3}}
\put(70,0){\circle*{1,3}}\put(70,0){\makebox(-2,-6)[c]{\tiny $n$}}
 \qbezier(45,0)(52.5,10)(60,0)
 \put(45,0){\makebox(-2,-6)[c]{\tiny $i$}}
 \put(60,0){\makebox(-2,-6)[c]{\tiny $\sigma(i)$}}
\end{picture}}\\
%%%%%%%%%%%%%%%%%%%%%%%%%%%%%%%%%%%%%
 {\setlength{\unitlength}{0.6mm}\begin{picture}(70,15)(0,-5)
\put(0,0){\line(1,0){25}}\put(35,0){\line(1,0){35}}
\put(0,0){\circle*{1,3}}\put(0,0){\makebox(-2,-6)[c]{\tiny 1}}
\put(25,0){\circle*{1,3}}\put(25,0){\makebox(-2,-6)[c]{\tiny $N_2$}}
\put(35,0){\circle*{1,3}}
\put(70,0){\circle*{1,3}}\put(70,0){\makebox(-2,-6)[c]{\tiny $n$}}
 \qbezier(45,0)(52.5,-10)(60,0)
 \put(45,0){\makebox(-2,6)[c]{\tiny $\sigma(i)$}}
 \put(60,0){\makebox(-2,6)[c]{\tiny $i$}}
\end{picture}}&\longrightarrow&\hspace{1cm}{\setlength{\unitlength}{0.6mm}\begin{picture}(70,15)(0,-5)
\put(0,0){\line(1,0){25}}\put(35,0){\line(1,0){35}}
\put(0,0){\circle*{1,3}}\put(0,0){\makebox(-2,-6)[c]{\tiny 1}}
\put(25,0){\circle*{1,3}}\put(25,0){\makebox(-2,-6)[c]{\tiny $N_2$}}
\put(35,0){\circle*{1,3}}
\put(70,0){\circle*{1,3}}\put(70,0){\makebox(-2,-6)[c]{\tiny $n$}}
 \qbezier(45,0)(52.5,-10)(60,0)
 \put(45,0){\makebox(-2,6)[c]{\tiny $\sigma(i)$}}
 \put(60,0){\makebox(-2,6)[c]{\tiny $i$}}
\end{picture}}\\[0.5cm]
\hline
%%%%%%%%%%%%%%%%%%%%%%%%%%%%%%%%%%%%%%%%%%%%%%%%%%%%%%%%%%%%%%%%%%%%%%%%%%%%%%%%%%%%%%%%%%%%%%%%%%%type 2
\hspace{1cm} {\setlength{\unitlength}{0.6mm}
\begin{picture}(70,15)(0,-5)
\put(0,0){\line(1,0){13}}\put(20,0){\line(1,0){27}}\put(55,0){\line(1,0){15}}
\put(0,0){\circle*{1,3}}\put(0,0){\makebox(-2,-6)[c]{\tiny1}}
\put(13,0){\circle*{1,3}}\put(13,0){\makebox(-2,-6)[c]{\tiny $n_1$}}
\put(47,0){\circle*{1,3}}\put(47,0){\makebox(-2,6)[c]{\tiny $N_2$}}
\put(70,0){\circle*{1,3}}\put(70,0){\makebox(-2,-6)[c]{\tiny $n$}}
\put(20,0){\circle*{1,3}} \qbezier(5,0)(17.5,10)(33,0)
\put(5,0){\makebox(-2,-6)[c]{\tiny
$i_t$}}\put(33,0){\makebox(2,-6)[c]{\tiny $N_2+1-j_t$}}
\end{picture}
}\hspace{1cm}&\longrightarrow&\hspace{1cm}
{\setlength{\unitlength}{0.6mm}
\begin{picture}(70,15)(0,-5)
\put(0,0){\line(1,0){13}}\put(20,0){\line(1,0){27}}\put(55,0){\line(1,0){15}}
\put(0,0){\circle*{1,3}}\put(0,0){\makebox(-2,-6)[c]{\tiny1}}
\put(13,0){\circle*{1,3}}\put(13,0){\makebox(-2,-6)[c]{\tiny $n_2$}}
\put(47,0){\circle*{1,3}}\put(47,0){\makebox(-2,6)[c]{\tiny $N_2$}}
\put(70,0){\circle*{1,3}}\put(70,0){\makebox(-2,-6)[c]{\tiny $n$}}
\put(20,0){\circle*{1,3}} \qbezier(5,0)(16,10)(33,0)
\put(5,0){\makebox(-2,-6)[c]{\tiny
$j_t$}}\put(33,0){\makebox(2,-6)[c]{\tiny $N_2+1-i_t$}}
\end{picture}
}\hspace{1cm}\\
%%%%%%%%%%%%%%%%%%%%%%%%%%%%%%%%%%%%%%%%%%%%%%%%%%%%%%%%
\hspace{1cm} {\setlength{\unitlength}{0.6mm}
\begin{picture}(70,15)(0,-5)
\put(0,0){\line(1,0){13}}\put(20,0){\line(1,0){27}}\put(55,0){\line(1,0){15}}
\put(0,0){\circle*{1,3}}\put(0,0){\makebox(-2,-6)[c]{\tiny1}}
\put(13,0){\circle*{1,3}}\put(13,0){\makebox(-2,6)[c]{\tiny $n_1$}}
\put(47,0){\circle*{1,3}}\put(47,0){\makebox(-2,-6)[c]{\tiny $N_2$}}
\put(70,0){\circle*{1,3}}\put(70,0){\makebox(-2,-6)[c]{\tiny $n$}}
\put(20,0){\circle*{1,3}} \qbezier(5,0)(17.5,-10)(33,0)
\put(5,0){\makebox(-2,6)[c]{\tiny
$k_t$}}\put(33,0){\makebox(2,6)[c]{\tiny $N_2+1-\ell_t$}}
\end{picture}
}\hspace{1cm}&\longrightarrow&\hspace{1cm}
{\setlength{\unitlength}{0.6mm}
\begin{picture}(70,15)(0,-5)
\put(0,0){\line(1,0){13}}\put(20,0){\line(1,0){27}}\put(55,0){\line(1,0){15}}
\put(0,0){\circle*{1,3}}\put(0,0){\makebox(-2,-6)[c]{\tiny1}}
\put(13,0){\circle*{1,3}}\put(13,0){\makebox(-2,6)[c]{\tiny $n_2$}}
\put(47,0){\circle*{1,3}}\put(47,0){\makebox(-2,-6)[c]{\tiny $N_2$}}
\put(70,0){\circle*{1,3}}\put(70,0){\makebox(-2,-6)[c]{\tiny $n$}}
\put(20,0){\circle*{1,3}}\qbezier(5,0)(16,-10)(33,0)
\put(5,0){\makebox(-2,6)[c]{\tiny
$\ell_t$}}\put(33,0){\makebox(2,6)[c]{\tiny $N_2+1-k_t$}}
\end{picture}
}\hspace{1cm}\\[0.5cm]
\hline
%%%%%%%%%%%%%%%%%%%%%%%%%%%%%%%%%%%%%%%%%%%%%%%%%%%%%%%%%%%%%%%%%%%%%%%%%%%%%%%%%%%%%%%%%%%%%%%%%% type 3
\hspace{1cm}{\setlength{\unitlength}{0.6mm}
\begin{picture}(70,15)(0,-5)
\put(0,0){\line(1,0){35}}\put(45,0){\line(1,0){25}}
\put(0,0){\circle*{1,3}}\put(0,0){\makebox(-2,-6)[c]{\tiny 1}}
\put(35,0){\circle*{1,3}}\put(35,0){\makebox(-2,-6)[c]{\tiny $N_2$}}
\put(45,0){\circle*{1,3}}
\put(70,0){\circle*{1,3}}\put(70,0){\makebox(-2,-6)[c]{\tiny $n$}}
\qbezier(15,0)(35,10)(55,0) \put(15,0){\makebox(-2,-6)[c]{\tiny
$c_j$}}\put(55,0){\makebox(2,-6)[c]{\tiny $r_{\alpha(j)}$}}
\end{picture}
}& \longrightarrow&\hspace{1cm}{\setlength{\unitlength}{0.6mm}
\begin{picture}(70,15)(0,-5)
\put(0,0){\line(1,0){35}}\put(45,0){\line(1,0){25}}
\put(0,0){\circle*{1,3}}\put(0,0){\makebox(-2,-6)[c]{\tiny 1}}
\put(35,0){\circle*{1,3}}\put(35,0){\makebox(-2,-6)[c]{\tiny $N_2$}}
\put(45,0){\circle*{1,3}}
\put(70,0){\circle*{1,3}}\put(70,0){\makebox(-2,-6)[c]{\tiny $n$}}
\qbezier(15,0)(35,10)(55,0)\put(15,0){\makebox(-2,-6)[c]{\tiny
$e_j$}}\put(55,0){\makebox(2,-6)[c]{\tiny $r_{\alpha(j)}$}}
\end{picture}
}\\
%%%%%%%%%%%%%%%%%%%%%%%%%%%%%%%%%%%%%%%%%%%%%%%%%%%%%%%%type 6
\hspace{1cm}{\setlength{\unitlength}{0.6mm}
\begin{picture}(70,15)(0,-5)
\put(0,0){\line(1,0){35}}\put(45,0){\line(1,0){25}}
\put(0,0){\circle*{1,3}}\put(0,0){\makebox(-2,6)[c]{\tiny 1}}
\put(35,0){\circle*{1,3}}\put(35,0){\makebox(-2,6)[c]{\tiny $N_2$}}
\put(45,0){\circle*{1,3}}
\put(70,0){\circle*{1,3}}\put(70,0){\makebox(-2,6)[c]{\tiny $n$}}
\qbezier(15,0)(35,-10)(55,0)\put(15,0){\makebox(-2,6)[c]{\tiny
$e_j$}}\put(55,0){\makebox(2,6)[c]{\tiny $s_{\beta(j)}$}}
\end{picture}
}& \longrightarrow&\hspace{1cm}{\setlength{\unitlength}{0.6mm}
\begin{picture}(70,15)(0,-5)
\put(0,0){\line(1,0){35}}\put(45,0){\line(1,0){25}}
\put(0,0){\circle*{1,3}}\put(0,0){\makebox(-2,6)[c]{\tiny 1}}
\put(35,0){\circle*{1,3}}\put(35,0){\makebox(-2,6)[c]{\tiny $N_2$}}
\put(45,0){\circle*{1,3}}
\put(70,0){\circle*{1,3}}\put(70,0){\makebox(-2,6)[c]{\tiny $n$}}
\qbezier(15,0)(35,-10)(55,0) \put(15,0){\makebox(-2,6)[c]{\tiny
$f_j$}}\put(55,0){\makebox(2,6)[c]{\tiny $s_{\beta(j)}$}}
\end{picture}
}\\[0.5cm]
\hline
\end{array}
$$
\caption{The mapping
$\Gamma^{(n_1,n_2)}:\sigma\mapsto\sigma'$}\label{tab:descriptionGamma}
\end{table}

For example, if we consider the permutation in $\S_{15}^{\,(3,4)}$
whose diagram is given by
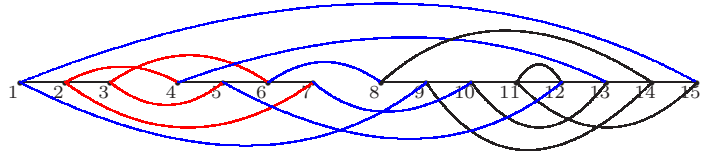
\begin{center}
{\setlength{\unitlength}{0.6mm}
\begin{picture}(150,40)(-10,10)
\put(-5,30){\line(1,0){20}}
\put(30,30){\line(1,0){30}}\put(75,30){\line(1,0){70}}
%%%%%%%%%%%%%%%%%% GRAPHES DE SIGMA
\put(-5,30){\circle*{1}}\put(-5,30){\makebox(-3,-4)[c]{\tiny 1}}
\put(5,30){\circle*{1}}\put(5,30){\makebox(-3,-4)[c]{\tiny 2}}
\put(15,30){\circle*{1}}\put(15,30){\makebox(-3,-4)[c]{\tiny 3}}
\put(30,30){\circle*{1}}\put(30,30){\makebox(-3,-4)[c]{\tiny 4}}
\put(40,30){\circle*{1}}\put(40,30){\makebox(-3,-4)[c]{\tiny 5}}
\put(50,30){\circle*{1}}\put(50,30){\makebox(-3,-4)[c]{\tiny 6}}
\put(60,30){\circle*{1}}\put(60,30){\makebox(-3,-4)[c]{\tiny 7}}
\put(75,30){\circle*{1}}\put(75,30){\makebox(-3,-4)[c]{\tiny 8}}
\put(85,30){\circle*{1}}\put(85,30){\makebox(-3,-4)[c]{\tiny 9}}
\put(95,30){\circle*{1}}\put(95,30){\makebox(-3,-4)[c]{\tiny 10}}
\put(105,30){\circle*{1}}\put(105,30){\makebox(-3,-4)[c]{\tiny
11}}\put(115,30){\circle*{1}}\put(115,30){\makebox(-3,-4)[c]{\tiny
12}}\put(125,30){\circle*{1}}\put(125,30){\makebox(-3,-4)[c]{\tiny
13}}\put(135,30){\circle*{1}}\put(135,30){\makebox(-3,-4)[c]{\tiny
14}}\put(145,30){\circle*{1}}\put(145,30){\makebox(-3,-4)[c]{\tiny
15}}
%%%%%%%%%%%%%%%%% EXCEDANCES DE SIGMA
\red{\qbezier(5,30)(17.5,37)(30,30)\qbezier(15,30)(32.5,42)(50,30)}\blue{\qbezier(-5,30)(82,65)(145,30)
\qbezier(50,30)(62.5,39)(75,30) \qbezier(30,30)(85,50)(125,30)}
\qbezier(75,30)(100,53)(135,30) \qbezier(105,30)(110,38)(115,30)
%%%%%%%%%%%%%%%% NONEXCEDANCES DE SIGMA
\red{\qbezier(5,30)(32,10)(60,30)\qbezier(15,30)(27.5,20)(40,30)}
\qbezier(85,30)(105,0)(135,30)\qbezier(95,30)(110,10)(125,30)\qbezier(105,30)(125,10)(145,30)
\blue{\qbezier(-5,30)(50,2)(85,30)\qbezier(60,30)(75,17)(95,30)\qbezier(40,30)(85,5)(115,30)}
\end{picture}
}\end{center}

then the diagram of $\Gamma^{(n_1,n_2)}(\sigma)$ is given by
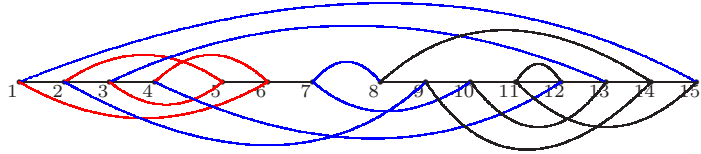
\begin{center}
{\setlength{\unitlength}{0.6mm}
\begin{picture}(150,40)(-5,10)
\put(-5,30){\line(1,0){30}}
\put(40,30){\line(1,0){20}}\put(75,30){\line(1,0){70}}
%%%%%%%%%%%%%%%%%% GRAPHES DE SIGMA
\put(-5,30){\circle*{1}}\put(-5,30){\makebox(-3,-4)[c]{\tiny 1}}
\put(5,30){\circle*{1}}\put(5,30){\makebox(-3,-4)[c]{\tiny 2}}
\put(15,30){\circle*{1}}\put(15,30){\makebox(-3,-4)[c]{\tiny 3}}
\put(25,30){\circle*{1}}\put(25,30){\makebox(-3,-4)[c]{\tiny 4}}
\put(40,30){\circle*{1}}\put(40,30){\makebox(-3,-4)[c]{\tiny 5}}
\put(50,30){\circle*{1}}\put(50,30){\makebox(-3,-4)[c]{\tiny 6}}
\put(60,30){\circle*{1}}\put(60,30){\makebox(-3,-4)[c]{\tiny 7}}
\put(75,30){\circle*{1}}\put(75,30){\makebox(-3,-4)[c]{\tiny 8}}
\put(85,30){\circle*{1}}\put(85,30){\makebox(-3,-4)[c]{\tiny 9}}
\put(95,30){\circle*{1}}\put(95,30){\makebox(-3,-4)[c]{\tiny 10}}
\put(105,30){\circle*{1}}\put(105,30){\makebox(-3,-4)[c]{\tiny
11}}\put(115,30){\circle*{1}}\put(115,30){\makebox(-3,-4)[c]{\tiny
12}}\put(125,30){\circle*{1}}\put(125,30){\makebox(-3,-4)[c]{\tiny
13}}\put(135,30){\circle*{1}}\put(135,30){\makebox(-3,-4)[c]{\tiny
14}}\put(145,30){\circle*{1}}\put(145,30){\makebox(-3,-4)[c]{\tiny
15}}
%%%%%%%%%%%%%%%%% EXCEDANCES DE SIGMA
\red{\qbezier(25,30)(37.5,42)(50,30)\qbezier(5,30)(22.5,42)(40,30)}\blue{\qbezier(-5,30)(82,65)(145,30)
\qbezier(60,30)(67.5,39)(75,30)\qbezier(15,30)(67.5,55)(125,30)}
\qbezier(75,30)(100,53)(135,30) \qbezier(105,30)(110,38)(115,30)
%%%%%%%%%%%%%%%% NONEXCEDANCES DE SIGMA
\red{\qbezier(15,30)(27.5,20)(40,30)\qbezier(-5,30)(22.5,14)(50,30)}\blue{\qbezier(5,30)(55,2)(85,30)
\qbezier(60,30)(75,17)(95,30)\qbezier(25,30)(77,5)(115,30)}
\qbezier(85,30)(105,0)(135,30)\qbezier(95,30)(110,10)(125,30)\qbezier(105,30)(125,10)(145,30)
\end{picture}
}\end{center}

It is not hard to check that $\Gamma^{(n_1,n_2)}: {\S_{n}^{\,(n_1,n_2)}}\to {\S_{n}^{\,(n_2,n_1)}}$  is well defined
and bijective because 
each step of the construction is reversible,
Actually we can prove,
the details are left to the reader, that
$(\Gamma^{(n_1,n_2)})^{-1}=\Gamma^{(n_2,n_1)}$.

\begin{prop}\label{prop:Gamma}
For each positive integers $n_1,n_2,n$, with $N_2\leq n$, the map
$\Gamma^{(n_1,n_2)}$ is a bijection from ${\S_{n}^{\,(n_1,n_2)}}$ to
${\S_{n}^{\,(n_2,n_1)}}$ such that for each
$\sigma\in\S_{n}^{\,(n_1,n_2)}$,  we have
\begin{align}\label{eq:propGamma}
(\exc,\cro)\Gamma^{(n_1,n_2)}(\sigma)=(\exc,\cro)\sigma.
\end{align}
\end{prop}
We first derive  Lemma~\ref{lem:2} from the above proposition. Let
$n=n_1+n_2+\cdots+n_k$. Then $\D(n_1,n_2,\ldots, n_k)\subseteq
{\S^{\,(n_1,n_2)}_{n}}$. By definition of $\Gamma^{(n_1,n_2)}$, for
any $\sigma\in\S_{n}^{\,(n_1,n_2)}$ and $i>N_2$ satisfying
$\sigma(i)>N_2$, we have
$i-\Gamma^{(n_1,n_2)}(\sigma)(i)=i-\sigma(i)$,
so $\Gamma^{(n_1,n_2)}(\D(n_1,n_2,\ldots, n_k))\subseteq
\D(n_2,n_3,\ldots,n_k,n_1)$. Since the cardinality of
$\D(n_1,n_2,\ldots, n_k)$ doesn't depend on  the order of the $n_i$'s
and $\Gamma^{(n_1,n_2)}$ is a bijection, we have
$$\Gamma^{(n_1,n_2)}(\D(n_1,n_2,\ldots,
n_k))=\D(n_2,n_3,\ldots,n_k,n_1).$$
Lemma~\ref{lem:2} then follows
from \eqref{eq:propGamma}.

\begin{table}[h]
$$
\begin{array}{|c|c|c|}
\hline
 i&G_i^{(n_1,n_2)}(\gamma)& G_i^{(n_2,n_1)}(\gamma)\\
 \hline
%%%%%%%%%%%%%%%%%%%%%%%%%%%%%%%%%%%%%%%%%%%%%%%%%%%%%%%%%%%%%%%%%%%%%%%%%%% type 1
1& \hspace{1cm}{
{\setlength{\unitlength}{0.6mm}\begin{picture}(70,15)(0,-5)
\put(0,0){\line(1,0){25}}\put(35,0){\line(1,0){35}}
\put(0,0){\circle*{1,3}}\put(0,0){\makebox(-2,-6)[c]{\small 1}}
\put(25,0){\circle*{1,3}}\put(25,0){\makebox(-2,-6)[c]{\tiny $N_2$}}
\put(70,0){\circle*{1,3}}\put(70,0){\makebox(-2,-6)[c]{\small $n$}}
\qbezier(40,0)(50,10)(60,0) \qbezier(45,0)(55,10)(65,0)
\end{picture}}}\hspace{1cm}&{
{\setlength{\unitlength}{0.6mm}\begin{picture}(70,15)(0,-5)
\put(0,0){\line(1,0){25}}\put(35,0){\line(1,0){35}}
\put(0,0){\circle*{1,3}}\put(0,0){\makebox(-2,-6)[c]{\small 1}}
\put(25,0){\circle*{1,3}}\put(25,0){\makebox(-2,-6)[c]{\tiny $N_2$}}
\put(70,0){\circle*{1,3}}\put(70,0){\makebox(-2,-6)[c]{\small $n$}}
\qbezier(40,0)(50,10)(60,0) \qbezier(45,0)(55,10)(65,0)
\end{picture}}}\\
%%%%%%%%%%%%%%%%%%%%%%%%%%%%%%%%%%%%% b
&{\hspace{1cm} {\setlength{\unitlength}{0.6mm}
\begin{picture}(70,15)(0,-5)
\put(0,0){\line(1,0){25}}\put(35,0){\line(1,0){35}}
\put(0,0){\circle*{1,3}}\put(0,0){\makebox(-2,-6)[c]{\small 1}}
\put(25,0){\circle*{1,3}}\put(25,0){\makebox(-2,-6)[c]{\tiny $N_2$}}
\put(70,0){\circle*{1,3}}\put(70,0){\makebox(-2,-6)[c]{\small $n$}}
\qbezier(40,0)(45,10)(50,0) \qbezier(50,0)(55,10)(60,0)
\end{picture}
}\hspace{1cm}}&{\hspace{1cm} {\setlength{\unitlength}{0.6mm}
\begin{picture}(70,15)(0,-5)
\put(0,0){\line(1,0){25}}\put(35,0){\line(1,0){35}}
\put(0,0){\circle*{1,3}}\put(0,0){\makebox(-2,-6)[c]{\small 1}}
\put(25,0){\circle*{1,3}}\put(25,0){\makebox(-2,-6)[c]{\tiny $N_2$}}
\put(70,0){\circle*{1,3}}\put(70,0){\makebox(-2,-6)[c]{\small $n$}}
\qbezier(40,0)(45,10)(50,0) \qbezier(50,0)(55,10)(60,0)
\end{picture}
}\hspace{1cm}}\\
%%%%%%%%%%%%%%%%%%%%%%%%%%%%%%%%%%%%%%%%%%%%%%%%%% c
&{\hspace{1cm}{\setlength{\unitlength}{0.6mm}
\begin{picture}(70,15)(0,-7)
\put(0,0){\line(1,0){25}}\put(35,0){\line(1,0){35}}
\put(0,0){\circle*{1,3}}\put(0,0){\makebox(-2,-6)[c]{\small 1}}
\put(25,0){\circle*{1,3}}\put(25,0){\makebox(-2,6)[c]{\tiny $N_2$}}
\put(70,0){\circle*{1,3}}\put(70,0){\makebox(-2,-6)[c]{\small $n$}}
\qbezier(40,0)(50,-10)(60,0) \qbezier(45,0)(55,-10)(65,0)
\end{picture}
}\hspace{1cm}}&{\hspace{1cm}{\setlength{\unitlength}{0.6mm}
\begin{picture}(70,15)(0,-7)
\put(0,0){\line(1,0){25}}\put(35,0){\line(1,0){35}}
\put(0,0){\circle*{1,3}}\put(0,0){\makebox(-2,-6)[c]{\small 1}}
\put(25,0){\circle*{1,3}}\put(25,0){\makebox(-2,6)[c]{\tiny $N_2$}}
\put(70,0){\circle*{1,3}}\put(70,0){\makebox(-2,-6)[c]{\small $n$}}
\qbezier(40,0)(50,-10)(60,0) \qbezier(45,0)(55,-10)(65,0)
\end{picture}
}\hspace{1cm}}\\
\hline
%%%%%%%%%%%%%%%%%%%%%%%%%%%%%%%%%%%%%%%%%%%%%%%%%%%%%%%%type 2
2&\hspace{1cm} {\setlength{\unitlength}{0.6mm}
\begin{picture}(70,15)(0,-5)
\put(0,0){\line(1,0){13}}\put(35,0){\line(-1,0){13}}\put(45,0){\line(1,0){25}}
\put(0,0){\circle*{1,3}}\put(0,0){\makebox(-2,-6)[c]{\small 1}}
\put(35,0){\circle*{1,3}}\put(37,0){\makebox(-2,-6)[c]{\tiny $N_2$}}
\put(70,0){\circle*{1,3}}\put(70,0){\makebox(-2,-6)[c]{\small $n$}}
\put(13,0){\circle*{1,3}}\put(13,0){\makebox(-2,-6)[c]{\tiny $n_1$}}
\qbezier(5,0)(15,10)(25,0) \qbezier(10,0)(20,10)(30,0)
\end{picture}
}\hspace{1cm}&\hspace{1cm} {\setlength{\unitlength}{0.6mm}
\begin{picture}(70,15)(0,-5)
\put(0,0){\line(1,0){13}}\put(35,0){\line(-1,0){13}}\put(45,0){\line(1,0){25}}
\put(0,0){\circle*{1,3}}\put(0,0){\makebox(-2,-6)[c]{\small 1}}
\put(35,0){\circle*{1,3}}\put(37,0){\makebox(-2,-6)[c]{\tiny $N_2$}}
\put(70,0){\circle*{1,3}}\put(70,0){\makebox(-2,-6)[c]{\small $n$}}
\put(13,0){\circle*{1,3}}\put(13,0){\makebox(-2,-6)[c]{\tiny $n_2$}}
\qbezier(5,0)(15,10)(25,0) \qbezier(10,0)(20,10)(30,0)
\end{picture}
}\hspace{1cm}\\
%%%%%%%%%%%%%%%%%%%%%%%%%%%%%%%%%%%%%%%%%%%%%%%%%%% b
&\hspace{1cm}{\setlength{\unitlength}{0.6mm}
\begin{picture}(70,15)(0,-5)
\put(0,0){\line(1,0){13}}\put(35,0){\line(-1,0){13}}\put(45,0){\line(1,0){25}}
\put(0,0){\circle*{1,3}}\put(0,0){\makebox(-2,-6)[c]{\small 1}}
\put(35,0){\circle*{1,3}}\put(37,0){\makebox(-2,6)[c]{\tiny $N_2$}}
\put(70,0){\circle*{1,3}}\put(70,0){\makebox(-2,-6)[c]{\small $n$}}
\put(13,0){\circle*{1,3}}\put(13,0){\makebox(-2,6)[c]{\tiny $n_1$}}
\qbezier(5,0)(15,-10)(25,0) \qbezier(10,0)(20,-10)(30,0)
\end{picture}
}\hspace{1cm}&\hspace{1cm}{\setlength{\unitlength}{0.6mm}
\begin{picture}(70,10)(0,-7)
\put(0,0){\line(1,0){13}}\put(35,0){\line(-1,0){13}}\put(45,0){\line(1,0){25}}
\put(0,0){\circle*{1,3}}\put(0,0){\makebox(-2,-6)[c]{\small 1}}
\put(35,0){\circle*{1,3}}\put(37,0){\makebox(-2,6)[c]{\tiny $N_2$}}
\put(70,0){\circle*{1,3}}\put(70,0){\makebox(-2,-6)[c]{\small $n$}}
\put(13,0){\circle*{1,3}}\put(13,0){\makebox(-2,6)[c]{\tiny $n_2$}}
\qbezier(5,0)(15,-10)(25,0) \qbezier(10,0)(20,-10)(30,0)
\end{picture}
}\hspace{1cm}\\
\hline
%%%%%%%%%%%%%%%%%%%%%%%%%%%%%%%%%%%%%%%%%%%%%%%%%%%%%%%%type 3
3&\hspace{1cm} {{\setlength{\unitlength}{0.6mm}
\begin{picture}(70,15)(0,-5)
\put(0,0){\line(1,0){25}}\put(35,0){\line(1,0){35}}
\put(0,0){\circle*{1,3}}\put(0,0){\makebox(-2,-6)[c]{\small 1}}
\put(25,0){\circle*{1,3}}\put(25,0){\makebox(-2,-6)[c]{\tiny $N_2$}}
\put(70,0){\circle*{1,3}}\put(70,0){\makebox(-2,-6)[c]{\small $n$}}
\qbezier(5,0)(30,10)(45,0) \qbezier(20,0)(40,10)(60,0)
\end{picture}
}\hspace{1cm}}& \hspace{1cm} {{\setlength{\unitlength}{0.6mm}
\begin{picture}(70,15)(0,-5)
\put(0,0){\line(1,0){25}}\put(35,0){\line(1,0){35}}
\put(0,0){\circle*{1,3}}\put(0,0){\makebox(-2,-6)[c]{\small 1}}
\put(25,0){\circle*{1,3}}\put(25,0){\makebox(-2,-6)[c]{\tiny $N_2$}}
\put(70,0){\circle*{1,3}}\put(70,0){\makebox(-2,-6)[c]{\small $n$}}
\qbezier(5,0)(30,10)(45,0) \qbezier(20,0)(40,10)(60,0)
\end{picture}
}\hspace{1cm}}\\
%%%%%%%%%%%%%%%%%%%%%%%%%%%%%%%%%%%%%%%%%%%%% b
&\hspace{1cm}{{\setlength{\unitlength}{0.6mm}
\begin{picture}(70,15)(0,-5)
\put(0,0){\line(1,0){35}}\put(45,0){\line(1,0){25}}
\put(0,0){\circle*{1,3}}\put(0,0){\makebox(-2,-6)[c]{\small 1}}
\put(35,0){\circle*{1,3}}\put(35,0){\makebox(-2,6)[c]{\tiny $N_2$}}
\put(70,0){\circle*{1,3}}\put(70,0){\makebox(-2,-6)[c]{\small $n$}}
\qbezier(10,0)(40,-10)(50,0) \qbezier(25,0)(45,-10)(65,0)
\end{picture}
}\hspace{1cm}} & \hspace{1cm}{{\setlength{\unitlength}{0.6mm}
\begin{picture}(70,15)(0,-5)
\put(0,0){\line(1,0){35}}\put(45,0){\line(1,0){25}}
\put(0,0){\circle*{1,3}}\put(0,0){\makebox(-2,-6)[c]{\small 1}}
\put(35,0){\circle*{1,3}}\put(35,0){\makebox(-2,6)[c]{\tiny $N_2$}}
\put(70,0){\circle*{1,3}}\put(70,0){\makebox(-2,-6)[c]{\small $n$}}
\qbezier(10,0)(40,-10)(50,0) \qbezier(25,0)(45,-10)(65,0)
\end{picture}
}\hspace{1cm}}\\
\hline
%%%%%%%%%%%%%%%%%%%%%%%%%%%%%%%%%%%%%%%%%%%%%%%%%%%%%%%%type 4
4&\hspace{1cm}{{\setlength{\unitlength}{0.6mm}
\begin{picture}(70,15)(0,-5)
\put(0,0){\line(1,0){25}}\put(35,0){\line(1,0){35}}
\put(0,0){\circle*{1,3}}\put(0,0){\makebox(-2,-6)[c]{\small 1}}
\put(25,0){\circle*{1,3}}\put(25,0){\makebox(-2,-6)[c]{\tiny $N_2$}}
\put(70,0){\circle*{1,3}}\put(70,0){\makebox(-2,-6)[c]{\small $n$}}
\qbezier(12.5,0)(37.5,10)(52.5,0) \qbezier(45,0)(52.5,10)(60,0)
\end{picture}
}\hspace{1cm}}& \hspace{1cm}{{\setlength{\unitlength}{0.6mm}
\begin{picture}(70,15)(0,-5)
\put(0,0){\line(1,0){25}}\put(35,0){\line(1,0){35}}
\put(0,0){\circle*{1,3}}\put(0,0){\makebox(-2,-6)[c]{\small 1}}
\put(25,0){\circle*{1,3}}\put(25,0){\makebox(-2,-6)[c]{\tiny $N_2$}}
\put(70,0){\circle*{1,3}}\put(70,0){\makebox(-2,-6)[c]{\small $n$}}
\qbezier(12.5,0)(37.5,10)(52.5,0) \qbezier(45,0)(52.5,10)(60,0)
\end{picture}
}\hspace{1cm}}\\
%%%%%%%%%%%%%%%%%%%%%%%%%%%%%%%%%%%%% b
&\hspace{1cm}{{\setlength{\unitlength}{0.6mm}
\begin{picture}(70,15)(0,-5)
\put(0,0){\line(1,0){25}}\put(35,0){\line(1,0){35}}
\put(0,0){\circle*{1,3}}\put(0,0){\makebox(-2,-6)[c]{\small 1}}
\put(25,0){\circle*{1,3}}\put(25,0){\makebox(-2,-6)[c]{\tiny $N_2$}}
\put(70,0){\circle*{1,3}}\put(70,0){\makebox(-2,-6)[c]{\small $n$}}
\qbezier(12.5,0)(37.5,10)(52.5,0) \qbezier(52.5,0)(57.5,10)(62.5,0)
\end{picture}
}\hspace{1cm}} &\hspace{1cm}{{\setlength{\unitlength}{0.6mm}
\begin{picture}(70,15)(0,-5)
\put(0,0){\line(1,0){25}}\put(35,0){\line(1,0){35}}
\put(0,0){\circle*{1,3}}\put(0,0){\makebox(-2,-6)[c]{\small 1}}
\put(25,0){\circle*{1,3}}\put(25,0){\makebox(-2,-6)[c]{\tiny $N_2$}}
\put(70,0){\circle*{1,3}}\put(70,0){\makebox(-2,-6)[c]{\small $n$}}
\qbezier(12.5,0)(37.5,10)(52.5,0) \qbezier(52.5,0)(57.5,10)(62.5,0)
\end{picture}
}\hspace{1cm}}\\
%%%%%%%%%%%%%%%%%%%%%%%%%%%%%%%%%%%%%%%%%%%%%%%%% c
&\hspace{1cm}{{\setlength{\unitlength}{0.6mm}
\begin{picture}(70,15)(0,-5)
\put(0,0){\line(1,0){25}}\put(35,0){\line(1,0){35}}
\put(0,0){\circle*{1,3}}\put(0,0){\makebox(-2,-6)[c]{\small 1}}
\put(25,0){\circle*{1,3}}\put(25,0){\makebox(-2,6)[c]{\tiny $N_2$}}
\put(70,0){\circle*{1,3}}\put(70,0){\makebox(-2,-6)[c]{\small $n$}}
\qbezier(12.5,0)(37.5,-10)(52.5,0) \qbezier(45,0)(52.5,-10)(60,0)
\end{picture}
}\hspace{1cm}}& \hspace{1cm}{{\setlength{\unitlength}{0.6mm}
\begin{picture}(70,15)(0,-5)
\put(0,0){\line(1,0){25}}\put(35,0){\line(1,0){35}}
\put(0,0){\circle*{1,3}}\put(0,0){\makebox(-2,-6)[c]{\small 1}}
\put(25,0){\circle*{1,3}}\put(25,0){\makebox(-2,6)[c]{\tiny $N_2$}}
\put(70,0){\circle*{1,3}}\put(70,0){\makebox(-2,-6)[c]{\small $n$}}
\qbezier(12.5,0)(37.5,-10)(52.5,0) \qbezier(45,0)(52.5,-10)(60,0)
\end{picture}
}\hspace{1cm}}\\
\hline
%%%%%%%%%%%%%%%%%%%%%%%%%%%%%%%%%%%%%%%%%%%%%%%%%%%%%%%%type 5
5&\hspace{1cm}{\setlength{\unitlength}{0.6mm}
\begin{picture}(70,15)(0,-5)
\put(0,0){\line(1,0){13}} \put(35,0){\line(-1,0){13}}
\put(45,0){\line(1,0){25}}
\put(0,0){\circle*{1,3}}\put(0,0){\makebox(-2,-6)[c]{\small 1}}
\put(35,0){\circle*{1,3}}\put(35,0){\makebox(-2,-6)[c]{\tiny $N_2$}}
\put(70,0){\circle*{1,3}}\put(70,0){\makebox(-2,-6)[c]{\small $n$}}
\put(13,0){\circle*{1,3}}\put(13,0){\makebox(-2,-6)[c]{\tiny $n_1$}}
\qbezier(10,0)(17.5,10)(25,0) \qbezier(17.5,0)(37.5,10)(57.5,0)
\end{picture}
}\hspace{1cm}&  \hspace{1cm}{\setlength{\unitlength}{0.6mm}
\begin{picture}(70,10)(0,-5)
\put(0,0){\line(1,0){13}} \put(35,0){\line(-1,0){13}}
\put(45,0){\line(1,0){25}}
\put(0,0){\circle*{1,3}}\put(0,0){\makebox(-2,-6)[c]{\small 1}}
\put(35,0){\circle*{1,3}}\put(35,0){\makebox(-2,-6)[c]{\tiny $N_2$}}
\put(70,0){\circle*{1,3}}\put(70,0){\makebox(-2,-6)[c]{\small $n$}}
\put(13,0){\circle*{1,3}}\put(13,0){\makebox(-2,-6)[c]{\tiny $n_2$}}
\qbezier(10,0)(17.5,10)(25,0) \qbezier(17.5,0)(37.5,10)(57.5,0)
\end{picture}
}\hspace{1cm}\\
%%%%%%%%%%%%%%%%%%%%%%%%%%%%%%%%%%%%% b
& \hspace{1cm}{\setlength{\unitlength}{0.6mm}
\begin{picture}(70,15)(0,-5)
\put(0,0){\line(1,0){13}}\put(35,0){\line(-1,0){13}}\put(45,0){\line(1,0){25}}
\put(0,0){\circle*{1,3}}\put(0,0){\makebox(-2,-6)[c]{\small 1}}
\put(35,0){\circle*{1,3}}\put(35,0){\makebox(-2,-6)[c]{\tiny $N_2$}}
\put(70,0){\circle*{1,3}}\put(70,0){\makebox(-2,-6)[c]{\small $n$}}
\put(13,0){\circle*{1,3}}\put(13,0){\makebox(-2,-6)[c]{\tiny $n_1$}}
\qbezier(10,0)(17.5,10)(25,0) \qbezier(25,0)(40,10)(57.5,0)
\end{picture}
}\hspace{1cm} &\hspace{1cm}{\setlength{\unitlength}{0.6mm}
\begin{picture}(70,15)(0,-5)
\put(0,0){\line(1,0){13}}\put(35,0){\line(-1,0){13}}\put(45,0){\line(1,0){25}}
\put(0,0){\circle*{1,3}}\put(0,0){\makebox(-2,-6)[c]{\small 1}}
\put(35,0){\circle*{1,3}}\put(35,0){\makebox(-2,-6)[c]{\tiny $N_2$}}
\put(70,0){\circle*{1,3}}\put(70,0){\makebox(-2,-6)[c]{\small $n$}}
\put(13,0){\circle*{1,3}}\put(13,0){\makebox(-2,-6)[c]{\tiny $n_2$}}
\qbezier(10,0)(17.5,10)(25,0) \qbezier(25,0)(40,10)(57.5,0)
\end{picture}
}\hspace{1cm}\\
%%%%%%%%%%%%%%%%%%%%%%%%%%%%%%%%% c
& \hspace{1cm}{\setlength{\unitlength}{0.6mm}
\begin{picture}(70,15)(0,-5)
\put(0,0){\line(1,0){13}}\put(35,0){\line(-1,0){13}}\put(45,0){\line(1,0){25}}
\put(0,0){\circle*{1,3}}\put(0,0){\makebox(-2,-6)[c]{\small 1}}
\put(35,0){\circle*{1,3}}\put(35,0){\makebox(-2,6)[c]{\tiny $N_2$}}
\put(70,0){\circle*{1,3}}\put(70,0){\makebox(-2,-6)[c]{\small $n$}}
\put(13,0){\circle*{1,3}}\put(13,0){\makebox(-2,6)[c]{\tiny $n_1$}}
\qbezier(10,0)(17.5,-10)(25,0) \qbezier(17.5,0)(37.5,-10)(57.5,0)
\end{picture}
}\hspace{1cm}& \hspace{1cm}{\setlength{\unitlength}{0.6mm}
\begin{picture}(70,15)(0,-5)
\put(0,0){\line(1,0){13}}\put(35,0){\line(-1,0){13}}\put(45,0){\line(1,0){25}}
\put(0,0){\circle*{1,3}}\put(0,0){\makebox(-2,-6)[c]{\small 1}}
\put(35,0){\circle*{1,3}}\put(35,0){\makebox(-2,6)[c]{\tiny $N_2$}}
\put(70,0){\circle*{1,3}}\put(70,0){\makebox(-2,-6)[c]{\small $n$}}
\put(13,0){\circle*{1,3}}\put(13,0){\makebox(-2,6)[c]{\tiny $n_2$}}
\qbezier(10,0)(17.5,-10)(25,0) \qbezier(17.5,0)(37.5,-10)(57.5,0)
\end{picture}
}\hspace{1cm}\\
\hline
\end{array}
$$
\caption{Forms of the crossings in $G^{(n_1,n_2)}_i(\gamma)$  and
$G^{(n_2,n_1)}_i(\gamma)$.}\label{tab:forme-cr-2}
\end{table}

\noindent{\bf Proof of Proposition~\ref{prop:Gamma}.} It was shown above
that $\Gamma^{(n_1,n_2)}$ is bijective. Let
$\sigma\in{\S_{n}^{\,(n_1,n_2)}}$ and
$\sigma':=\Gamma^{(n_1,n_2)}(\sigma)$. The equality
$\exc(\sigma')=\exc(\sigma)$ is an immediate consequence of the
definition of $\Gamma^{(n_1,n_2)}$. It then remains to prove that
$\cro(\sigma')=\cro(\sigma)$. The idea is the same as  for the
proof of Eq.~\eqref{lem:1}. We first decompose the number of
crossings of $\sigma$ and $\sigma'$. For each permutation
$\gamma\in\S_n$, set
\begin{align*}
G^{(n_1,n_2)}_1(\gamma)&= \{(i,j)\;|\;N_2<i<j\leq
\gamma(i)<\gamma(j)
\quad\text{or}\quad i>j>\gamma(i)>\gamma(j)>N_2\},\\
G^{(n_1,n_2)}_2(\gamma)&= \{(i,j)\;|\;i<j< \gamma(i)<\gamma(j)\leq
N_2
\quad\text{or}\quad N_2\geq i>j>\gamma(i)>\gamma(j)\},\\
G^{(n_1,n_2)}_3(\gamma)&=\{(i,j)\;|\;i<j\leq N_2<\gamma(i)<\gamma(j)
\quad\text{or}\quad i>j>N_2\geq\gamma(i)>\gamma(j)\},\\
G^{(n_1,n_2)}_4(\gamma)&=\{(i,j)\;|\;i\leq
N_2<j\leq\gamma(i)<\gamma(j)
\quad\text{or}\quad i>j>\gamma(i)>N_2\geq \gamma(j)\},\\
G^{(n_1,n_2)}_5(\gamma)&=\{(i,j)\;|\;i<j\leq\gamma(i)\leq
N_2<\gamma(j) \quad\text{or}\quad i>N_2\geq j>\gamma(i)>\gamma(j)\}.
\end{align*}

Clearly, for any $\gamma\in\S^{\,(n_1,n_2)}_n$, we have
$\cro(\gamma)=\sum_{i=1}^{5}|G_i^{(n_1,n_2)}(\gamma)|$.
 In particular,
\begin{align}\label{eq:decomp Sn1n2-Sn2n1}
\cro(\sigma)=\sum_{i=1}^{5}|G_i^{(n_1,n_2)}(\sigma)|\quad\text{and}\quad
\cro(\sigma')=\sum_{i=1}^{5}|G_i^{(n_2,n_1)}(\sigma')|.
\end{align}

 The crossings
of  $G_i^{(n_1,n_2)}$'s and $G_i^{(n_2,n_1)}$'s are illustrated in
Table~\ref{tab:forme-cr-2}. By  the definition of
$\Gamma^{(n_1,n_2)}$,  it is readily seen (see Row~1 in
Table~\ref{tab:effet Gamma sur cr}) that
$G^{(n_1,n_2)}_1(\sigma)=G^{(n_2,n_1)}_1(\sigma')$ and thus
$|G^{(n_1,n_2)}_1(\sigma)|=|G^{(n_2,n_1)}_1(\sigma')|$. Similarly, we can prove (see Table~\ref{tab:effet Gamma sur cr})
that $|G^{(n_1,n_2)}_i(\sigma)|=|G^{(n_2,n_1)}_i(\sigma')|$ for
$i=2,3,4$. It remains to prove that
$|G^{(n_1,n_2)}_5(\sigma)|=|G^{(n_2,n_1)}_5(\sigma')|$. This  will
follow from the following lemma.

\begin{table}[h]%crmap
$$
\begin{array}{|ccc|}
\hline
\sigma&\longrightarrow& \sigma'\\
 \hline
%%%%%%%%%%%%%%%%%%%%%%%%%%%%%%%%%%%%%%%%%%%%%%%%%%%%%%%%%%%%%%%%%%%%%%%%%%%%%%%%%%%%%%%%%%%%%%%%%%%%%% type 1
\hspace{0.5cm}{\setlength{\unitlength}{0.6mm}\begin{picture}(80,15)(0,-5)
\put(0,0){\line(1,0){15}}\put(25,0){\line(1,0){55}}
\put(0,0){\circle*{1,3}}\put(0,0){\makebox(-2,-6)[c]{\tiny 1}}
\put(15,0){\circle*{1,3}}\put(15,0){\makebox(-2,6)[c]{\tiny $N_2$}}
\put(80,0){\circle*{1,3}}\put(80,0){\makebox(2,6)[c]{\tiny $n$}}
\qbezier(30,0)(46,10)(63,0) \qbezier(47,0)(60,10)(75,0)
\put(30,0){\circle*{1,3}}\put(30,0){\makebox(-2,-6)[c]{\tiny $i$}}
\put(63,0){\circle*{1,3}}\put(63,0){\makebox(-2,-6)[c]{\tiny$\sigma_i$}}
\put(47,0){\circle*{1,3}}\put(47,0){\makebox(-2,-6)[c]{\tiny $j$}}
\put(75,0){\circle*{1,3}}\put(75,0){\makebox(-2,-6)[c]{\tiny$\sigma_j$}}
\end{picture}}&&
{\setlength{\unitlength}{0.6mm}\begin{picture}(80,15)(0,-5)
\put(0,0){\line(1,0){15}}\put(25,0){\line(1,0){55}}
\put(0,0){\circle*{1,3}}\put(0,0){\makebox(-2,-6)[c]{\tiny 1}}
\put(15,0){\circle*{1,3}}\put(15,0){\makebox(-2,6)[c]{\tiny $N_2$}}
\put(80,0){\circle*{1,3}}\put(80,0){\makebox(2,6)[c]{\tiny $n$}}
\qbezier(30,0)(46,10)(63,0) \qbezier(47,0)(60,10)(75,0)
\put(30,0){\circle*{1,3}}\put(30,0){\makebox(-2,-6)[c]{\tiny $i$}}
\put(63,0){\circle*{1,3}}\put(63,0){\makebox(-2,-6)[c]{\tiny$\sigma_i$}}
\put(47,0){\circle*{1,3}}\put(47,0){\makebox(-2,-6)[c]{\tiny $j$}}
\put(75,0){\circle*{1,3}}\put(75,0){\makebox(-2,-6)[c]{\tiny$\sigma_j$}}
\end{picture}}\hspace{0.5cm}\\
%%%%%%%%%%%%%%%%%%%%%%%%%%%%%%%%%%%%% b
\hspace{0.5cm}{\setlength{\unitlength}{0.6mm}\begin{picture}(80,15)(0,-5)
\put(0,0){\line(1,0){15}}\put(25,0){\line(1,0){55}}
\put(0,0){\circle*{1,3}}\put(0,0){\makebox(-2,-6)[c]{\tiny 1}}
\put(15,0){\circle*{1,3}}\put(15,0){\makebox(-2,6)[c]{\tiny $N_2$}}
\put(80,0){\circle*{1,3}}\put(80,0){\makebox(2,6)[c]{\tiny $n$}}
\qbezier(35,0)(43,10)(52,0) \qbezier(52,0)(61,10)(70,0)
\put(35,0){\circle*{1,3}}\put(35,0){\makebox(-2,-6)[c]{\tiny $i$}}
\put(52,0){\circle*{1,3}}\put(52,0){\makebox(-2,-6)[c]{\tiny$j$}}
\put(70,0){\circle*{1,3}}\put(70,0){\makebox(-2,-6)[c]{\tiny$\sigma_j$}}
\end{picture}}&&
{\setlength{\unitlength}{0.6mm}\begin{picture}(80,15)(0,-5)
\put(0,0){\line(1,0){15}}\put(25,0){\line(1,0){55}}
\put(0,0){\circle*{1,3}}\put(0,0){\makebox(-2,-6)[c]{\tiny 1}}
\put(15,0){\circle*{1,3}}\put(15,0){\makebox(-2,6)[c]{\tiny $N_2$}}
\put(80,0){\circle*{1,3}}\put(80,0){\makebox(2,6)[c]{\tiny $n$}}
\qbezier(35,0)(43,10)(52,0) \qbezier(52,0)(61,10)(70,0)
\put(35,0){\circle*{1,3}}\put(35,0){\makebox(-2,-6)[c]{\tiny $i$}}
\put(52,0){\circle*{1,3}}\put(52,0){\makebox(-2,-6)[c]{\tiny$j$}}
\put(70,0){\circle*{1,3}}\put(70,0){\makebox(-2,-6)[c]{\tiny$\sigma_j$}}
\end{picture}}\hspace{0.5cm}\\
%%%%%%%%%%%%%%%%%%%%%%%%%%%%%%%%%%%%%%%%%%%%%%%%%% c
\hspace{0.5cm}{\setlength{\unitlength}{0.6mm}
\begin{picture}(80,15)(0,-7)
\put(0,0){\line(1,0){15}}\put(25,0){\line(1,0){55}}
\put(0,0){\circle*{1,3}}\put(0,0){\makebox(-2,-6)[c]{\tiny 1}}
\put(15,0){\circle*{1,3}}\put(15,0){\makebox(-2,6)[c]{\tiny $N_2$}}
\put(80,0){\circle*{1,3}}\put(80,0){\makebox(2,-6)[c]{\tiny $n$}}
\qbezier(30,0)(46,-10)(63,0) \qbezier(47,0)(60,-10)(75,0)
\put(30,0){\circle*{1,3}}\put(30,0){\makebox(-2,6)[c]{\tiny$\sigma_j$}}
\put(63,0){\circle*{1,3}}\put(63,0){\makebox(-2,6)[c]{\tiny$j$}}
\put(47,0){\circle*{1,3}}\put(47,0){\makebox(-2,6)[c]{\tiny$\sigma_i$}}
\put(75,0){\circle*{1,3}}\put(75,0){\makebox(-2,6)[c]{\tiny$i$}}
\end{picture}}&&{\setlength{\unitlength}{0.6mm}
\begin{picture}(80,15)(0,-7)
\put(0,0){\line(1,0){15}}\put(25,0){\line(1,0){55}}
\put(0,0){\circle*{1,3}}\put(0,0){\makebox(-2,-6)[c]{\tiny 1}}
\put(15,0){\circle*{1,3}}\put(15,0){\makebox(-2,6)[c]{\tiny $N_2$}}
\put(80,0){\circle*{1,3}}\put(80,0){\makebox(2,-6)[c]{\tiny $n$}}
\qbezier(30,0)(46,-10)(63,0) \qbezier(47,0)(60,-10)(75,0)
\put(30,0){\circle*{1,3}}\put(30,0){\makebox(-2,6)[c]{\tiny$\sigma_j$}}
\put(63,0){\circle*{1,3}}\put(63,0){\makebox(-2,6)[c]{\tiny$j$}}
\put(47,0){\circle*{1,3}}\put(47,0){\makebox(-2,6)[c]{\tiny$\sigma_i$}}
\put(75,0){\circle*{1,3}}\put(75,0){\makebox(-2,6)[c]{\tiny$i$}}
\end{picture}}\hspace{0.5cm}\\
\hline
%%%%%%%%%%%%%%%%%%%%%%%%%%%%%%%%%%%%%%%%%%%%%%%%%%%%%%%%%%%%%%%%%%%%%%%%%%%%%%%%%%%%%%%%%%%%%%%%%%%%%%%%%%%% type 2
\hspace{0.5cm}{\setlength{\unitlength}{0.6mm}
\begin{picture}(80,15)(0,-5)
\put(0,0){\line(1,0){20}}\put(25,0){\line(1,0){35}}\put(70,0){\line(1,0){10}}
\put(0,0){\circle*{1,3}}\put(0,0){\makebox(-2,-6)[c]{\tiny 1}}
\put(60,0){\circle*{1,3}}\put(60,0){\makebox(2,6)[c]{\tiny $N_2$}}
\put(80,0){\circle*{1,3}}\put(80,0){\makebox(-2,-6)[c]{\tiny $n$}}
\put(20,0){\circle*{1,3}}\put(20,0){\makebox(-2,-6)[c]{\tiny $n_1$}}
\qbezier(5,0)(17,10)(30,0) \qbezier(13,0)(30,10)(50,0)
\put(5,0){\circle*{1,3}}\put(5,0){\makebox(-2,-6)[c]{\tiny $i_s$}}
\put(30,0){\circle*{1,3}}\put(30,0){\makebox(5,-6)[c]{\tiny$N_2$+$1$-$j_s$}}
\put(13,0){\circle*{1,3}}\put(13,0){\makebox(-2,-6)[c]{\tiny $i_t$}}
\put(50,0){\circle*{1,3}}\put(50,0){\makebox(4,-6)[c]{\tiny$N_2$+$1$-$j_t$}}
\end{picture}}&&
{\setlength{\unitlength}{0.6mm}
\begin{picture}(80,15)(0,-5)
\put(0,0){\line(1,0){20}}\put(25,0){\line(1,0){35}}\put(70,0){\line(1,0){10}}
\put(0,0){\circle*{1,3}}\put(0,0){\makebox(-2,-6)[c]{\tiny 1}}
\put(60,0){\circle*{1,3}}\put(60,0){\makebox(2,6)[c]{\tiny $N_2$}}
\put(80,0){\circle*{1,3}}\put(80,0){\makebox(-2,-6)[c]{\tiny $n$}}
\put(20,0){\circle*{1,3}}\put(20,0){\makebox(-2,-6)[c]{\tiny $n_2$}}
\qbezier(5,0)(17,10)(30,0) \qbezier(13,0)(30,10)(50,0)
\put(5,0){\circle*{1,3}}\put(5,0){\makebox(-2,-6)[c]{\tiny $j_t$}}
\put(30,0){\circle*{1,3}}\put(30,0){\makebox(5,-6)[c]{\tiny$N_2$+$1$-$i_t$}}
\put(13,0){\circle*{1,3}}\put(13,0){\makebox(-2,-6)[c]{\tiny $j_s$}}
\put(50,0){\circle*{1,3}}\put(50,0){\makebox(4,-6)[c]{\tiny$N_2$+$1$-$i_s$}}
\end{picture}}\hspace{0.5cm}\\
%%%%%%%%%%%%%%%%%%%%%%%%%%%%%%%%%%%%%%%%%%%%%%%%%%% b
\hspace{0.5cm}{\setlength{\unitlength}{0.6mm}
\begin{picture}(80,15)(0,-7)
\put(0,0){\line(1,0){20}}\put(25,0){\line(1,0){35}}\put(70,0){\line(1,0){10}}
\put(0,0){\circle*{1,3}}\put(0,0){\makebox(-2,-6)[c]{\tiny 1}}
\put(60,0){\circle*{1,3}}\put(60,0){\makebox(2,-6)[c]{\tiny $N_2$}}
\put(80,0){\circle*{1,3}}\put(80,0){\makebox(-2,-6)[c]{\tiny $n$}}
\put(20,0){\circle*{1,3}}\put(20,0){\makebox(-2,6)[c]{\tiny $n_1$}}
\qbezier(5,0)(17,-10)(30,0) \qbezier(13,0)(30,-10)(50,0)
\put(5,0){\circle*{1,3}}\put(5,0){\makebox(-2,6)[c]{\tiny $k_s$}}
\put(30,0){\circle*{1,3}}\put(30,0){\makebox(5,6)[c]{\tiny$N_2$+$1$-$\ell_s$}}
\put(13,0){\circle*{1,3}}\put(13,0){\makebox(-2,6)[c]{\tiny $k_t$}}
\put(50,0){\circle*{1,3}}\put(50,0){\makebox(4,6)[c]{\tiny$N_2$+$1$-$\ell_t$}}
\end{picture}}&&
{\setlength{\unitlength}{0.6mm}
\begin{picture}(80,15)(0,-7)
\put(0,0){\line(1,0){20}}\put(25,0){\line(1,0){35}}\put(70,0){\line(1,0){10}}
\put(0,0){\circle*{1,3}}\put(0,0){\makebox(-2,-6)[c]{\tiny 1}}
\put(60,0){\circle*{1,3}}\put(60,0){\makebox(2,-6)[c]{\tiny $N_2$}}
\put(80,0){\circle*{1,3}}\put(80,0){\makebox(-2,-6)[c]{\tiny $n$}}
\put(20,0){\circle*{1,3}}\put(20,0){\makebox(-2,6)[c]{\tiny $n_2$}}
\qbezier(5,0)(17,-10)(30,0) \qbezier(13,0)(30,-10)(50,0)
\put(5,0){\circle*{1,3}}\put(5,0){\makebox(-2,6)[c]{\tiny $\ell_t$}}
\put(30,0){\circle*{1,3}}\put(30,0){\makebox(5,6)[c]{\tiny$N_2$+$1$-$k_t$}}
\put(13,0){\circle*{1,3}}\put(13,0){\makebox(-2,6)[c]{\tiny$\ell_s$}}
\put(50,0){\circle*{1,3}}\put(50,0){\makebox(4,6)[c]{\tiny$N_2$+$1$-$k_s$}}
\end{picture}
}\hspace{0.5cm}\\
\hline
%%%%%%%%%%%%%%%%%%%%%%%%%%%%%%%%%%%%%%%%%%%%%%%%%%%%%%%%%%%%%%%%%%%%%%%%%%%%%%%%%%%%%%%%%%%%%%%%%%%%%%% type 3
\hspace{0.5cm}{\setlength{\unitlength}{0.6mm}\begin{picture}(80,15)(0,-5)
\put(0,0){\line(1,0){35}}\put(45,0){\line(1,0){35}}
\put(0,0){\circle*{1,3}}\put(0,0){\makebox(-2,-6)[c]{\tiny 1}}
\put(35,0){\circle*{1,3}}\put(35,0){\makebox(-2,-6)[c]{\tiny $N_2$}}
\put(80,0){\circle*{1,3}}\put(80,0){\makebox(2,6)[c]{\tiny $n$}}
\qbezier(12,0)(35,12)(57,0) \qbezier(24,0)(47,12)(69,0)
\put(12,0){\circle*{1,3}}\put(12,0){\makebox(-2,-6)[c]{\tiny $c_i$}}
\put(24,0){\circle*{1,3}}\put(24,0){\makebox(-2,-6)[c]{\tiny$c_j$}}
\put(57,0){\circle*{1,3}}\put(57,0){\makebox(-2,-6)[c]{\tiny
$r_{\alpha_i}$}}
\put(69,0){\circle*{1,3}}\put(69,0){\makebox(-2,-6)[c]{\tiny$r_{\alpha_j}$}}
\end{picture}}&&
{\setlength{\unitlength}{0.6mm}\begin{picture}(80,15)(0,-5)
\put(0,0){\line(1,0){35}}\put(45,0){\line(1,0){35}}
\put(0,0){\circle*{1,3}}\put(0,0){\makebox(-2,-6)[c]{\tiny 1}}
\put(35,0){\circle*{1,3}}\put(35,0){\makebox(-2,-6)[c]{\tiny $N_2$}}
\put(80,0){\circle*{1,3}}\put(80,0){\makebox(2,6)[c]{\tiny $n$}}
\qbezier(12,0)(35,12)(57,0) \qbezier(24,0)(47,12)(69,0)
\put(12,0){\circle*{1,3}}\put(12,0){\makebox(-2,-6)[c]{\tiny $e_i$}}
\put(24,0){\circle*{1,3}}\put(24,0){\makebox(-2,-6)[c]{\tiny$e_j$}}
\put(57,0){\circle*{1,3}}\put(57,0){\makebox(-2,-6)[c]{\tiny
$r_{\alpha_i}$}}
\put(69,0){\circle*{1,3}}\put(69,0){\makebox(-2,-6)[c]{\tiny$r_{\alpha_j}$}}
\end{picture}}\hspace{0.5cm}\\
%%%%%%%%%%%%%%%%%%%%%%%%%%%%%%%%%%%%%%%%%%%%% b
\hspace{0.5cm}{\setlength{\unitlength}{0.6mm}\begin{picture}(80,15)(0,-7)
\put(0,0){\line(1,0){35}}\put(45,0){\line(1,0){35}}
\put(0,0){\circle*{1,3}}\put(0,0){\makebox(-2,-6)[c]{\tiny 1}}
\put(35,0){\circle*{1,3}}\put(35,0){\makebox(-2,6)[c]{\tiny $N_2$}}
\put(80,0){\circle*{1,3}}\put(80,0){\makebox(2,6)[c]{\tiny $n$}}
\qbezier(12,0)(35,-12)(57,0) \qbezier(24,0)(47,-12)(69,0)
\put(12,0){\circle*{1,3}}\put(12,0){\makebox(-2,6)[c]{\tiny $d_i$}}
\put(24,0){\circle*{1,3}}\put(24,0){\makebox(-2,6)[c]{\tiny$d_j$}}
\put(57,0){\circle*{1,3}}\put(57,0){\makebox(-2,6)[c]{\tiny
$s_{\beta_i}$}}
\put(69,0){\circle*{1,3}}\put(69,0){\makebox(-2,6)[c]{\tiny$s_{\beta_j}$}}
\end{picture}}&&
{\setlength{\unitlength}{0.6mm}\begin{picture}(80,15)(0,-7)
\put(0,0){\line(1,0){35}}\put(45,0){\line(1,0){35}}
\put(0,0){\circle*{1,3}}\put(0,0){\makebox(-2,-6)[c]{\tiny 1}}
\put(35,0){\circle*{1,3}}\put(35,0){\makebox(-2,6)[c]{\tiny $N_2$}}
\put(80,0){\circle*{1,3}}\put(80,0){\makebox(2,6)[c]{\tiny $n$}}
\qbezier(12,0)(35,-12)(57,0) \qbezier(24,0)(47,-12)(69,0)
\put(12,0){\circle*{1,3}}\put(12,0){\makebox(-2,6)[c]{\tiny $f_i$}}
\put(24,0){\circle*{1,3}}\put(24,0){\makebox(-2,6)[c]{\tiny$f_j$}}
\put(57,0){\circle*{1,3}}\put(57,0){\makebox(-2,6)[c]{\tiny
$s_{\beta_i}$}}
\put(69,0){\circle*{1,3}}\put(69,0){\makebox(-2,6)[c]{\tiny$s_{\beta_j}$}}
\end{picture}}\hspace{0.5cm}\\
\hline
%%%%%%%%%%%%%%%%%%%%%%%%%%%%%%%%%%%%%%%%%%%%%%%%%%%%%%%%%%%%%%%%%%%%%%%%%%%%%%%%%%%%%%%%%%%%%%%%%%% type 4
\hspace{0.5cm}{\setlength{\unitlength}{0.6mm}\begin{picture}(80,15)(0,-5)
\put(0,0){\line(1,0){15}}\put(25,0){\line(1,0){55}}
\put(0,0){\circle*{1,3}}\put(0,0){\makebox(-2,-6)[c]{\tiny 1}}
\put(15,0){\circle*{1,3}}\put(15,0){\makebox(-2,-6)[c]{\tiny $N_2$}}
\put(80,0){\circle*{1,3}}\put(80,0){\makebox(2,6)[c]{\tiny $n$}}
\qbezier(7,0)(28,12)(50,0) \qbezier(30,0)(50,10)(70,0)
\put(7,0){\circle*{1,3}}\put(7,0){\makebox(-2,-6)[c]{\tiny $c_i$}}
\put(50,0){\circle*{1,3}}\put(50,0){\makebox(-2,-6)[c]{\tiny$r_{\alpha_i}$}}
\put(30,0){\circle*{1,3}}\put(30,0){\makebox(-2,-6)[c]{\tiny $j$}}
\put(70,0){\circle*{1,3}}\put(70,0){\makebox(-2,-6)[c]{\tiny$\sigma_j$}}
\end{picture}}&&
{\setlength{\unitlength}{0.6mm}\begin{picture}(80,15)(0,-5)
\put(0,0){\line(1,0){15}}\put(25,0){\line(1,0){55}}
\put(0,0){\circle*{1,3}}\put(0,0){\makebox(-2,-6)[c]{\tiny 1}}
\put(15,0){\circle*{1,3}}\put(15,0){\makebox(-2,-6)[c]{\tiny $N_2$}}
\put(80,0){\circle*{1,3}}\put(80,0){\makebox(2,6)[c]{\tiny $n$}}
\qbezier(7,0)(28,12)(50,0) \qbezier(30,0)(50,10)(70,0)
\put(7,0){\circle*{1,3}}\put(7,0){\makebox(-2,-6)[c]{\tiny $e_i$}}
\put(50,0){\circle*{1,3}}\put(50,0){\makebox(-2,-6)[c]{\tiny$r_{\alpha_i}$}}
\put(30,0){\circle*{1,3}}\put(30,0){\makebox(-2,-6)[c]{\tiny $j$}}
\put(70,0){\circle*{1,3}}\put(70,0){\makebox(-2,-6)[c]{\tiny$\sigma_j$}}
\end{picture}}\hspace{0.5cm}\\
%%%%%%%%%%%%%%%%%%%%%%%%%%%%%%%%%%%%% b
\hspace{0.5cm}{\setlength{\unitlength}{0.6mm}\begin{picture}(80,15)(0,-5)
\put(0,0){\line(1,0){15}}\put(25,0){\line(1,0){55}}
\put(0,0){\circle*{1,3}}\put(0,0){\makebox(-2,-6)[c]{\tiny 1}}
\put(15,0){\circle*{1,3}}\put(15,0){\makebox(-2,-6)[c]{\tiny $N_2$}}
\put(80,0){\circle*{1,3}}\put(80,0){\makebox(2,6)[c]{\tiny $n$}}
\qbezier(7,0)(24,12)(40,0) \qbezier(40,0)(55,10)(70,0)
\put(7,0){\circle*{1,3}}\put(7,0){\makebox(-2,-6)[c]{\tiny $c_i$}}
\put(40,0){\circle*{1,3}}\put(40,0){\makebox(-2,-6)[c]{\tiny$r_{\alpha_i}$}}
\put(70,0){\circle*{1,3}}\put(70,0){\makebox(-2,-6)[c]{\tiny $j$}}
\end{picture}}&&
{\setlength{\unitlength}{0.6mm}\begin{picture}(80,15)(0,-5)
\put(0,0){\line(1,0){15}}\put(25,0){\line(1,0){55}}
\put(0,0){\circle*{1,3}}\put(0,0){\makebox(-2,-6)[c]{\tiny 1}}
\put(15,0){\circle*{1,3}}\put(15,0){\makebox(-2,-6)[c]{\tiny $N_2$}}
\put(80,0){\circle*{1,3}}\put(80,0){\makebox(2,6)[c]{\tiny $n$}}
\qbezier(7,0)(24,12)(40,0) \qbezier(40,0)(55,10)(70,0)
\put(7,0){\circle*{1,3}}\put(7,0){\makebox(-2,-6)[c]{\tiny $e_i$}}
\put(40,0){\circle*{1,3}}\put(40,0){\makebox(-2,-6)[c]{\tiny$r_{\alpha_i}$}}
\put(70,0){\circle*{1,3}}\put(70,0){\makebox(-2,-6)[c]{\tiny $j$}}
\end{picture}}\hspace{0.5cm}\\
%%%%%%%%%%%%%%%%%%%%%%%%%%%%%%%%%%%%%%%%%%%%%%%%% c
\hspace{0.5cm}{\setlength{\unitlength}{0.6mm}\begin{picture}(80,15)(0,-7)
\put(0,0){\line(1,0){15}}\put(25,0){\line(1,0){55}}
\put(0,0){\circle*{1,3}}\put(0,0){\makebox(-2,6)[c]{\tiny 1}}
\put(15,0){\circle*{1,3}}\put(15,0){\makebox(-2,6)[c]{\tiny $N_2$}}
\put(80,0){\circle*{1,3}}\put(80,0){\makebox(2,6)[c]{\tiny $n$}}
\qbezier(7,0)(28,-12)(50,0) \qbezier(30,0)(50,-10)(70,0)
\put(7,0){\circle*{1,3}}\put(7,0){\makebox(-2,6)[c]{\tiny $d_i$}}
\put(50,0){\circle*{1,3}}\put(50,0){\makebox(-2,6)[c]{\tiny$s_{\beta_i}$}}
\put(30,0){\circle*{1,3}}\put(30,0){\makebox(-2,6)[c]{\tiny$\sigma_j$}}
\put(70,0){\circle*{1,3}}\put(70,0){\makebox(-2,6)[c]{\tiny$j$}}
\end{picture}}&&
{\setlength{\unitlength}{0.6mm}\begin{picture}(80,15)(0,-7)
\put(0,0){\line(1,0){15}}\put(25,0){\line(1,0){55}}
\put(0,0){\circle*{1,3}}\put(0,0){\makebox(-2,6)[c]{\tiny 1}}
\put(15,0){\circle*{1,3}}\put(15,0){\makebox(-2,6)[c]{\tiny $N_2$}}
\put(80,0){\circle*{1,3}}\put(80,0){\makebox(2,6)[c]{\tiny $n$}}
\qbezier(7,0)(28,-12)(50,0) \qbezier(30,0)(50,-10)(70,0)
\put(7,0){\circle*{1,3}}\put(7,0){\makebox(-2,6)[c]{\tiny $f_i$}}
\put(50,0){\circle*{1,3}}\put(50,0){\makebox(-2,6)[c]{\tiny$s_{\beta_i}$}}
\put(30,0){\circle*{1,3}}\put(30,0){\makebox(-2,6)[c]{\tiny$\sigma_j$}}
\put(70,0){\circle*{1,3}}\put(70,0){\makebox(-2,6)[c]{\tiny$j$}}
\end{picture}}\hspace{0.5cm}\\
\hline
\end{array}
$$
\caption{Effects of the mapping $\Gamma^{(n_1,n_2)}$ on the
crossings of $\sigma$ and $\sigma'$.}\label{tab:effet Gamma sur cr}
\end{table}

\begin{lem}\label{lem:lemme pour lemme 2}
Let $n_1,n_2$ and $n$ be positive integers with $N_2\leq n$ and
$\gamma\in{\S_{n}^{\,(n_1,n_2)}}$. Suppose that
\begin{align}
B(\gamma):=&\{(i,\gamma(i))\;|\;i<\gamma(i)\leq N_2\}=\{(i_1,j_1),(i_2,j_2),\ldots,(i_p,j_p)\},\label{eq:B1}\\
B(\gamma^{-1})=&\{(\gamma(i),i)\;|\;\gamma(i)<i\leq
N_2\}=\{(k_1,\ell_1),(k_2,\ell_2),\ldots,(k_q,\ell_q)\},\label{eq:B2}
\end{align}
with $ i_1<i_2<\cdots<i_p$ and $k_1<k_2<\cdots<k_q$. Then we have
\begin{equation}\label{eq:lemG5}
|G^{(n_1,n_2)}_5(\gamma)|=\sum_{r=1}^p(j_r-i_r)+\sum_{r=1}^q(\ell_r-k_r-1)-{p+q\choose2}.
\end{equation}
\end{lem}

 Indeed, suppose
\begin{align*}
B(\sigma)&=\{(i_1,N_2+1-j_1),\ldots,(i_p,N_2+1-j_p)\},\\
B(\sigma^{-1})&=\{(k_1,N_2+1-\ell_{1}),\ldots,(k_q,N_2+1-\ell_{q})\},
\end{align*}
then, by construction of $\sigma'$, we have
\begin{align*}
B(\sigma')&=\{(j_1,N_2+1-i_1),\ldots,(j_p,N_2+1-i_p)\},\\
B(\sigma'^{\,-1})&=\{(\ell_{1},N_2+1-k_1),\ldots,(\ell_{q},N_2+1-k_q)\}.
\end{align*}
By symmetry, the identity $\eqref{eq:lemG5}$ is also valid on
${\S_{n}^{(n_2,n_1)}}$. Applying $\eqref{eq:lemG5}$ to $\sigma'$ and
${\sigma}$ leads to
$|G^{(n_1,n_2)}_5(\sigma)|=|G^{(n_2,n_1)}_5(\sigma')|$. The proof of Lemma~\ref{prop:Gamma} is thus completed.

 \noindent{\bf Proof of
Lemma~\ref{lem:lemme pour lemme 2}.} For any $\gamma\in {\S_{n}^{\,(n_2,n_1)}}$, by definition, we have
\begin{align}
|G^{(n_1,n_2)}_5(\gamma)|&=|\{(i,j)\;|\;i<j<\gamma(i)\leq
N_2<\gamma(j)\}|+ |\{(i,j)\;|\;\gamma(j)<\gamma(i)<j\leq N_2<i\}|\nonumber\\
&\quad + |\{i\;|\;i<\gamma(i)\leq N_2<\gamma^2(i)\}|.\label{eq:G5}
\end{align}

Now, by the definition of $B(\gamma)$
we get
 \begin{align*}
&|\{(i,j)\;|\;i<j<\gamma(i)\leq N_2<\gamma(j)\}|\nonumber\\
&=
\sum_{r=1}^p|\{x\;|\;i_r<x<j_{r}\leq
N_2<\gamma(x)\}|\\
&=\sum_{r=1}^p(|\{x\;|\;i_r<x< j_{r}\}|-|\{x\;|\;i_r<x<
j_{r}\,,\,\gamma(x)\leq N_2\}|).
\end{align*}
For any $r\in [1,p]$,  we have $|\{x\;|\;i_r<x<
j_{r}\}|=j_r-i_r-1$ and
\begin{align*}
&|\{x\;|\;i_r<x< j_{r}\,,\,\gamma(x)\leq N_2\}|\\
=&|\{x\;|\;i_r<x< j_{r}\,,\,x<\gamma(x)\leq N_2\}|+|\{x\;|\;i_r<x<
j_{r}\,,\,\gamma(x)<x\leq
N_2\}|\\
=&|\{t\;|\;i_r<i_t< j_{r}\}| +|\{t\;|\;i_r<\ell_t<
j_{r}\}|\quad\text{(by definition of $B(\gamma)$ and $B(\gamma^{-1})$)}\\
=&|\{t\;|\;i_r<i_t\}|+|\{t\;|\;\ell_t< j_{r}\}|,
\end{align*}
because,  by definition of $\S_{n}^{(n_1,n_2)}$, \eqref{eq:B1} and \eqref{eq:B2},  for any
integers $r$ and $t$, we have $i_t\leq n_1$, $k_t\leq n_1$, $j_r>n_1$ and
$\ell_t>n_1$,  therefore  $i_t<j_r$ and $i_r<\ell_t$.

Summing over all $r$ yields
\begin{equation}\label{eq:G5a}
|\{(i,j)\;|\;i<j<\gamma(i)\leq
N_2<\gamma(j)\}|=\sum_{r=1}^p(j_r-i_r-1-|\{t\;|\;i_r<i_t\}|-|\{t\;|\;\ell_t<
j_{r}\}|).
\end{equation}
It follows  that
 \begin{align}
&|\{(i,j)\;|\;\gamma(j)<\gamma(i)<j\leq N_2<i\}|\nonumber\\
&=|\{(i,j)\;|\;i<j<\gamma^{-1}(i)\leq N_2<\gamma^{-1}(j)\}|\nonumber\\
&=\sum_{r=1}^q(\ell_r-k_r-1-|\{t\;|\;k_r<k_t\}|-|\{t\;|\;j_t<
\ell_{r}\}|).\label{eq:G5b}
\end{align}
As  $|\{i\;|\;i<\gamma(i)\leq
N_2<\gamma^2(i)\}|=|\{t\;|\;\gamma(j_t)>N_2\}|$, plugging 
\eqref{eq:G5a} and \eqref{eq:G5b} into  \eqref{eq:G5} leads to
\begin{align}
|G^{(n_1,n_2)}_5(\gamma)|&=\sum_{r=1}^p(j_r-i_r-1)
+\sum_{r=1}^q(\ell_r-k_r-1)+|\{t\;|\;\gamma(j_t)>N_2\}|-\sum_{r=1}^p|\{t\;|\;i_r<i_t\}|\nonumber\\
&-\sum_{r=1}^p|\{t\;|\;\ell_t<
j_{r}\}|-\sum_{r=1}^q|\{t\;|\;k_r<k_t\}|-\sum_{r=1}^q|\{t\;|\;j_t<
\ell_{r}\}|.\label{eq:G5-simplifie}
\end{align}

Since the $i_r$'s and $k_r$'s are distinct we have
\begin{align}
\sum_{r=1}^p|\{t\;|\; i_r<i_t\}|={p\choose 2}\quad\text{and}
\quad\sum_{r=1}^q|\{t\;|\;k_r<k_t\}|={q\choose 2}.\label{eq:detail1}
\end{align}

On the other hand,
\begin{align}
\sum_{r=1}^p|\{t\;|\;\ell_t< j_{r}\}|+\sum_{r=1}^q|\{t\;|\;j_t<
\ell_{r}\}|
&=\sum_{r=1}^p|\{t\;|\;\ell_t\neq j_r\}|\nonumber\\
&=pq-|\{t\;|\;j_t\in\{\ell_1,\ell_2,\ldots,\ell_q\}|\nonumber\\
&=pq-\sum_{s=1}^k|\{t\;|\;\gamma(j_t)\leq N_2\}|,\label{eq:detail2}
\end{align}
where the last identity follows from the definitions of $B(\gamma)$
and $B(\gamma^{-1})$. Inserting \eqref{eq:detail1} and
\eqref{eq:detail2} in \eqref{eq:G5-simplifie}  we get
\eqref{eq:lemG5}. This concludes the proof of Lemma~\ref{lem:lemme
pour lemme 2}. \qed

\subsection*{Acknowledgement} This work was partially supported by the French National Research Agency under the grant ANR-08-BLAN-0243-03.

%%%%%%%%%%%%%%%%%%%%%%%%%%

\end{document}